\documentclass[reqno,12pt,a4wide]{article}

\usepackage{eucal,enumitem,mathrsfs}
\usepackage{amsmath,amssymb,epsfig,bbm,amsthm}
\usepackage{xcolor}

\usepackage[pagewise]{lineno}\nolinenumbers

\numberwithin{equation}{section}

\newtheorem{theorem}{Theorem}[section]

\newtheorem{proposition}[theorem]{Proposition}
\newtheorem{definition}[theorem]{Definition}

\theoremstyle{definition}
\newtheorem{example}[theorem]{Example}
\newtheorem{remark}[theorem]{Remark}
\newtheorem{assumption}[theorem]{Assumption}
\newcommand{\K}{\mathcal{K}}

\usepackage{hyperref}

\title{A Kernel-Density-Estimator Minimizing Movement Scheme for Diffusion Equations}
\begin{document}

\author{
Florentine Catharina Flei$\ss$ner 
\thanks{Technische Universit\"at M\"unchen  email:
  \textsf{fleissne@ma.tum.de}.
  } 
}
 \date{}

\maketitle

\begin{abstract}
The mathematical theory of a novel variational approximation scheme for general second and fourth order partial differential equations 
\begin{equation}\label{eq: diffusion equation abstract}
\partial_t u  - \nabla\cdot\Big(u\nabla\frac{\delta\phi}{\delta u}(u)\Big|\nabla\frac{\delta\phi}{\delta u}(u)\Big|^{q-2}\Big) \ = \ 0, \quad\quad u\geq0, 
\end{equation}
$q\in(1, +\infty)$, is developed; the Kernel-Density-Estimator Minimizing Movement Scheme (KDE-MM-Scheme) preserves the structure of \eqref{eq: diffusion equation abstract} as a gradient flow with regard to an energy functional $\phi$ and a Wasserstein distance in the space of probability measures, at the same time imitating the steepest descent motion of a finite number of particles / data points on a discrete timescale. Roughly speaking, the KDE-MM-Scheme constitutes a simplification of the classical Minimizing Movement scheme for \eqref{eq: diffusion equation abstract} (often referred to as `JKO scheme'), in which the corresponding minimum problems are relaxed and restricted to the values of Kernel Density Estimators each associated with a finite dataset. Rigorous mathematical proofs show that the KDE-MM-Scheme yields (weak) solutions to \eqref{eq: diffusion equation abstract} if we let the time step sizes and the dataset sizes (particle numbers) simultaneously go to zero and infinity respectively. Uniting abstract analysis in metric spaces with application-orientated concepts from statistics and machine learning, our examinations will form the mathematical foundation for a novel computationally tractable algorithm approximating solutions to \eqref{eq: diffusion equation abstract}. 

A particular ingredient for our theory is a thorough and general analysis of the JKO scheme under the occurrence of $\Gamma$-perturbations $\phi_n$ of the energy functional $\phi$. The discrete-time steepest descents w.r.t. $\phi_n, \ n\in\mathbb{N},$ are directly linked with \eqref{eq: diffusion equation abstract} through an appropriate correlation between time step sizes and parameters that only depends on the velocity of $\Gamma$-convergence $\phi_n\stackrel{\Gamma}{\to}\phi$. 
\end{abstract}

\section{Introduction}\label{sec: intro}

The partial differential equation 
\begin{equation}\label{eq: diffusion equation intro}
\partial_t u  - \nabla\cdot\Big(u\nabla\frac{\delta\phi}{\delta u}(u)\Big|\nabla\frac{\delta\phi}{\delta u}(u)\Big|^{q-2}\Big) \ = \ 0, \quad\quad u\geq0, 
\end{equation}
$q\in(1, +\infty)$, 
represents a common model for describing the evolution in time of the density $u$ of some quantity in a physical, chemical, biological, ecological or economic process, where total mass is conserved and diffusion is governed by the variational derivative $\frac{\delta\phi}{\delta u}$ of an energy functional $\phi$. Presumably the best-known second order examples of \eqref{eq: diffusion equation intro} (for $q=2$) are the heat equation
\begin{equation*}
\partial_t u - \Delta u \ =  \ 0 \quad\quad \Big(\phi(u):=\int{u\log u \mathrm{d}x}\Big), 
\end{equation*}
the porous medium equation (for $m>1$) / fast diffusion equation (for $m<1$) 
\begin{equation*}
\partial_t u - \Delta (u^m) \ = \ 0 \quad\quad \Big(\phi(u):= \frac{1}{m-1} \int{u^m\mathrm{d}x}\Big)
\end{equation*}
and a general second order diffusion equation with external potential and interaction term
\begin{eqnarray*}
\partial_t u - \nabla\cdot\Big(u\Big(\nabla F'(u) + \nabla V + (\nabla W)\ast u \Big)\Big) \ = \ 0 \quad\quad \\
\Big(\phi(u):=\int{\int{\Big[F(u(x)) + V(x) u(x) + \frac{1}{2}W(x-y)u(x) \Big] u(y)\mathrm{d}x}\mathrm{d}y}\Big), 
\end{eqnarray*}
whose areas of application are ranging from the theory of heat conduction, heat radiation in plasmas, models for studying the motion of a gas in a porous medium, the motion of a population or the evolution in time of a region occupied by water where groundwater infiltration through a porous stratum occurs  (cf. \cite{widder1976heat, vazquez2007porous, zeldovich, berryman1978nonlinear, darcy, leibenzon1930motion, muskat1937flow, gurtin1977diffusion, boussinesq}) to the description of time-dependent probability distributions of velocities in the theory of granules affected by their environment, friction and inelastic collisions between granules with different velocities (cf. \cite{benedetto1997kinetic, benedetto1998non, toscani2000one, carrillo2003kinetic, carrillo2006contractions}), to name but a few. Classic fourth order examples of \eqref{eq: diffusion equation intro} are the thin film equation
\begin{equation*}
\partial_t u + \nabla\cdot \big(u\nabla\Delta u\big) \ = \ 0  \quad\quad \Big(\phi(u):=\frac{1}{2}\int{|\nabla u|^2\mathrm{d}x}\Big), 
\end{equation*}
applied to lubrication theory for describing the motion of a moving contact line (cf. e.g. \cite{bertozzi1998mathematics}), and the quantum drift diffusion equation
\begin{equation*}
\partial_t u + 4\nabla\cdot\big(u\nabla\frac{\Delta\sqrt{u}}{\sqrt{u}}\big) \ = \ 0 \quad\quad \Big(\phi(u):=\int{\frac{|\nabla u|^2}{u}\mathrm{d}x}\Big), 
\end{equation*}
in which $u$ stands e.g. for the electron density in a quantum model for semiconductors (cf. \cite{pinnau2000linearized, jungel2001positivity, jungel2009transport}). 

In the 90's, Felix Otto originated, together with Jordan and Kinderlehrer \cite{otto1998dynamics, jordan1997free, jordan1998variational, otto1996double, otto1998lubrication, otto2001geometry}, the interpretation of dynamics governed by \eqref{eq: diffusion equation intro} as a steepest descent with regard to $\phi$ and a Wasserstein distance in the space of probability measures, building on a Riemannian formalism (`Otto calculus') and an approximation by discrete-time steepest descents (`JKO scheme'). The common knowledge of the associated gradient flow structure advanced through the book \cite{AGS08} by Ambrosio, Gigli and Savar\'e.

This paper addresses the issue of a computationally tractable approximation scheme to nonnegative solutions $u$ of \eqref{eq: diffusion equation intro} built on the aforementioned gradient flow structure. Our approach marries the abstract theory of De Giorgi's Minimizing Movement scheme as fundamental variational approximation technique for gradient flows with the application-orientated concept of Kernel Density Estimation from the fields of machine learning and statistics. A careful analysis based on the author's paper ``$\Gamma$-Convergence and Relaxations for Gradient Flows in Metric Spaces: a Minimizing Movement Approach'' \cite{fleissner2016gamma} serves to bridge the gap between these two concepts, providing a rigorous mathematical justification and foundation for our \textsf{\textit{Kernel-Density-Estimator Minimizing Movement Scheme}} or \textsf{\textit{KDE-MM-Scheme}} for short. 

\subsection{The Minimizing Movement Approach}\label{sec: MM approach}
The notion of \textit{Minimizing Movements} was introduced by Ennio De Giorgi at the beginning of the 90's \cite{DeGiorgi93}. De Giorgi who drew his inspiration from the paper \cite{almgren1993curvature} by Almgren, Taylor and Wang regarded his concept as ``natural meeting point'' of many evolution problems from different research fields; indeed, the concept of Minimizing Movements has proved extremely useful, with a wide range of applications in analysis, geometry, physics and numerical analysis, see e.g. \cite{DeGiorgi93, almgren1993curvature, ambrosio1995minimizing, AGS08, fleissner2016gamma, fleissner2021minimizing, fleissner2020reverse}, \cite{braides2014local} and \cite{mielke2011differential, mielke2016balanced}. A general Minimizing Movement (MM) scheme for a time-invariant evolution system in some topological space $\mathcal{X}$ consists of a functional $\Phi: (0,1)\times\mathcal{X}\times\mathcal{X}\to[-\infty, +\infty]$ and of successively solving the minimum problems
\begin{equation}\label{eq: general MM scheme}
\mu_\tau^m \quad\text{ is a minimizer for }\quad \Phi(\tau, \mu_\tau^{m-1},\cdot) \quad\quad\quad (m\in\mathbb{N}), 
\end{equation}
for given time step size $\tau>0$ and initial datum $\mu_\tau^0\in\mathcal{X}$, with the purpose of studying the limiting behaviour of the corresponding piecewise constant interpolations
\begin{equation}\label{eq: piecewise constant interpolation}
\mu_\tau(0)=\mu_\tau^0, \quad\quad \mu_\tau(t) \equiv \mu_\tau^m \quad \text{ if } t\in((m-1)\tau, m\tau], \ m\in\mathbb{N},
\end{equation}
as the time step size $\tau\downarrow0$. 

The seminal articles \cite{almgren1993curvature} and \cite{DeGiorgi93} established the Minimizing Movement scheme \eqref{eq: general MM scheme} associated with
\begin{equation}\label{eq: discrete time gf}
\Phi(\tau, \mu, \nu):= \phi(\nu) + \frac{1}{2\tau}{\sf d}(\nu, \mu)^2
\end{equation}
as discrete-time steepest descent motion with regard to an energy functional $\phi: \mathcal{X}\to(-\infty, +\infty]$ and a distance ${\sf d}$ in a general metric space $(\mathcal{X}, {\sf d})$ (provided that solutions to the minimum problems \eqref{eq: general MM scheme} exist for small $\tau$). Besides giving concrete examples thereof, De Giorgi \cite{DeGiorgi93} anticipated the close connection between Minimizing Movements and the notion of continuous-time gradient flows in general metric spaces, which was proved later by Ambrosio, Gigli and Savar\'e \cite{AGS08}. 

Let $\mathcal{X}$ be the space $\mathcal{P}_2(\mathbb{R}^d)$ of probability measures with finite moments of second order (i.e. $\int_{\mathbb{R}^d}{|x|^2\mathrm{d} \mu} < +\infty$) endowed with the $2$-Wasserstein distance $\mathcal{W}_2$ and  $\phi: \mathcal{P}_2(\mathbb{R}^d)\to(-\infty,+\infty]$ be an energy functional.
A natural chain rule enables the linkage between three approaches to the notion of gradient flows in $(\mathcal{P}_2(\mathbb{R}^d), \mathcal{W}_2)$, namely discrete-time steepest descents \eqref{eq: general MM scheme}, \eqref{eq: discrete time gf} w.r.t. $\phi$ and $\mathcal{W}_2$, the gradient-flow-type differential equation 
\begin{equation}\label{eq: differential inclusion intro}
v_t \ = \ -D_l\phi(\mu(t))
\end{equation}
and the energy dissipation equality 
\begin{equation}\label{eq: ede intro}
\phi(\mu(0)) - \phi(\mu(t)) \ = \ \frac{1}{2}\int_0^t{|\partial^-\phi|^2(\mu(r))\mathrm{d}r} \ + \ \frac{1}{2}\int_0^t{|\mu'|^2(r)\mathrm{d}r}
\end{equation}
arising from the abstract theory of gradient flows in metric spaces, see Proposition \ref{prop: 2.1}. All relevant definitions (Wasserstein distance, slope $|\partial^-\phi|$ and subdifferential $\{D_l\phi\}$ of $\phi$, metric derivative $|\mu'|$ and tangent vector field $v$ of the locally absolutely continuous curve $\mu$ in $(\mathcal{P}_2(\mathbb{R}^d), \mathcal{W}_2)$) are given in the first part of Section \ref{sec: 2.1}. As outlined therein, $D_l\phi$ can be identified with $\nabla\frac{\delta\phi}{\delta u}$ so that \eqref{eq: differential inclusion intro} is a weak reformulation of \eqref{eq: diffusion equation intro}, $q=2$, in the space of probability measures. We carefully examine, as an example, the underlying structure \eqref{eq: differential inclusion intro}, \eqref{eq: ede intro} of second order diffusion equations 
\begin{equation}\label{eq: second order diffusion equation}
\partial_t u - \nabla\cdot\Big(u\Big(\nabla F'(u) + \nabla V + (\nabla W)\ast u\Big) \Big) \ = \ 0\quad \text{ in } (0, +\infty)\times\Omega
\end{equation}
with no-flux boundary condition
\begin{equation}\label{eq: second order diffusion equation 2}
u\Big(\nabla F'(u) + \nabla V + (\nabla W) \ast u\Big)\cdot {\sf n} \ = \ 0 \quad \text{ on } (0, +\infty)\times\partial\Omega, 
\end{equation}
associated with $\phi: \mathcal{P}_2(\mathbb{R}^d)\to(-\infty, +\infty]$, 
\begin{equation}\label{eq: energy functional intro}
\phi(\mu):= \int_{\mathbb{R}^d\times\mathbb{R}^d}{[F(u(x)) + V(x)u(x)+\frac{1}{2}W(x-y)u(x)]u(y)\mathrm{d}x\mathrm{d}y} 
\end{equation}
if $\mu=u\mathcal{L}^d\ll\mathcal{L}^d\llcorner\Omega$ (i.e. $u\equiv0$ on $\mathbb{R}^d\setminus\Omega$) and $\phi(\mu):=+\infty$ else; our general assumptions on $F, V, W$ and $\Omega$ include the case that $\phi$ is \textit{not} displacement convex in $(\mathcal{P}_2(\mathbb{R}^d), \mathcal{W}_2)$, see Example \ref{ex: chain rule} and Proposition \ref{prop: chain rule} for the validation of the corresponding chain rule. 

We attach a particular value not only to the approximation of solutions to \eqref{eq: diffusion equation intro} by discrete-time steepest descents but also to the exact characterizations \eqref{eq: differential inclusion intro} and \eqref{eq: ede intro} of continuous-time gradient flows in $(\mathcal{P}_2(\mathbb{R}^d), \mathcal{W}_2)$ because we are concerned with the effect of $\Gamma$-perturbations $\phi_n$ of $\phi$ on the steepest descent motion; our general stability theory involves the notion of both discrete-time and continuous-time steepest descents, see Section \ref{sec: 2.2} and Section \ref{sec: 1.3} for an overview. 

If no chain rule is considered, the Minimizing Movement scheme \eqref{eq: general MM scheme} associated with \eqref{eq: discrete time gf}, $\phi$ and $\mathcal{W}_2$ (often referred to as `JKO scheme') produces solutions to a weak form of \eqref{eq: differential inclusion intro} 
satisfying an energy inequality instead of \eqref{eq: ede intro}, cf. Thms. 11.1.6 and 2.3.3 in \cite{AGS08} and Remark \ref{rem: absence of chain rule}. 
It first came to common knowledge through the examinations by Jordan, Kinderlehrer and Otto \cite{otto1998dynamics, jordan1997free, jordan1998variational, otto1996double, otto1998lubrication} and is nowadays a widely used variational approximation technique for tackling partial differential equations of the form \eqref{eq: diffusion equation intro} with $q=2$ (cf. Chaps. 10.1, 10.4 and 11 in \cite{AGS08}, Chap. 8 in \cite{santambrogio2015optimal}, Chap. 4 in \cite{santambrogio2017euclidean} and the references in all three of them).  It is a well-known fact that the JKO scheme, associated with the respective energy functionals $\phi$, yields e.g. (weak) solutions to the heat equation \cite{jordan1998variational}, the porous medium / fast diffusion equation \cite{otto1998dynamics, otto1996double}, the thin film equation \cite{matthes2009family}, the quantum drift diffusion equation \cite{gianazza2009wasserstein} and
to general second order diffusion equations with external potential and interaction term \cite{AGS08}. 

Finally, we point out that in Section \ref{sec: 2.1}, the Minimizing Movement approach to \eqref{eq: diffusion equation intro} and the corresponding gradient flow structure are treated for $q\in(1,+\infty)$.

\subsection{Kernel Density Estimation}\label{sec: kde}

A function estimator is defined as a random variable with values in some function space $\mathfrak{F}$, emerging from a mapping 
of idependent and identically distributed (i.i.d.) data points $X_1, ... , X_n$ drawn from the same but generally unknown probability distribution (cf. e.g. Sect. 5.4 in \cite{goodfellow2016deep}). 

Supposing that $X_1, ... , X_n$ represent an i.i.d. sample from a Borel probability distribution on $\mathbb{R}^d$ having a Lebesgue density function $\rho$ and setting $\mathfrak{F}:=\{u\in\mathrm{L}^1(\mathbb{R}^d) \ | \ u\geq 0, \int_{\mathbb{R}^d}{u(x)\mathrm{d}x} = 1\}$  ($=$ set of probability density functions), we can try to capture $\rho$ by finding suitable estimators
\begin{equation}\label{eq: function estimator}
\hat{\rho}_n := f_n(X_1, ..., X_n), \quad\quad f_n: \Big(\mathbb{R}^d\Big)^n\to\mathfrak{F}, \quad\quad n\in\mathbb{N}. 
\end{equation}
Therein lies the purpose of \textit{probability density estimation} going back to \cite{rosenblatt1956remarks, whittle1958smoothing, parzen1962estimation}. As a first natural approach to it, it is proposed in \cite{rosenblatt1956remarks} to consider, for the case $d=1$, the random difference quotient
\begin{equation*}
\frac{F_n(x+h)-F_n(x-h)}{2h}
\end{equation*}
with $F_n$ being the sample distribution function (i.e. $F_n(y)$ is equal to the number of observations $\leq y$ among $X_1,...,X_n$ divided by $n$), and to pass to the limit in the sample size $n\uparrow+\infty$ and bandwidth $h\downarrow0$ simultaneously. The fact that this approach corresponds to the choice
\begin{equation*}
f_n(y_1,...,y_n)(x):=\frac{1}{n\cdot h(n)}\sum_{i=1}^{n}{\K\Big(\frac{x-y_i}{h(n)}\Big)}  \ \text{ with } \ \K(y):=\begin{cases} \frac{1}{2} &\text{ if } y\in[-1,1] \\ 0 &\text{ else }\end{cases}   
\end{equation*}
and $h(n)>0$ ($\lim_{n\to+\infty}h(n)=0$) in \eqref{eq: function estimator} motivated Murray Rosenblatt \cite{rosenblatt1956remarks} and Emanuel Parzen \cite{parzen1962estimation} to allow for general functions $\K$ in the above definition of their probability density estimators wherein the origin of \textit{Kernel Density Estimation} lies. 

\begin{definition}[Kernel function]\label{def: kernel function}
A kernel function is defined as a nonnegative function $\K: \mathbb{R}^d\to[0,+\infty)$ with $\int_{\mathbb{R}^d}{\K(x)  \mathrm{d}x} = 1$. 

Every kernel function $\K$ is associated with a family of functions
\begin{equation}\label{eq: K_h}
\K_h: \mathbb{R}^d\to[0,+\infty), \quad \K_h(x) := \frac{1}{h^d}\K\Big(\frac{x}{h}\Big), \quad h >0.  
\end{equation}
\end{definition}

The convolution between such scaled kernel function \eqref{eq: K_h} and the empirical measure $\frac{1}{n}\sum_{i=1}^{n}{\delta_{X_i}}$ yields a Kernel Density Estimator ($\delta_y$ denotes the Dirac measure with centre $y\in\mathbb{R}^d$). 

\begin{definition}[Kernel Density Estimator]\label{def: kernel density estimator}
Assuming that $X_1, ... , X_n$ is an i.i.d. sample from a probability distribution with Lebesgue density function $\rho$ and $\K$ is a kernel function, the function estimator
\begin{equation}\label{eq: kernel density estimator}
\hat{\rho}_{n,h}:=\frac{1}{n}\sum_{i=1}^{n}{\K_h(\cdot-X_i)}
\end{equation}
is called Kernel Density Estimator (KDE) for $\rho$ associated with the sample size $n\in\mathbb{N}$ and bandwidth $h>0$. 
\end{definition}

The asymptotic behaviour of Kernel Density Estimators as the sample size $n\uparrow+\infty$ and the bandwidth $h\downarrow0$ simultaneously is well examined, including extremely useful strong consistency results and uniform convergence rates, see Section \ref{sec: 2.3}. 

Kernel Density Estimators are used for example in the wide areas of clustering and topological data analysis with various applications in computer vision, text analysis, biology, chemistry and astronomy including image processing, anomaly detection, unsupervised and semi-supervised classification, genetic profiling and protein analysis, to name but a few (cf. \cite{rosenblatt1956remarks, parzen1962estimation, abraham2004asymptotic, li2007nonparametric, rinaldo2010generalized, arias2016estimation, wang2017optimal, jiang2017uniform, fasy2014confidence, kim2019uniform} and the references therein). 

\subsection{$\Gamma$-Convergence for Steepest Descents and the Kernel-Density-Estimator Minimizing Movement Scheme}\label{sec: 1.3}

The concepts of Minimizing Movements and Kernel Density Estimators are married by a thorough examination of the discrete-time steepest descent motion under the occurrence of $\Gamma$-perturbations of the energy functional. 

\subsubsection{$\Gamma$-Convergence for Steepest Descents}\label{sec: 1.3.1}

The Minimizing Movement scheme \eqref{eq: general MM scheme}, \eqref{eq: discrete time gf} 
can be performed along a sequence of $\Gamma$-converging functionals $\phi_n\stackrel{\Gamma}{\rightarrow}\phi$ in $(\mathcal{P}_2(\mathbb{R}^d), \mathcal{W}_2)$ by assigning a parameter $n=n(\tau)\in\mathbb{N}$ to every time step size $\tau>0$, replacing \eqref{eq: discrete time gf} with 
\begin{equation}\label{eq: MM scheme along sequence}
\Phi(\tau, \mu, \nu) \ := \ \phi_{n(\tau)}(\nu) \ + \ \frac{1}{2\tau}\mathcal{W}_2(\nu, \mu)^2
\end{equation}
and simultaneously passing to the limit in the time step sizes $\tau\downarrow0$ and parameters $n(\tau)\uparrow+\infty$.
There is a direct connection between the Minimizing Movement scheme associated with \eqref{eq: MM scheme along sequence} and the weak reformulation \eqref{eq: differential inclusion intro} of the diffusion equation 
\begin{equation}\label{eq: diffusion equation intro 2}
\partial_t u  - \nabla\cdot\Big(u\nabla\frac{\delta\phi}{\delta u}(u)\Big) \ = \ 0, \quad\quad u\geq0
\end{equation}
(i.e. \eqref{eq: diffusion equation intro} with $q=2$):  
\begin{itemize}
\item Let energy functionals $\phi, \ \phi_n: \mathcal{P}_2(\mathbb{R}^d)\to(-\infty, +\infty] \ (n\in\mathbb{N})$ be given with $\phi_n\stackrel{\Gamma}{\rightarrow}\phi$ in $(\mathcal{P}_2(\mathbb{R}^d), \mathcal{W}_2)$. 
Assuming a natural coercivity condition and a natural chain rule, there exists a sequence $(n_\tau)_{\tau>0}, \ n_\tau\in\mathbb{N},$ 
such that the \textit{relaxed Minimizing Movement scheme}
\begin{equation}\label{eq: relaxed MM scheme}
\Phi(\tau, \mu_\tau^{m-1}, \mu_\tau^m) \ \leq \ \inf_{\nu\in\mathcal{P}_2(\mathbb{R}^d)}{\Phi(\tau, \mu_\tau^{m-1}, \nu)} \ + \ \gamma_\tau \quad\quad  (m\in\mathbb{N}, \ \gamma_\tau>0)
\end{equation}
associated with \eqref{eq: MM scheme along sequence} yields solutions to \eqref{eq: differential inclusion intro} and \eqref{eq: ede intro} (continuous-time gradient flow of $\phi$) whenever the parameters $n(\tau)\geq n_\tau$, the order of the error term $\gamma_\tau$ is $o(\tau)$ and the sequence of initial data $\mu_\tau^0\to\mu^0$ is a  so-called recovery sequence, i.e. $\phi_{n(\tau)}(\mu_\tau^0)\to\phi(\mu^0)$. 

(see Theorems \ref{thm: MM approach I} and \ref{thm: MM approach II}) 
\item Such appropriate correlations between time step sizes $\tau$ and parameters $n=n(\tau)$ can be precisely determined through a general condition relating the Minimizing Movement motion generated by \eqref{eq: MM scheme along sequence} with the slope $|\partial^-\phi|$ of $\phi$.
(see Assumption \ref{ass: 2})
\end{itemize}
The crucial point is the sequence $(n_\tau)_{\tau>0}$ is completely independent of initial data $\mu_\tau^0, \mu^0$ and of (approximate) minimizers $\mu_\tau^m$ in \eqref{eq: relaxed MM scheme}; it solely depends on the velocity of $\Gamma$-convergence $\phi_n\stackrel{\Gamma}{\rightarrow}\phi$. We thereby obtain a general and robust stability statement for steepest descents under the occurrence of $\Gamma$-perturbations $\phi_n$ of the energy functional $\phi$. The fact that we are able to relax the minimum problems in every step \eqref{eq: relaxed MM scheme} of the scheme is a further special feature of our theory. 

We note that the key to stability is to examine the discrete-time steepest descents w.r.t. $\phi_n, \ n\in\mathbb{N}$, focusing on the interplay between parameters and time step sizes. We do not restrict ourselves to the study of the limiting behaviour of the continuous-time steepest descents which would amount to letting $\tau\downarrow0$ for fixed $n\in\mathbb{N}$ first and only then $n\uparrow+\infty$; the limit of a sequence of continuous-time steepest descents w.r.t. $\phi_n, \ n\in\mathbb{N},$ is in general \textit{no} solution to the continuous-time gradient flow of $\phi$ because the slopes of the $\Gamma$-perturbations are not related to the slope of $\phi$, see Sect. 1 in \cite{fleissner2016gamma}. 

The proof of Theorems \ref{thm: MM approach I} and \ref{thm: MM approach II} is based on the careful study of the structure \eqref{eq: differential inclusion intro} and \eqref{eq: ede intro} of continuous-time gradient flows in $(\mathcal{P}_2(\mathbb{R}^d), \mathcal{W}_2)$ from Section \ref{sec: 2.1} (whence the ``chain rule'' naturally arises) and an application of Thms. 3.4 and 6.1 from \cite{fleissner2016gamma}.
Section \ref{sec: 2.2} provides detailed and additional information on our stability theory including a non-uniform distribution of the error term (i.e. $\gamma_\tau^m$ depending on $m\in\mathbb{N}$ instead of $\gamma_\tau$, see Remark \ref{rem: error term}) and the case that $q\neq2$ in \eqref{eq: diffusion equation intro}. We refer the reader to Remark \ref{rem: main ass} for special cases: e.g. if $\phi_n\stackrel{\Gamma}{\to}\phi$ and the corresponding slopes satisfy a $\Gamma$-liminf inequality (`Serfaty-Sandier condition' from \cite{sandier2004gamma, serfaty2011gamma}), then every choice $n=n(\tau)\uparrow+\infty$ (as $\tau\downarrow0$) is appropriate.

Thanks to our theory, new possibilities are opening up in the study of the variational limit of steepest descent movements (cf. \cite{fleissner2016gamma}). 
This paper focuses on a novel approach to transforming the JKO scheme into a computationally tractable Minimizing Movement scheme for \eqref{eq: diffusion equation intro}. 

\subsubsection{The KDE-MM-Scheme}\label{sec: KDE-MM-Scheme intro}

The idea behind the \textsf{KDE-MM-Scheme} is to carry out the relaxed Minimizing Movement scheme \eqref{eq: relaxed MM scheme} associated with \eqref{eq: MM scheme along sequence} along a \textsf{$\Gamma$-KDE-Approximation} of the energy functional $\phi$; a $\Gamma$-KDE-Approximation is a sequence of $\Gamma$-perturbations $\phi_n$ of $\phi$ whose effective domains $\{\phi_n < +\infty\}$ are concentrated in the respective KDE ranges corresponding to a given kernel function $\K$ (with $\K(\cdot)\mathcal{L}^d\in\mathcal{P}_2(\mathbb{R}^d)$), the sample size $n\in\mathbb{N}$ and bandwidth $h=h(n)>0$ (with $h(n)\downarrow0$ as $n\uparrow+\infty$) and for which Kernel Density Estimators almost surely constitute a recovery sequence corresponding to each $\mu^0\in\{\phi<+\infty\}$, i.e.
\begin{equation}\label{eq: effective domain phi_n}
\phi_n(\mu) < +\infty \quad\Rightarrow\quad \exists y_1, .., y_n: \ \mu =  \Big(\frac{1}{n}\sum_{i=1}^{n}{\K_{h(n)}(\cdot-y_i)}\Big)\mathcal{L}^d
\end{equation}
and 
\begin{equation}
\lim_{n\to+\infty}\mathcal{W}_2(\hat{\mu}_{n, h(n)}, \mu^0) \ = \ 0 \quad\text{and}\quad \lim_{n\to+\infty}\phi_n(\hat{\mu}_{n, h(n)}) \ = \ \phi(\mu^0)
\end{equation}
with probability $1$ whenever $\mu^0=\rho\mathcal{L}^d\in\{\phi< +\infty\}$, $\hat{\rho}_{n, h(n)}$, $n\in\mathbb{N}$, is a sequence of Kernel Density Estimators \eqref{eq: kernel density estimator} for $\rho$ and $\hat{\mu}_{n, h(n)}:=\hat{\rho}_{n, h(n)}\mathcal{L}^d$. 
(see Definition \ref{def: Gamma KDE Approximation})

It is reasonable to assume that the exact form of the initial probability density is generally unknown in practical applications of our theory so that we need to capture it by an i.i.d. sample from the corresponding probability distribution. This is our first motivation for the definition of a $\Gamma$-KDE-Approximation. Secondly, under \eqref{eq: effective domain phi_n}, the process of minimization \eqref{eq: relaxed MM scheme}, \eqref{eq: MM scheme along sequence} is restricted to a clearly structured subset of probability densities emerging from $\K$ through basic function operations, which should prove advantageous with regard to a practical implementation of the scheme. The KDE-MM-Scheme both preserves the original steepest descent character of the JKO scheme in the space of probability measures and mimics the motion of particles / data points (the picture of the particle motion becomes particularly apparent assuming that $\K$ has compact support with maximum value at $0$). 

A special feature of the KDE-MM-Scheme, which is due to \eqref{eq: effective domain phi_n} and the relaxation of the minimum problems \eqref{eq: relaxed MM scheme}, is that  the Wasserstein distance term in \eqref{eq: MM scheme along sequence} can be replaced by a simple `particle distance' corresponding to the optimal transport between discrete measures $\frac{1}{n(\tau)}\sum_{i=1}^{n(\tau)}{\delta_{y_i}}$. 

We refer the reader to Definition \ref{def: KDE-MM-Scheme} and the beginning of Theorem \ref{thm: KDE-MM-Scheme consistency} for detailed instructions on how to carry out the KDE-MM-Scheme for the diffusion equation \eqref{eq: diffusion equation intro} with $q\in(1, +\infty)$; 
if $q=2$, the \textsf{KDE-MM-Scheme} associated with a $\Gamma$-KDE-Approximation $(\phi_n)_n$ of $\phi$ consists in successively solving the relaxed minimum problems 
\begin{equation}\label{eq: KDE-MM-Scheme intro}
\Psi(\tau, Y_\tau^{m-1}, Y_\tau^m) \ \leq \ \inf_{Z=(z_1, ..., z_{n(\tau)})}{\Psi(\tau, Y_\tau^{m-1}, Z)} \ + \ \gamma_\tau \quad\quad  (m\in\mathbb{N}, \ \gamma_\tau>0)
\end{equation}
with
\begin{equation}\label{eq: KDE-MM-Scheme intro 2}
\Psi(\tau, Y , Z) \ := \ \phi_{n(\tau)}\Big(\frac{1}{n(\tau)}\sum_{i=1}^{n(\tau)}{\K_{h(n(\tau))}(\cdot-z_i)\mathcal{L}^d}\Big) \ + \ \frac{1}{2\tau}\sum_{i=1}^{n(\tau)}{\frac{|z_i-y_i|^2}{n(\tau)}}
\end{equation}
for $Y:=\big(y_1,...,y_{n(\tau)}\big), \ Z:=\big(z_1, ..., z_{n(\tau)}\big),  \ y_i, \ z_i\in\mathbb{R}^d$. 
The initial data for the KDE-MM-Scheme are given by 
\begin{equation}\label{eq: KDE-MM-Scheme intro 3}
Y_\tau^0:=(X_1, ..., X_{n(\tau)}), 
\end{equation}
where $X_1, ..., X_{n(\tau)}$ is an i.i.d. sample from some initial probability measure $\mu^0\in\{\phi<+\infty\}$. We let $\tau\downarrow0$ and $n(\tau)\uparrow+\infty$ simultaneously, studying the limiting behaviour of the piecewise constant interpolations \eqref{eq: piecewise constant interpolation} of the probability measures
\begin{equation*}
\mu_{\tau}^m \ := \ \frac{1}{n(\tau)}\sum_{i=1}^{n(\tau)}{\K_{h(n(\tau))}(\cdot-y_{i,\tau}^m)}\mathcal{L}^d, \quad\quad m\in\mathbb{N}_0,
\end{equation*}
associated with initial data \eqref{eq: KDE-MM-Scheme intro 3} and solutions $Y_\tau^m:=\big(y_{1,\tau}^m, ..., y_{n(\tau), \tau}^m\big)$ to the scheme \eqref{eq: KDE-MM-Scheme intro}, \eqref{eq: KDE-MM-Scheme intro 2}. 
\begin{itemize}
\item Let $\phi_n: \mathcal{P}_2(\mathbb{R}^d)\to(-\infty, +\infty], \ n\in\mathbb{N},$ be a $\Gamma$-KDE-Approximation of $\phi$ in $(\mathcal{P}_2(\mathbb{R}^d), \mathcal{W}_2)$. 
Assuming a natural coercivity condition and a natural chain rule, there exists a sequence $(n_\tau)_{\tau>0}, \ n_\tau\in\mathbb{N},$ 
such that the KDE-MM-Scheme \eqref{eq: KDE-MM-Scheme intro}, \eqref{eq: KDE-MM-Scheme intro 2} yields solutions to \eqref{eq: differential inclusion intro} and \eqref{eq: ede intro} (continuous-time gradient flow of $\phi$) with probability $1$ whenever the sample sizes $n(\tau)\geq n_\tau$, the sequence of initial data $Y_\tau^0$ is given by \eqref{eq: KDE-MM-Scheme intro 3} and an i.i.d. sample from some initial probability measure $\mu^0\in\{\phi<+\infty\}$ and the order of the error term $\gamma_\tau$ is $o(\tau)$. 

The sequence $(n_\tau)_\tau$ only depends on $(\phi_n)_n$ and $\phi$; it is completely independent of initial data $\mu^0, \ Y_\tau^0$ and of solutions $Y_\tau^m$ to the scheme. 

There is a general condition quantifying such appropriate correlations between time step sizes $\tau$ and sample sizes $n=n(\tau)$. 

(see Theorem \ref{thm: KDE-MM-Scheme consistency}) 
\end{itemize}

The reader is reminded that the differential equation \eqref{eq: differential inclusion intro} is a weak reformulation of \eqref{eq: diffusion equation intro}, $q=2$, in the space of probability measures. The strong consistency statement / convergence statement from Theorem \ref{thm: KDE-MM-Scheme consistency} includes the case that $q\neq 2$ in \eqref{eq: diffusion equation intro} and a non-uniform distribution of the error term in \eqref{eq: KDE-MM-Scheme intro}. Extensions of the theory are treated in Remarks \ref{rem: Gamma KDE Approximation variation} and \ref{rem: further relaxations}. 

To illustrate our theory of the KDE-MM-Scheme, we apply it to the second order diffusion equation \eqref{eq: second order diffusion equation} with no-flux boundary condition \eqref{eq: second order diffusion equation 2}. 
\begin{itemize}
\item We define a $\Gamma$-KDE-Approximation of the energy functional \eqref{eq: energy functional intro} so that the strong consistency statement / convergence statement from Theorem \ref{thm: KDE-MM-Scheme consistency} applies and the corresponding KDE-MM-Scheme is a sound approximation scheme for \eqref{eq: second order diffusion equation}, \eqref{eq: second order diffusion equation 2}. We exemplify the selection of appropriate correlations $\tau\mapsto n(\tau)$ between time step sizes and parameters of the $\Gamma$-KDE-Approximation, giving concrete examples of suitable choices $n=n(\tau)$. 

(see Proposition \ref{prop: Gamma KDE} and Theorem \ref{thm: selection of parameters}) 
\end{itemize}

Our assumptions on $F,V,W$ and $\Omega$ are of a quite general nature covering not only frequent examples such as the linear Fokker Planck equation with nonlocal term ($F(u)=u\log u$) or the porous medium / fast diffusion equation ($F(u)=\frac{1}{m-1}u^m$) but also general second order equations including the case that the energy functional $\phi$ is \textit{not} displacement convex in $(\mathcal{P}_2(\mathbb{R}^d), \mathcal{W}_2)$. 

A careful analysis of the relaxed minimum problems \eqref{eq: KDE-MM-Scheme intro} reveals particularities of the KDE-MM-Scheme for \eqref{eq: second order diffusion equation}, \eqref{eq: second order diffusion equation 2}. 

\begin{itemize}
\item We may restrict \eqref{eq: KDE-MM-Scheme intro} to \textit{finite} subsets $\mathcal{S}_{\omega(\tau)}$ of $\Omega$, only searching for $Y_\tau^m=(y_{1,\tau}^m, ..., y_{n(\tau),\tau}^m)$, $m\in\mathbb{N}$, such that $y_{i,\tau}^m\in\mathcal{S}_{\omega(\tau)}$ and 
\begin{equation*}
\Psi(\tau, Y_\tau^{m-1}, Y_\tau^m) \ \leq \ \Psi(\tau, Y_\tau^{m-1}, Z) \ + \ \gamma_\tau
\end{equation*}
for all $Z=(z_1, ..., z_{n(\tau)}), \ z_i\in\mathcal{S}_{\omega(\tau)}$. Appropriate sets $\mathcal{S}_{\omega(\tau)}$ can be exactly determined so that the resultant scheme functions as an approximation scheme for \eqref{eq: second order diffusion equation}, \eqref{eq: second order diffusion equation 2}.

A particular direct consequence thereof is that for every $\tau >0$, only a \textit{finite} number of minimization steps \eqref{eq: KDE-MM-Scheme intro} is necessary. 

(see Remark \ref{rem: example simplifications KDE-MM-Scheme})

\item According to \eqref{eq: KDE-MM-Scheme intro 2}, the KDE-MM-Scheme uses a simple `particle distance' instead of a Wasserstein distance; the potential energy and the interaction energy can be simplified in a similar way. 

(see Remark \ref{rem: example simplifications KDE-MM-Scheme})

\item If the initial probability density is uniformly bounded and we aim at finding a uniformly bounded solution to \eqref{eq: second order diffusion equation}, \eqref{eq: second order diffusion equation 2}, it is possible to take account of this extra condition a priori, modifying the KDE-MM-Scheme and the selection of the parameters $n=n(\tau)$ correspondingly. 

(see Remark \ref{rem: uniformly bounded solution})
\end{itemize}

The KDE-MM-Scheme for a second order diffusion equation on the \textit{unbounded domain} $\Omega=\mathbb{R}^d$ is treated in Remarks \ref{rem: example weak Gamma KDE} ($\Gamma$-KDE-Approximation) and \ref{rem: selection unbounded domain} (selection of parameters). 

Furthermore, it is noteworthy that the concept of $\Gamma$-KDE-Approximation and KDE-MM-Scheme is well suited for \textit{fourth order examples} of \eqref{eq: diffusion equation intro}, too, see Remarks \ref{rem: Gamma KDE fourth order} and \ref{rem: fourth order}. 

The paper's purpose is to provide rigorous mathematical proofs substantiating good reasons for performing the KDE-MM-Scheme as variational approximation technique for \eqref{eq: diffusion equation intro}. 

An in-depth investigation of the literature on computational approaches to \eqref{eq: diffusion equation intro} will show that our approximation scheme is novel and unites the advantages of different approaches, making us optimistic about a practical implementation of the KDE-MM-Scheme, see Section \ref{sec: literature}. Such a process always involves a second stage besides the mathematical theory, in which part of the mathematical accuracy is carefully sacrificed for computational feasibility and cost economy. Corresponding further simplifications of the KDE-MM-Scheme are beyond the scope of this introductory paper. 

\textsf{Plan of the paper.} Section \ref{sec: literature} offers a thorough comparison of our KDE-MM-Scheme with existing approximation schemes for \eqref{eq: diffusion equation intro} from the literature. Sections \ref{sec: 2.1} and \ref{sec: 2.2} provide detailed information on \eqref{eq: diffusion equation intro} as steepest descent motion with regard to a Wasserstein distance and a general mathematical theory of a Minimizing Movement approach to \eqref{eq: diffusion equation intro} under the occurrence of $\Gamma$-perturbations of the energy functional. Section \ref{sec: 2.3} deals with uniform convergence rates for Kernel Density Estimators. The concept of $\Gamma$-KDE-Approximation and KDE-MM-Scheme is explicated in Section \ref{sec: 3.1}. Therein a general convergence statement / strong consistency statement is proved, showing that the KDE-MM-Scheme is a sound approximation scheme for gradient flows \eqref{eq: diffusion equation intro} in the space of probability measures. In Sections \ref{sec: 3.2} and \ref{sec: 3.3}, the KDE-MM-Scheme is applied to a wide class of second order diffusion equations; appropriate correlations between time step sizes and parameters are precisely quantified.

\section{Numerical Approaches from the Literature}\label{sec: literature}

The literature offers diverse numerical approaches to the diffusion equation
\begin{equation}\label{eq: diffusion equation section 2}
\partial_t u  - \nabla\cdot\Big(u\nabla\frac{\delta\phi}{\delta u}(u)\Big) \ = \ 0, \quad\quad u\geq0, 
\end{equation}
built on the structure of solutions as a gradient flow with regard to the energy functional $\phi: \mathcal{P}_2(\mathbb{R}^d)\to(-\infty, +\infty]$ and the $2$-Wasserstein distance $\mathcal{W}_2$, which is defined as
\begin{equation*}
\mathcal{W}_2(\mu_1, \mu_2)^2 := \min_{\gamma\in\Gamma(\mu_1,\mu_2)}\int_{\mathbb{R}^d\times\mathbb{R}^d}{|x-y|^2\mathrm{d}\gamma}, \ \ \ \mu_i\in\mathcal{P}_2(\mathbb{R}^d), 
\end{equation*}
with $\Gamma(\mu_1, \mu_2)$ being the set of Borel probability measures on $\mathbb{R}^d\times\mathbb{R}^d$ whose first and second marginals coincide with $\mu_1$ and $\mu_2$ respectively (see e.g. \cite{Villani03, Villani09} for a detailed account of the theory of Optimal Transport and Wasserstein distances). Unlike the KDE-MM-Scheme, most of these approximation schemes lack mathematical proofs of convergence / consistency statements or focus on the numerical simplification of a single step in the JKO scheme
\begin{equation}\label{eq: JKO}
\mu_\tau^m \quad\text{ is a minimizer for }\quad \phi(\cdot) \ + \ \frac{1}{2\tau}\mathcal{W}_2(\cdot, \mu_\tau^{m-1})^2 
\end{equation}
not necessarily reproducing the dynamics of the whole scheme. Still, the mathematical ideas behind them are interesting and moreover, the corresponding articles provide a computational implementation of the schemes; it is worthwhile comparing them with our novel theory. 

A key issue is the computation of the Wasserstein distance in the JKO scheme, which we can circumvent by using a `particle distance' in $\mathbb{R}^d$ in the KDE-MM-Scheme, see \eqref{eq: KDE-MM-Scheme intro 2}. 

In \cite{benamou2016discretization}, the $\mathcal{W}_2$-term is tackled by first applying Brenier's Theorem \cite{brenier1991polar}, thus reformulating the minimum problem \eqref{eq: JKO} as a search for a convex function $g_\tau^m: \mathbb{R}^d\to \mathbb{R}$ solving
\begin{equation}\label{eq: brenier}
\min\Big\{\phi(\nabla g_{\#}\mu_\tau^{m-1}) + \frac{1}{2\tau} \int_{\mathbb{R}^d}{|x-\nabla g(x)|^2\mathrm{d}\mu_\tau^{m-1}} \ \Big| \ g:\mathbb{R}^d\to\mathbb{R} \text{ convex}\Big\}. 
\end{equation}
In a second stage, \eqref{eq: brenier} is replaced by a minimum problem associated with a finite subset $\{x_j\}_j\subset\mathbb{R}^d$, involving a discrete probability measure $\sum_j{a_j\delta_{x_j}}$ in place of $\mu_\tau^{m-1}$ and the notion of a discrete counterpart of a convex function and its subdifferential defined on $\{x_j\}_j$ only. Supposing that the density $u_\tau^{m-1}$ of $\mu_\tau^{m-1}=u_\tau^{m-1}\mathcal{L}^d$ is bounded from above and from below by positive constants and that $\mu_\tau^{m-1}$ is approximated by discrete probability measures, the authors can finally construct a minimizing sequence to the minimum problem \eqref{eq: JKO} for a wide class of energy functionals $\phi$. Introducing Laguerre cells as discrete subdifferentials, they use techniques from computational geometry for the computational implementation of their scheme. Also the authors of \cite{mokrov2021large} start out with the minimum problem \eqref{eq: brenier} but they restrict it to convex functions $g_\theta$ generated by a neural network \cite{amos2017input}. They estimate the energy and distance terms using i.i.d. samples $X_1, ..., X_n$ from $\mu_\tau^{m-1}$ and $\nabla g_\theta(X_1), ... , \nabla g_\theta(X_n)$ from $\nabla g_{\theta_{\#}}\mu_\tau^{m-1}$ with the aim of finding the optimum parameters for the resultant minimum problem associated with the chosen neural network. Another way of dealing with the $\mathcal{W}_2$-term is presented in \cite{benamou2016augmented} in which the Benamou-Brenier formula \cite{Benamou-Brenier00}
\begin{equation}\label{eq: benamou}
\mathcal{W}_2(\rho_0\mathcal{L}^d, \rho_1\mathcal{L}^d)^2 = \inf\Big\{\int_{0}^1{\int_{\mathbb{R}^d}{\rho(t,x)|v(t,x)|^2\mathrm{d}x}\mathrm{d}t}\ \Big| \ \rho_0\stackrel{(\rho, v)}{\rightsquigarrow}\rho_1\Big\}, 
\end{equation}
$\rho_0\stackrel{(\rho, v)}{\rightsquigarrow}\rho_1 \ :\Leftrightarrow \ \partial_t\rho + \nabla\cdot(\rho v) = 0, \ \rho(i, \cdot)=\rho_i \ (i=0,1),$ is substituted in \eqref{eq: JKO}. Transforming the constraint $\rho_0\stackrel{(\rho, v)}{\rightsquigarrow}\rho_1$ by means of Lagrange multipliers, the authors obtain a saddle-point problem which they approach by a version of the augmented Lagrangian algorithm ALG2 from \cite{fortin2000augmented} combined with a finite element method. A different approach also based on the Benamou-Brenier formula \eqref{eq: benamou} and a finite element method can be found in \cite{burger2009mixed}. 

Simplifying the energy term and the minimum problem \eqref{eq: JKO} as a whole is another key issue besides the $\mathcal{W}_2$-term. In \cite{hwang2021deep}, it is proposed to construct probability density functions by normalizing nonnegative functions that are generated by a feedforward neural network and to restrict \eqref{eq: JKO} to probability measures with such densities. The authors are motivated by Thm. 3 in \cite{hornik1991approximation} saying that if $G\in\mathrm{C}^1(\mathbb{R})$ is bounded and nonconstant and $X\subset\mathbb{R}^d$ is compact, then the set of functions
\begin{equation}\label{eq: universal approximation nn}
\bigcup_{n\in\mathbb{N}}\Big\{X \ni x\mapsto \sum_{j=1}^{n}{\beta_j G(a_j^{T}x-c_j)} \ \Big| \ \beta_j\in \mathbb{R}, a_j\in\mathbb{R}^d, c_j\in\mathbb{R} \Big\}
\end{equation}
is dense in $\mathrm{C}^1(X)$; this theorem belongs to the universal approximation theory of feedforward neural networks, which has its origin in \cite{cybenko1989approximation, hornik1989multilayer, hornik1990universal}. In \cite{hwang2021deep}, no rule is provided, however, stating how to choose a neural network from \eqref{eq: universal approximation nn}, i.e. how to fix the number of its hidden units (corresponds to the number $n$ in \eqref{eq: universal approximation nn}) when approaching \eqref{eq: JKO}. Furthermore, the authors admit the drawback of their approach based on \eqref{eq: universal approximation nn} that there is ``no direct easy way to evaluate'' the $2$-Wasserstein distance between the measures they construct. 

Our KDE-MM-Scheme \eqref{eq: KDE-MM-Scheme intro}-\eqref{eq: KDE-MM-Scheme intro 2} offers both a comparatively easy computation of the distance term and a significant simplification of the whole minimization procedure by restricting it to clearly structured probability measures of the form
$\frac{1}{n(\tau)}\sum_{i=1}^{n(\tau)}{\K_{h(n(\tau))}(\cdot - y_i)}\mathcal{L}^d$. As outlined in Section \ref{sec: KDE-MM-Scheme intro}, our motivation for bringing such measures into the focus of our approach lies in the KDE theory, but we note that the corresponding probability densities can also be regarded as output functions of a specific feedforward neural network. Obviously having a different architecture than the neural networks from \eqref{eq: universal approximation nn}, it allows of an interpretation of the KDE-MM-Scheme not only as discrete-time steepest descent motion in $\mathcal{P}_2(\mathbb{R}^d)$ but also as discrete-time steepest descent \textit{particle} motion, cf. \eqref{eq: KDE-MM-Scheme intro}-\eqref{eq: KDE-MM-Scheme intro 2}. 

The following approaches from the literature are aimed at translating \eqref{eq: diffusion equation section 2} into a scheme for moving particles. In \cite{matthes2014convergence, matthes2017convergent, junge2017fully, carrillo2018lagrangian}, they start out with the Lagrangian formulation 
\begin{equation*}
\partial_t \mathbb{X}_t \ = \ v_t\circ\mathbb{X}_t, \quad \quad v_t \ = \ -\nabla\frac{\delta\phi}{\delta u}(u_t), 
\end{equation*}
of \eqref{eq: diffusion equation section 2}, which, by \cite{evans2005diffeomorphisms}, can be transformed into an $\mathrm{L}^2$-gradient flow equation for the Lagrangian map $\mathbb{X}$ related to $u$ through `pushing forward' a reference measure. This $\mathrm{L}^2$-gradient flow is tackled by a suitable Minimizing Movement scheme, in turn simplified by spatial discretizations. 
If the dimension $d=1$, the proof of certain convergence statements is possible \cite{matthes2014convergence, matthes2017convergent}. The particle methods presented in \cite{leclerc2020lagrangian} and \cite{carrillo2019blob} each consist in setting up ODE systems 
\begin{equation}\label{eq: ODE system particles}
a_j \mathrm{x}_{j}'(t) \ = \ -\nabla_{x_j} f_{\epsilon}^{(N)}(\mathrm{x}_{1}(t), ..., \mathrm{x}_{N}(t))
\end{equation}
with $f_{\epsilon}^{(N)}(x_1, ..., x_N):=\phi_{\epsilon}\Big(\sum_{j=1}^{N}{a_j\delta_{x_j}}\Big), \ \sum_{j=1}^{N}{a_j}=1,$ for moving particles $\mathrm{x}_{1}, ..., \mathrm{x}_{N},$ associated with $\Gamma$-perturbations $\phi_\epsilon$ of $\phi$ having finite energy $\phi_\epsilon(\sum_{j=1}^{N}{a_j\delta_{\mathrm{x}_j}}) < +\infty$ at discrete probability measures. Solutions of \eqref{eq: ODE system particles} are assigned discrete measures $\mu_{N, \epsilon}(t):=\sum_{j=1}^{N}{a_j\delta_{\mathrm{x}_{j}(t)}}$. In \cite{leclerc2020lagrangian}, $\phi_\epsilon$ is defined as the Moreau-Yosida approximation
\begin{equation*}
\phi_\epsilon(\mu) \ := \ \inf_\nu \Big[\phi(\nu) + \frac{1}{2\epsilon}\mathcal{W}_2(\nu, \mu)^2\Big]
\end{equation*}
of e.g. $\phi(\mu):= \begin{cases}\int{u\log(u)\mathrm{d}x} \text{ if } \mu=u\mathcal{L}^d\llcorner\Omega, \\ +\infty \text{ else, }\end{cases}$ and if $d=1$, $\Omega$ is a bounded interval and $a_j\equiv \frac{1}{N}$ in \eqref{eq: ODE system particles}, the authors can precisely quantify the relation between the particle numbers $N$ and the parameters $\epsilon=\epsilon_N$ (depending on initial data) to obtain convergence (up to a subsequence) of $\mu_{N, \epsilon_N}(\cdot)$ to a solution of the heat equation as $N\uparrow+\infty$ and $\epsilon_N\downarrow0$. The numerical computation of $f_{\epsilon}^{(N)}$ is based on the relation between semi-discrete optimal transport and the concept of Laguerre cells and on tools from computational geometry, see \cite{kitagawa2019convergence}. The authors of \cite{carrillo2019blob} define 
\begin{equation*}
\phi_\epsilon(\mu):=\frac{1}{m-1}\int{((\xi_\epsilon\ast\xi_\epsilon)\ast\mu)^{m-1}\mathrm{d}\mu}, \quad\quad \xi_\epsilon:=\epsilon^{-d}\xi(\cdot/\epsilon), 
\end{equation*}
for $\phi(\mu):=\begin{cases} \frac{1}{m-1}\int{u^m} \text{ if } \mu=u\mathcal{L}^d, \\ +\infty \text{ else,}\end{cases} m\geq2$, and a suitable function $\xi$ (e.g. Gaussian function), proving that $\mu_{N, \epsilon}$ defined as above is the unique weak solution to the diffusion equation
\begin{equation}\label{eq: diffusion equation epsilon section 2}
\partial_t u_\epsilon  - \nabla\cdot\Big(u_\epsilon\nabla\frac{\delta\phi_{\epsilon}}{\delta u}(u_\epsilon)\Big) \ = \ 0
\end{equation}
with initial datum $\sum_{j=1}^{N}{a_j\delta_{\mathrm{x}_j(0)}}$ (``particles remain particles''). 
If $m=2$ and under a specific assumption on the discrete initial data, the convergence of such `particle solutions' $\mu_{N_\epsilon, \epsilon}$ of \eqref{eq: diffusion equation epsilon section 2} to a weak solution of the porous medium equation as $\epsilon\downarrow0$ can be proved \cite{craig2022blob} by means of the Serfaty-Sandier approach \cite{sandier2004gamma, serfaty2011gamma}, the `flow interchange method' from \cite{matthes2009family} and `contraction estimates' from \cite{AGS08}, using the $\lambda_\epsilon$-convexity of $\phi_\epsilon$ along generalized geodesics. 

In terms of the Minimizing Movement scheme 
\begin{equation}\label{eq: JKO epsilon}
\mu_{\tau, \epsilon}^m \quad\text{ is a minimizer for }\quad \phi_\epsilon(\cdot) \ + \ \frac{1}{2\tau}\mathcal{W}_2(\cdot, \mu_{\tau, \epsilon}^{m-1})^2 \quad\quad (m\in\mathbb{N}),
\end{equation}
associated with $\Gamma$-perturbations $\phi_\epsilon$ of $\phi$, considerations regarding the limit behaviour of \eqref{eq: diffusion equation epsilon section 2} as $\epsilon\downarrow0$ relate to first letting the time step size $\tau\downarrow0$ in \eqref{eq: JKO epsilon} for a fixed parameter $\epsilon>0$ and only then letting $\epsilon\downarrow0$. By comparison, our analysis of the gradient flow motion along a sequence of $\Gamma$-converging energy functionals, which is built on \cite{fleissner2016gamma}, focuses on a `joint discrete-time steepest descent motion', i.e. we let the time step sizes $\tau$ and parameters $\epsilon$ simultaneously go to $0$, thus establishing a direct connection with \eqref{eq: diffusion equation section 2} 
(see Section \ref{sec: 2.2} and Section \ref{sec: 1.3.1} for an overview). The limit $\epsilon\downarrow0$ in \eqref{eq: diffusion equation epsilon section 2}, by contrast, is in general \textit{not} related to \eqref{eq: diffusion equation section 2}, cf. Sect. 1 in \cite{fleissner2016gamma}. The special cases in which solutions of \eqref{eq: diffusion equation epsilon section 2} do converge (up to a subsequence) to a solution of \eqref{eq: diffusion equation section 2} as $\epsilon\downarrow0$ correspond to cases in which our joint discrete-time steepest descent motion yields solutions of \eqref{eq: diffusion equation section 2} for any choice $\epsilon=\epsilon(\tau)\to0$, cf. Sect. 5 in \cite{fleissner2016gamma} and Remark \ref{rem: main ass}\ref{rem: main ass iii}. 

We end our account of literature with a reference to the concept of `entropic regularization' from \cite{peyre2015entropic, carlier2017convergence}; therein the JKO scheme \eqref{eq: JKO} is replaced by minimum problems
\begin{equation}\label{eq: JKO entropic}
\mu_\tau^m \quad\text{ is a minimizer for }\quad \phi(\cdot) \ + \ \frac{1}{2\tau}\mathcal{W}_{2,\epsilon}(\cdot, \mu_\tau^{m-1})^2 
\end{equation}
associated with the non-distance $\mathcal{W}_{2,\epsilon}, \ \epsilon >0,$ defined as
\begin{equation*}
\mathcal{W}_{2,\epsilon}(\mu_1, \mu_2)^2 := \inf_{\gamma\in\Gamma(\mu_1,\mu_2)}\Big[\int_{\mathbb{R}^d\times\mathbb{R}^d}{|x-y|^2\mathrm{d}\gamma} \ + \ \epsilon\cdot\mathcal{H}(\gamma)\Big], 
\end{equation*}
with $\mathcal{H}(\gamma):=\begin{cases}\int_{\mathbb{R}^d\times\mathbb{R}^d}{w(x)\log(w(x))\mathrm{d}x} \text{ if } \gamma=w\mathcal{L}^{2d}, \\ +\infty\text{ else.}\end{cases}$ Following the classical `JKO proof strategy' from \cite{jordan1998variational}, the authors of \cite{carlier2017convergence} show that under certain assumptions on $\phi$ and the relation $\epsilon\sim\tau$, also the `entropic scheme' \eqref{eq: JKO entropic} yields solutions to \eqref{eq: diffusion equation section 2}. In practice, their numerical approach consists in a spatial grid discretization $\{x_j\}_{j=1}^{N}$, a simplification of \eqref{eq: JKO entropic} restricted to discrete probability measures $\sum_{j=1}^{N}{a_j\delta_{x_j}}$ on the grid and an application of a modified version of Sinkhorn's algorithm (cf. Sect. 4.2 in \cite{peyre2019computational}).

\section{Gradient Flows, Minimizing Movements and Kernel Density Estimation}\label{sec: 2}
\subsection{Gradient Flows in $(\mathcal{P}_p(\mathbb{R}^d), \mathcal{W}_p)$}\label{sec: 2.1}
A gradient flow in a metric space is typically characterized by an energy dissipation inequality; the abstract notion thereof originates from \cite{DeGiorgi-Marino-Tosques80} with further developments in \cite{Degiovanni-Marino-Tosques85, Marino-Saccon-Tosques89} and \cite{AGS08}. Our analysis in Sections \ref{sec: 2.2} and \ref{sec: 3.1}-\ref{sec: 3.3} is based on such characterization of gradient flows in  $(\mathcal{P}_p(\mathbb{R}^d), \mathcal{W}_p)$. This section provides the relevant definitions (locally absolutely continuous curve, metric derivative, continuity equation, local and relaxed slope, strong and limiting subdifferential) and serves to bridge the gap between the diffusion equation \eqref{eq: diffusion equation intro} and its associated energy dissipation (in)equality by a suitable chain rule, see Proposition \ref{prop: 2.1} below. 

Let $p\in(1,+\infty)$ and $\mathcal{P}_p(\mathbb{R}^d)$ be the space of Borel probability measures with finite moments of order $p$ (i.e. $\int_{\mathbb{R}^d}{|x|^p\mathrm{d} \mu} < +\infty$), endowed with the $p$-Wasserstein distance $\mathcal{W}_p$, 
\begin{equation*}
\mathcal{W}_p(\mu_1, \mu_2)^p := \min_{\gamma\in\Gamma(\mu_1,\mu_2)}\int_{\mathbb{R}^d\times\mathbb{R}^d}{|x-y|^p\mathrm{d}\gamma}, \ \ \ \mu_i\in\mathcal{P}_p(\mathbb{R}^d), 
\end{equation*}
with $\Gamma(\mu_1, \mu_2)$ being the set of Borel probability measures on $\mathbb{R}^d\times\mathbb{R}^d$ whose first and second marginals coincide with $\mu_1$ and $\mu_2$ respectively (see e.g. \cite{Villani03, Villani09} for a detailed account of the theory of Optimal Transport and Wasserstein distances). Moreover, let $q\in(1,+\infty)$ be the conjugate exponent of $p$, and 
\begin{equation}\label{eq: j_q}
j_q: \mathbb{R}^d\to\mathbb{R}^d, \quad  j_q(v) := \begin{cases} |v|^{q-2}v & \text{ if } v\neq 0, \\ 0 & \text{ if } v = 0. \end{cases}
\end{equation}

\begin{definition}[Locally absolutely continuous curve]\label{def: absolutely continuous}
We say that a curve $\nu: \ [0, +\infty) \rightarrow \mathcal{P}_p(\mathbb{R}^d)$ is locally absolutely continuous in $(\mathcal{P}_p(\mathbb{R}^d), \mathcal{W}_p)$ 
if there exists $m\in \mathrm{L}_{\mathrm{loc}}^1(0,+\infty)$ such that 
\begin{equation*}
\mathcal{W}_p(\nu(s),\nu(t)) \leq \int^{t}_{s}{m(r) \mathrm{d}r} \quad \text{for all } 0 \leq s\leq t < +\infty. 
\end{equation*}
 \end{definition}

If $(\nu_t)_{t\ge0}$ is a locally absolutely continuous curve in $(\mathcal{P}_p(\mathbb{R}^d), \mathcal{W}_p)$, the limit
\begin{equation*}
|\nu'|(t) := \mathop{\lim}_{s\to t} \frac{\mathcal{W}_p(\nu(s),\nu(t))}{|s-t|}
\end{equation*}
exists for $\mathcal{L}^1$-a.e. $t\in (0, +\infty)$, the \textit{metric derivative} $t \mapsto |\nu'|(t)$ belongs to $\mathrm{L}_{\mathrm{loc}}^1(0, +\infty)$ and is $\mathcal{L}^1$-a.e. the smallest admissible function $m$ in the definition above (cf. Thm. 1.1.2 in \cite{AGS08}). Moreover, there exists an essentially unique Borel vector field $w: [0, +\infty)\times\mathbb{R}^d\to\mathbb{R}^d$ satisfying both
\begin{equation}\label{eq: tangent 1}
w_t\in\mathrm{L}^p(\nu_t; \mathbb{R}^d), \quad\quad \|w_t\|_{\mathrm{L}^p(\nu_t; \mathbb{R}^d)} \ = \ |\nu'|(t) \quad \text{for } \mathcal{L}^1\text{-a.e. } t > 0, 
\end{equation}
and the \textit{continuity equation}
\begin{equation}\label{eq: tangent 2}
\partial_t\nu_t  + \nabla\cdot(w_t\nu_t) \ = \ 0,
\end{equation}
in the distributional sense, i.e. 
\begin{equation*}
\int_0^{+\infty}{\int_{\mathbb{R}^d}{(\partial_t\xi(t,x) + \langle\nabla\xi(t,x), w(t,x)\rangle) \mathrm{d}\nu_t(x)}\mathrm{d}t} \ = \ 0 
\end{equation*}
for all $\xi\in\mathrm{C}^{\infty}_{c}((0,+\infty)\times \mathbb{R}^d)$ (cf. Thm. 8.3.1 and Prop. 8.4.5 in \cite{AGS08}). We refer to $w$ satisfying \eqref{eq: tangent 1} and \eqref{eq: tangent 2} as ``tangent'' vector field associated with the curve $(\nu_t)_{t\ge0}$.

An abstraction of the modulus of the gradient to the general metric and nonsmooth setting and a subdifferential calculus in $(\mathcal{P}_p(\mathbb{R}^d), \mathcal{W}_p)$ are further basic ingredients for the theory of gradient flows with regard to an energy functional $\phi: \mathcal{P}_p(\mathbb{R}^d)\to(-\infty, +\infty]$ and the $p$-Wasserstein distance $\mathcal{W}_p$.

We assume that 
\begin{enumerate}[label=($\phi$\arabic*)]
\item 
$\phi(\mu)<+\infty$ implies $\mu\ll\mathcal{L}^d$, \label{ass: phi1}
\item there exist $A, B > 0, \ \mu_\star\in\mathcal{P}_p(\mathbb{R}^d)$ s.t. 
$\phi(\cdot) \ \geq \ -A-B\mathcal{W}_p(\cdot, \mu_\star)^p$, \label{ass: phi2}
\item $\phi$ is lower semicontinuous in $(\mathcal{P}_p(\mathbb{R}^d), \mathcal{W}_p)$, 
\label{ass: phi3}
\item $\mathcal{W}_p$-bounded subsets of sublevel sets of $\phi$ are relatively compact, i.e. 
\begin{equation*}
\sup_{n\in\mathbb{N}}\{\phi(\mu_n), \mathcal{W}_p(\mu_n, \mu_\star)\} < +\infty \quad \Rightarrow \quad \exists n_k\uparrow+\infty, \mu: \ \mathcal{W}_p(\mu_{n_k}, \mu) \to 0. 
\end{equation*}
\label{ass: phi4}
\end{enumerate}

\begin{definition}[local and relaxed slope]\label{def: slopes}
 The local slope $|\partial\phi|$ of $\phi$ at a probability measure $\mu\in\{\phi<+\infty\}$ is defined as
\begin{equation*}
|\partial\phi|(\mu):= \mathop{\limsup}_{\mathcal{W}_p(\nu,\mu) \to 0} \frac{(\phi(\mu)-\phi(\nu))^+}{\mathcal{W}_p(\mu,\nu)}. 
\end{equation*}
The relaxed slope $|\partial^-\phi|$ is a slight modification of the lower semicontinuous envelope of the local slope, i.e.
\begin{equation*}
|\partial^-\phi|(\mu):= \inf\{\liminf_{n\to+\infty} |\partial\phi|(\mu_n) \ : \  \lim_{n\to+\infty}\mathcal{W}_p(\mu_n, \mu) = 0, \ \sup_n \phi(\mu_n)< +\infty\}. 
\end{equation*}
\end{definition}

\begin{definition}[strong and limiting subdifferential]\label{def: subdifferentials}
The strong subdifferential $\partial_s\phi(\mu)$ of $\phi$ at $\mu\in\{\phi<+\infty\}$ is defined as the set of vector fields $\zeta\in\mathrm{L}^q(\mu; \mathbb{R}^d)$ satisfying 
\begin{equation*}
\phi(T_{\#}\mu) - \phi(\mu) \ \geq \ \int_{\mathbb{R}^d}{\langle \zeta(x), T(x)-x\rangle\mathrm{d}\mu(x)} \ + \ o(\|T-\mathrm{id}\|_{\mathrm{L}^p(\mu; \mathbb{R}^d)})
\end{equation*}
for every $T\in\mathrm{L}^p(\mu; \mathbb{R}^d)$. 

The limiting subdifferential $\partial_l\phi(\mu)$ of $\phi$ at $\mu\in\{\phi<+\infty\}$ is defined as the set of vector fields $\zeta\in\mathrm{L}^q(\mu; \mathbb{R}^d)$ for which there exist $\mu_n\stackrel{\mathcal{W}_p}{\longrightarrow}\mu$ and $\zeta_n\in\partial_s\phi(\mu_n) \ (n\in\mathbb{N})$ such that $\sup_n\Big\{\phi(\mu_n), \int_{\mathbb{R}^d}{|\zeta_n(x)|^q\mathrm{d}\mu_n(x)}\Big\} < +\infty$ and $\zeta_n\mu_n$ converges to $\zeta\mu$ in the distributional sense. 
\end{definition}

Whereas local and relaxed slopes play an important role in the abstract theory of gradient flows in general metric spaces (cf. \cite{DeGiorgi-Marino-Tosques80, Degiovanni-Marino-Tosques85, Marino-Saccon-Tosques89, AGS08}), the definitions of $\partial_s\phi$ and $\partial_l\phi$ translate the concept of the Fr\'echet subdifferential and its closure from a Banach space into a suitable notion in $(\mathcal{P}_p(\mathbb{R}^d), \mathcal{W}_p)$. Assuming $\phi(\mu):=\int{\mathcal{F}(x,u)\mathrm{d}x}$ or $\phi(\mu):=\int{\mathcal{F}(x,u, \nabla u)\mathrm{d}x}$ (for $\mu=u\mathcal{L}^d$), it is not difficult to see that Definition \ref{def: subdifferentials} of $\partial_s\phi(\mu)$ and $\partial_l\phi(\mu)$ generalizes the expression $\nabla\frac{\delta\phi}{\delta u}(u)$ from standard variational calculus for integral functionals to the nonsmooth setting in $(\mathcal{P}_p(\mathbb{R}^d), \mathcal{W}_p)$, cf. Lem. 10.4.1 in \cite{AGS08}. The heuristic principle that 
\begin{equation*}
\partial_l\phi(\mu) = \Big\{\nabla\frac{\delta\phi}{\delta u}(u)\Big\}
\end{equation*}
is further substantiated through the concrete computation of limiting subdifferentials, see Example \ref{ex: chain rule} of a general second order diffusion equation, and Sect. 5.3 in \cite{gianazza2009wasserstein} and Sect. 2.4 in \cite{matthes2009family} both dealing with fourth order examples of \eqref{eq: diffusion equation intro}. 
Definitions \ref{def: slopes} and \ref{def: subdifferentials} of relaxed slope and limiting subdifferential are ideally suited for a Minimizing Movement approach to a steepest descent with regard to a general energy functional $\phi$ and the $p$-Wasserstein distance. 

 \begin{proposition}[Gradient flow in  $(\mathcal{P}_p(\mathbb{R}^d), \mathcal{W}_p)$] \label{prop: 2.1} 
We assume \ref{ass: phi1}, \ref{ass: phi2}, \ref{ass: phi3}, \ref{ass: phi4} and that for every $\nu\in\{\phi<+\infty\}$ the limiting subdifferential $\partial_l\phi(\nu)$ contains at most one element to which we refer as $D_l\phi(\nu)(\cdot)$ if it exists. 

If $(\tau_k)_{k\in\mathbb{N}}, \ \tau_k\downarrow0,$ is a sequence of time step sizes, $\bar{\mu}_{\tau_k}$ are discrete solutions \eqref{eq: piecewise constant interpolation} of the Minimizing Movement scheme \eqref{eq: general MM scheme} associated with
\begin{equation}\label{eq: MM p}
\Phi(\tau, \mu, \nu) \ := \ \phi(\nu)+\frac{1}{p\tau^{p-1}}\mathcal{W}_p(\nu, \mu)^p
\end{equation}
and initial data $\bar{\mu}_{\tau_k}(0)\stackrel{\mathcal{W}_p}{\longrightarrow}\mu^0\in\{\phi<+\infty\}, \ \lim_{k\to+\infty}\phi(\bar{\mu}_{\tau_k}(0)) =  \phi(\mu^0)$, 
then there exist a subsequence of time step sizes $(\tau_{k_l})_{l\in\mathbb{N}}, \ \tau_{k_l}\downarrow0,$ and a curve $\mu: [0, +\infty)\to\mathcal{P}_p(\mathbb{R}^d)$ such that
\begin{equation*}
\lim_{l\to+\infty}\mathcal{W}_p(\bar{\mu}_{\tau_{k_l}}(t), \mu(t)) \ = \ 0 \quad \text{ for all } t\ge0. 
\end{equation*}
Every such limit curve $\mu$ is locally absolutely continuous in $(\mathcal{P}_p(\mathbb{R}^d), \mathcal{W}_p)$ and 
\begin{equation}\label{eq: energy inequality}
\phi(\mu(0)) - \phi(\mu(t)) \ \ge \ \frac{1}{q} \int^{t}_{0}{|\partial^-\phi|^q(\mu(r)) \mathrm{d}r} \ + \ \frac{1}{p} \int^{t}_{0}{|\mu'|^p(r) \mathrm{d}r}
\end{equation}
for all $t\ge0$. 

Assuming in turn $\mu$ is a locally absolutely continuous curve in $(\mathcal{P}_p(\mathbb{R}^d), \mathcal{W}_p)$ with tangent vector field $v$ and that $\phi \circ  \mu \in \mathrm{C}([0, +\infty))\cap \mathrm{W}_{\mathrm{loc}}^{1,1}((0, +\infty))$ with
\begin{equation}\label{eq: chain rule} 
\partial_l\phi(\mu(t))\neq\emptyset, \quad\quad (\phi\circ\mu)'(t) \ = \ \int_{\mathbb{R}^d}{\langle D_l\phi(\mu(t))(x), v_t(x)\rangle  \mathrm{d}\mu_t(x)}
\end{equation}
for $\mathcal{L}^1$-a.e. $t>0$, 
the following three statements \ref{st: a}, \ref{st: b} and \ref{st: c} are equivalent: 
\begin{enumerate}[label=(\alph*)]
\item $\mu$ satisfies the energy inequality \eqref{eq: energy inequality} for all $t\ge0$. \label{st: a}
\item $\mu$ satisfies the energy dissipation equality 
\begin{equation}\label{eq: ede}
\phi(\mu(s)) - \phi(\mu(t)) \ = \ \frac{1}{q} \int^{t}_{s}{|\partial^-\phi|^q(\mu(r)) \mathrm{d}r} \ + \ \frac{1}{p} \int^{t}_{s}{|\mu'|^p(r) \mathrm{d}r} 
\end{equation}
for all $0\leq s \leq t < +\infty$. \label{st: b}
\item $\mu$ satisfies both the differential equation
\begin{equation}\label{eq: differential inclusion}
v_t \ = \ -j_q(D_l\phi(\mu(t))) \quad\quad \mu_t\text{-a.e.}
\end{equation}
and 
\begin{equation}\label{eq: subdifferential slope equality}
|\partial^-\phi|(\mu(t)) = \|D_l\phi(\mu(t))\|_{\mathrm{L}^q(\mu_t; \mathbb{R}^d)}
\end{equation}
for $\mathcal{L}^1$-a.e. $t>0$.  \label{st: c}
\end{enumerate}
\end{proposition}

\begin{proof}
Such a limit curve $\mu$ of discrete solutions of the scheme \eqref{eq: general MM scheme}, \eqref{eq: MM p} exists and is locally absolutely continuous by Prop. 2.2.3 in \cite{AGS08} and it directly follows from the first part of the proof of Thm. 2.3.3  and Rem. 3.2.5 therein that $\mu$ satisfies the energy inequality \eqref{eq: energy inequality}. 

Let us prove the second part of Proposition \ref{prop: 2.1}. We assume that  $\mu$ is a locally absolutely continuous curve in $(\mathcal{P}_p(\mathbb{R}^d), \mathcal{W}_p)$ with tangent vector field $v$ and that the chain rule \eqref{eq: chain rule} holds good for $\phi \circ  \mu \in \mathrm{C}([0, +\infty))\cap \mathrm{W}_{\mathrm{loc}}^{1,1}((0, +\infty))$, leading to
\begin{equation}\label{eq: chain rule integral}
\phi(\mu(t)) - \phi(\mu(s)) \ = \ \int_s^t{\int_{\mathbb{R}^d}{\langle D_l\phi(\mu(r))(x), v_r(x)\rangle \mathrm{d}\mu_r(x)}\mathrm{d}r}
\end{equation}
for all $0\leq s\leq t < +\infty$. A simple adaptation of the proof of Lem. 4.6 in \cite{ambrosio2006stability} and an application of Lem. 10.3.4 and Rem. 3.1.7 in \cite{AGS08} show that
\begin{equation}\label{eq: subdifferential slope inequality}
|\partial^-\phi|(\nu) < +\infty \quad \Rightarrow \quad \partial_l\phi(\nu)\neq\emptyset, \quad \|D_l\phi(\nu)\|_{\mathrm{L}^q(\nu; \mathbb{R}^d)} \ \leq \ |\partial^-\phi|(\nu). 
\end{equation}
We infer from \eqref{eq: chain rule integral}, \eqref{eq: subdifferential slope inequality}, \eqref{eq: tangent 1} and an application of Cauchy-Schwarz inequality and Young's inequality that \ref{st: a} $\Rightarrow$ \ref{st: b} $\Rightarrow$ \ref{st: c} $\Rightarrow$ \ref{st: a}. The proof is complete. 

\end{proof}

Proposition \ref{prop: 2.1} offers a connection between three different approaches to the notion of steepest descent in $(\mathcal{P}_p(\mathbb{R}^d), \mathcal{W}_p)$. A gradient flow in a metric space (\textit{``curves of maximal slope''}) is typically characterized by an energy dissipation inequality like 
\begin{equation}\label{eq: edi}
\phi(\mu(s)) - \phi(\mu(t)) \ \geq \ \frac{1}{q} \int^{t}_{s}{|\partial^-\phi|^q(\mu(r)) \mathrm{d}r} \ + \ \frac{1}{p} \int^{t}_{s}{|\mu'|^p(r) \mathrm{d}r} 
\end{equation}
for $0\leq s \leq t < +\infty$  \cite{DeGiorgi-Marino-Tosques80, Degiovanni-Marino-Tosques85, Marino-Saccon-Tosques89, AGS08}.
Limit curves of \textit{discrete-time steepest descents} in metric spaces are generally characterized by such energy inequality \eqref{eq: energy inequality} and the proof that also an energy dissipation (in)equality like \eqref{eq: ede} (\eqref{eq: edi}) holds good typically involves a metric chain rule, cf. Def. 1.2.1 of ``strong upper gradient'' in \cite{AGS08} and the proof of Thm. 2.3.3 therein. 
As outlined above, the limiting subdifferential can be identified with $\nabla\frac{\delta\phi}{\delta u}(u)$ so that the gradient-flow-type \textit{differential equation} \eqref{eq: differential inclusion}, together with the continuity equation \eqref{eq: tangent 2}, represents a weak reformulation of the diffusion equation \eqref{eq: diffusion equation intro} in the space of probability measures. 

\begin{example}[The displacement convex case]\label{ex: convex case}
The linkage between different approaches to the notion of a steepest descent w.r.t. a displacement convex energy functional is well examined, see Thm. 11.1.3 in \cite{AGS08}. The underlying structure fits into our Proposition \ref{prop: 2.1}. 

Let $T_\mu^\nu$ denote the unique optimal transport map from $\mu$ to $\nu$ w.r.t. $\mathcal{W}_p$ for $\mu, \nu \in \{\phi < +\infty\}$ (see Thm. 1.2 from \cite{gangbo1996geometry} and condition \ref{ass: phi1}). If $\phi$ is displacement convex, i.e. 
\begin{equation*}
[0,1]\ni t\mapsto \phi\big(((1-t)\mathrm{id} + t\cdot T_\mu^\nu)_{\#}\mu\big) \quad\quad \text{ is convex }
\end{equation*}
for every $\mu, \nu\in\{\phi < +\infty\}$ (see \cite{mccann1997convexity}, often referred to as convex along constant speed geodesics \cite{AGS08}), then $|\partial^-\phi|\equiv|\partial\phi|$, every $\zeta\in\partial_l\phi(\mu)$ satisfies
\begin{equation*}
\phi(\nu) - \phi(\mu) \ \geq \ \int_{\mathbb{R}^d}{\langle \zeta(x), T_\mu^\nu(x)-x\rangle\mathrm{d}\mu(x)} \quad\quad \forall \nu\in\{\phi<+\infty\}, 
\end{equation*}
and whenever $\mu$ is a locally absolutely continuous curve with tangent vector field $v$ and $\int_0^t{|\partial\phi|(\mu(r))|\mu'|(r)\mathrm{d}r} < +\infty$ for all $t>0$, the function $\phi\circ\mu$ belongs to $\mathrm{C}([0, +\infty))\cap \mathrm{W}_{\mathrm{loc}}^{1,1}((0, +\infty))$ and the chain rule \eqref{eq: chain rule} holds good
for $\mathcal{L}^1$-a.e. $t>0$, see Cor. 2.4.10, Lem. 10.1.3, Lem. 10.3.4, Rem. 3.1.7, Thm. 5.4.4 and ``chain rule'' in Sect. 10.1.2 in \cite{AGS08}. 

The same is true if $\phi$ is $\lambda$-convex along constant speed geodesics for a constant $\lambda\in\mathbb{R}$, see Def. 9.1.1 in \cite{AGS08}.
\end{example}

Our next Example \ref{ex: chain rule} illustrates the application of Proposition \ref{prop: 2.1} to a wide class of energy functionals including functionals that are neither displacement convex nor $\lambda$-convex along constant speed geodesics. 

\begin{example}[Second order diffusion equation with no-flux boundary condition]\label{ex: chain rule}
The energy functional $\phi: \mathcal{P}_p(\mathbb{R}^d)\to(-\infty, +\infty]$, 
\begin{equation*}
\phi(\mu):= \int_{\mathbb{R}^d\times\mathbb{R}^d}{[F(u(x)) + V(x)u(x)+\frac{1}{2}W(x-y)u(x)]u(y)\mathrm{d}x\mathrm{d}y} 
\end{equation*}
if $ \mu=u\mathcal{L}^d\ll\mathcal{L}^d\llcorner\Omega$ (i.e. $u\equiv0$ on $\mathbb{R}^d\setminus\Omega$) and $\phi(\mu):=+\infty$ else, is defined on $(\mathcal{P}_p(\mathbb{R}^d), \mathcal{W}_p)$, with 
\begin{enumerate}[label={(A\arabic*)}]
\item $\Omega$ being an open and bounded subset of $\mathbb{R}^d$ with $\mathrm{C}^1$-boundary $\partial\Omega$, \label{ass: A0}
\item $F:[0,+\infty)\to\mathbb{R}$ being continuous with $F(0)=0$, 
twice continuously differentiable in $(0,+\infty)$ with $F''>0$ 
and satisfying 
\begin{equation*}
\lim_{s\to\infty}\frac{F(s)}{s} \ = \ +\infty
\end{equation*}
and 
\begin{equation*}
\exists C_F > 0: \ \  F(r+s)\leq C_F(1+F(r)+F(s)) \ \ \text{ for all } r,s \geq 0, 
\end{equation*}
 \label{ass: A1}
\item $V:\mathbb{R}^d\to\mathbb{R}$ being locally Lipschitz continuous,
  \label{ass: A2}
\item $W: \mathbb{R}^d\to [0,+\infty)$ being convex (thereby locally Lipschitz continuous) and even. 
\label{ass: A3}
\end{enumerate}
It is obvious from \ref{ass: A0}-\ref{ass: A3} that $\phi$ is bounded from below and its sublevel sets are compact in $(\mathcal{P}_p(\mathbb{R}^d), \mathcal{W}_p)$  (see e.g. the proof of Lem. 3.3 in \cite{fleissner2021minimizing}, Dunford-Pettis theorem and Thm. 7.12 in \cite{Villani03}) and therefore the energy functional satisfies \ref{ass: phi1}-\ref{ass: phi4}. 

A typical approach to the strong subdifferential $\partial_s\phi$ of an energy functional consists in computing directional derivatives of the functional, cf. Sect. 10.4 in \cite{AGS08}, Sect. 4 in \cite{gianazza2009wasserstein} and Sect. 2.4 in \cite{matthes2009family}, based on the first variation calculus introduced in \cite{jordan1998variational}. Let $\mu=u\mathcal{L}^d\in\{\phi < +\infty\}$ and $\xi\in\mathrm{C}_{\mathrm{c}}^{\infty}(\mathbb{R}^d; \mathbb{R}^d)$ with associated flux $(\mathbb{X}_t)_{t\in\mathbb{R}}, \ \mathbb{X}_t: \mathbb{R}^d\to\mathbb{R}^d$, 
\begin{equation*}
\partial_t\mathbb{X}_t(x) \ = \ \xi(\mathbb{X}_t(x)), \quad \mathbb{X}_0(x) \ = \ x \quad \text{ for all } x\in\mathbb{R}^d, \ t\in\mathbb{R}. 
\end{equation*}
For every $t\in\mathbb{R}$, the map $\mathbb{X}_t$ is a diffeomorphism and $\mathbb{X}_t(\Omega) = \Omega$ by Thm. 1 in \cite{brezis1970characterization}. We note that
\begin{equation*}
(\mathbb{X}_{t})_{\#}\mu \ = \ \bar{u}_t\mathcal{L}^d, \quad\quad |\mathrm{det}(\mathrm{D}\mathbb{X}_t(x))|\cdot \bar{u}_t(\mathbb{X}_t(x)) \ = \ u(x)
\end{equation*}
by the change of variables formula ($\mathrm{D}\mathbb{X}_t$ denotes the differential of $\mathbb{X}_t$), $\mathrm{det}(\mathrm{D}\mathbb{X}_t(x)) > 0$ for $t$ in a neighbourhood around $0$ and $t\mapsto\mathrm{det}(\mathrm{D}\mathbb{X}_t(x))$ is differentiable at $t=0$ with derivative equal to 
$\nabla\cdot\xi(x)$. Using \ref{ass: A0}-\ref{ass: A3}, similar arguments as in the second part of the proof of Lem. 10.4.4 in \cite{AGS08} and the dominated convergence theorem, we obtain the differentiability of $t\mapsto\phi((\mathbb{X}_t)_{\#}\mu)$ at $t=0$ with derivative equal to 
\begin{equation*}
-\int_{\mathbb{R}^d}{L_F(u(x))\nabla\cdot\xi(x) \mathrm{d}x} \ + \ \int_{\mathbb{R}^d}{\langle\nabla V(x) + (\nabla W\ast\mu)(x), \xi(x)\rangle\mathrm{d}\mu(x)}, 
\end{equation*}
where
\begin{equation}\label{eq: L_F}
L_F(s):= \begin{cases} sF'(s) - F(s) & \text{ if } s\in(0,+\infty), \\ 0 & \text{ if } s=0, \end{cases}
\end{equation}
and $L_F(u)\in\mathrm{L}^1(\mathbb{R}^d)$. Moreover, if $\partial_s\phi(\mu)\neq\emptyset$ and $\zeta\in\partial_s\phi(\mu)$, then 
\begin{equation*}
\frac{\mathrm{d}}{\mathrm{d}t}\bigg\vert_{t=0} \phi((\mathbb{X}_t)_{\#}\mu) \ = \ \int_{\mathbb{R}^d}{\langle\zeta(x), \xi(x)\rangle\mathrm{d}\mu(x)}
\end{equation*}
by Definition \ref{def: subdifferentials} of the strong subdifferential, from which we can directly infer that $L_F(u)\in\mathrm{W}^{1,1}(\mathbb{R}^d)$ with 
\begin{equation}\label{eq: nabla L_F}
\nabla L_F(u) \ = \ u\cdot\Big(\zeta - \nabla V - (\nabla W\ast \mu)\Big)
\end{equation}
(since $\xi\in\mathrm{C}_{\mathrm{c}}^{\infty}(\mathbb{R}^d; \mathbb{R}^d)$ was chosen arbitrarily). Hence, assuming the strong subdifferential at $\mu$ is nonempty, it contains a unique vector field $\zeta\in\mathrm{L}^q(\mu; \mathbb{R}^d)$ and $L_F(u)$ belongs to $\mathrm{W}^{1,1}(\mathbb{R}^d)$ satisfying \eqref{eq: nabla L_F}. 

The next step is to prove that also the limiting subdifferential is characterized by \eqref{eq: nabla L_F}. Let $\mu=u\mathcal{L}^d\in\{\phi<+\infty\}$ with $\partial_l\phi(\mu)\neq\emptyset$. If $\zeta\in\partial_l\phi(\mu)$, then by Definition \ref{def: subdifferentials}, there exist $\mu_n=u_n\mathcal{L}^d\stackrel{\mathcal{W}_p}{\longrightarrow}\mu$ and $\zeta_n\in\partial_s\phi(\mu_n)$ such that $\sup_n\Big\{\phi(\mu_n), \int_{\mathbb{R}^d}{|\zeta_n(x)|^q\mathrm{d}\mu_n(x)}\Big\} < +\infty$ and $\zeta_n\mu_n$ converges to $\zeta\mu$ in the distributional sense. The same arguments that yielded the compactness of the sublevel sets in $(\mathcal{P}_p(\mathbb{R}^d), \mathcal{W}_p)$ now show the weak $\mathrm{L}^1$-convergence of $u_n$ to $u$. Furthermore, $(L_F(u_n))_{n\in\mathbb{N}}$ is bounded in $\mathrm{W}^{1,1}(\mathbb{R}^d)$ and $(\nabla L_F(u_n))_{n\in\mathbb{N}}$ is equiintegrable because
\begin{equation*}
0 \ \leq \ L_F(u_n) \ \leq \ C_F(1+2F(u_n)) - 2 \min_{r\geq0} F(r), \quad L_F(u_n)\equiv 0 \text{ on } \mathbb{R}^d\setminus \Omega, 
\end{equation*}
\eqref{eq: nabla L_F} holds good for $\zeta_n, u_n, \mu_n$, the functions $V, W$ are locally Lipschitz continuous, $\sup_n \Big\{\phi(\mu_n), \int_{\mathbb{R}^d}{|\zeta_n(x)|^q\mathrm{d}\mu_n(x)}\Big\} < +\infty$ and $(u_n)_{n\in\mathbb{N}}$ is equiintegrable. Using Rellich-Kondrachov theorem, the facts that $L_F$ is strictly increasing and $L_F(u_n)\equiv0$ on $\mathbb{R}^d\setminus\Omega$, Egorov theorem and Dunford-Pettis theorem, we infer that $L_F(u)\in\mathrm{W}^{1,1}(\mathbb{R}^d)$ and $\nabla L_F(u_n)\stackrel{\mathrm{L}^1}{\rightharpoonup}\nabla L_F(u)$, see e.g. the second part of the proof of Thm. 3.4 in \cite{fleissner2021minimizing}. Moreover, it is not difficult to see that $u_n\nabla V \stackrel{\mathrm{L}^1}{\rightharpoonup}u\nabla V$ and $u_n\cdot(\nabla W\ast\mu_n)\stackrel{\mathrm{L}^1}{\rightharpoonup}u\cdot(\nabla W\ast\mu)$. Altogether, we have proved 
\begin{proposition}\label{prop: limiting subdifferential}
Let $\phi$ be the energy functional defined in Example \ref{ex: chain rule}. If $\partial_l\phi(\mu)$ is nonempty for $\mu=u\mathcal{L}^d\in\{\phi < +\infty\}$, then $L_F(u)\in\mathrm{W}^{1,1}(\mathbb{R}^d)$, the limiting subdifferential at $\mu$ contains a unique vector field $\zeta\in\mathrm{L}^q(\mu; \mathbb{R}^d)$ and \eqref{eq: nabla L_F} holds good. 
\end{proposition}

Lastly, we are concerned with the validation of the chain rule \eqref{eq: chain rule}. 
\begin{proposition}\label{prop: chain rule}
Let $\phi$ be the energy functional defined in Example \ref{ex: chain rule} and let $\mu: [0, +\infty) \to \{\phi<+\infty\}\subset \mathcal{P}_p(\mathbb{R}^d)$ be a locally absolutely continuous curve in $(\mathcal{P}_p(\mathbb{R}^d), \mathcal{W}_p)$ with tangent vector field $v$. We assume \ref{ass: B1} or \ref{ass: B2} 
(it suffices to assume only one of the following two conditions): 
\begin{enumerate}[label={(B\arabic*)}]
\item The density functions associated with $\mu$ are locally uniformly bounded, i.e. $\mu_t=u_t\mathcal{L}^d$ and $\sup_{t\in[0, T]}\|u_t\|_{\infty}<+\infty$ for every $T>0$. \label{ass: B1}
\item $F$ satisfies
\begin{equation*}
s\mapsto s^dF(s^{-d}) \ \ \text{ is convex in } (0,+\infty). 
\end{equation*}\label{ass: B2}
\end{enumerate}
If $(|\partial^-\phi|\circ\mu)\cdot|\mu'| \in \mathrm{L}^1_{\mathrm{loc}}(0,+\infty)$, then  $\phi \circ  \mu \in \mathrm{C}([0, +\infty))\cap \mathrm{W}_{\mathrm{loc}}^{1,1}((0, +\infty))$ and for $\mathcal{L}^1$-a.e. $t>0$, the limiting subdifferential $\partial_l\phi(\mu(t))$ contains a unique element $D_l\phi(\mu(t))$ and \eqref{eq: chain rule} holds good. 
\end{proposition}
\begin{proof}
It follows from \eqref{eq: subdifferential slope inequality} and Proposition \ref{prop: limiting subdifferential} that whenever $|\partial^-\phi|(\mu(t)) < +\infty, \ \mu(t)=u(t)\mathcal{L}^d,$ then $\partial_l\phi(\mu(t))\neq\emptyset$ contains a unique element $D_l\phi(\mu(t))$ with $\|D_l\phi(\mu(t))\|_{L^q(\mu_t; \mathbb{R}^d)} \leq |\partial^-\phi|(\mu(t))$, and $L_F(u(t))$ belongs to $\mathrm{W}^{1,1}(\mathbb{R}^d)$ with 
\begin{equation}\label{eq: nabla L_F(u(t))}
\nabla L_F(u(t)) \ = \ u(t)\cdot\Big(D_l\phi(\mu(t)) - \nabla V - (\nabla W\ast \mu(t))\Big).
\end{equation}

Let $R>0$ s.t. $\Omega\subset\overline{\mathcal{B}(0; R)}:=\{x\in\mathbb{R}^d: \ |x|\leq R\}$. 
As the functional $\mathsf{W}: \mathcal{P}_p(\overline{\mathcal{B}(0; R)})\ni\nu\mapsto\frac{1}{2}\int{\int{W(x-y)\mathrm{d}\nu}\mathrm{d}\nu}$ is displacement convex according to Prop. 1.2 in \cite{mccann1997convexity}, Prop. 9.3.5 in \cite{AGS08}, and $W$ is locally Lipschitz continuous, we may apply the last part of the proof of Thm. 10.4.11, Thm. 10.3.11 and Prop. 10.3.18  from \cite{AGS08} to $\mathsf{W}$ and the curve $\mu$ to obtain $\mathsf{W}\circ  \mu \in \mathrm{C}([0, +\infty))\cap \mathrm{W}_{\mathrm{loc}}^{1,1}((0, +\infty))$ with 
\begin{equation*}
(\mathsf{W}\circ\mu)'(t) \ = \ \int_{\mathbb{R}^d}{\langle(\nabla W\ast\mu(t))(x), v_t(x)\rangle\mathrm{d}\mu_t(x)} 
\end{equation*}
for $\mathcal{L}^1$-a.e. $t>0$. Also the functional $\mathsf{V}: \mathcal{P}_p(\Omega)\ni\nu\mapsto\int{V(x)\mathrm{d}\nu}$ satisfies $\mathsf{V}\circ  \mu \in \mathrm{C}([0, +\infty))\cap \mathrm{W}_{\mathrm{loc}}^{1,1}((0, +\infty))$ with chain rule 
\begin{equation*}
(\mathsf{V}\circ\mu)'(t) \ = \ \int_{\mathbb{R}^d}{\langle\nabla V(x), v_t(x)\rangle\mathrm{d}\mu_t(x)} 
\end{equation*}
for $\mathcal{L}^1$-a.e. $t>0$. Indeed, the theory of Sobolev extensions and mollifiers (see e.g. Sects. 5.4, 5.3.1 and C.5 in \cite{evans2010partial}) provides us with $\mathrm{C}^\infty_{\mathrm{c}}(\mathbb{R}^d)$-functions $\xi_\epsilon$ that are uniformly bounded in $\mathrm{W}^{1,\infty}(\mathbb{R}^d)$ and approximate the locally Lipschitz continuous function $V$ and its weak gradient in $\Omega$, i.e. $\sup_\epsilon\|\xi_\epsilon\|_{\mathrm{W}^{1,\infty}(\mathbb{R}^d)}<+\infty$, 
\begin{equation*}
\xi_\epsilon\to V \quad \text{and} \quad \nabla\xi_\epsilon\to\nabla V \quad \mathcal{L}^d\text{-a.e. in } \Omega; 
\end{equation*}
testing the continuity equation \eqref{eq: tangent 2} for $\mu, v$ on $\mathrm{C}^\infty_{\mathrm{c}}((0, +\infty)\times\mathbb{R}^d)$-functions
\begin{equation*}
(t,x)\mapsto \eta(t)\xi_\epsilon(x), \quad \quad \eta\in\mathrm{C}^\infty_{\mathrm{c}}((0, +\infty)), 
\end{equation*}
letting $\epsilon\downarrow0$ and applying the dominated convergence theorem and \eqref{eq: tangent 1}, we obtain the above statement for $\mathsf{V}\circ\mu$. 

All that remains to be proved is $\mathsf{F}\circ\mu\in\mathrm{C}([0,+\infty))\cap\mathrm{W}^{1,1}_{\mathrm{loc}}((0,+\infty))$, 
\begin{equation}\label{eq: chain rule F}
(\mathsf{F}\circ\mu)'(t)  \ = \ \int_{\mathbb{R}^d}{\langle \nabla L_F(u(t))(x), v_t(x) \rangle\mathrm{d}x} \quad\quad \mathcal{L}^1\text{-a.e.}
\end{equation}
for $\mathsf{F}(\mu(\cdot)):=\int_{\mathbb{R}^d}{F(u(\cdot,x))\mathrm{d}x}$. 

Assuming \ref{ass: B1}, we define $F_\epsilon: [0, +\infty)\to\mathbb{R}$, 
\begin{equation*}
F_\epsilon(s):=\begin{cases}
F(s)-F(\epsilon)+F'(\epsilon)\epsilon &\text{ if } \epsilon < s, \\  F'(\epsilon)s &\text{ if } 0\le s \le \epsilon, \end{cases}
\end{equation*}
and $\mathsf{F}_\epsilon(\mu(\cdot)):=\int_{\mathbb{R}^d}{F_\epsilon(u(\cdot,x))\mathrm{d}x}$ for $\epsilon >0$. It is not difficult to see that $\mathsf{F}_\epsilon(\mu(t))\to\mathsf{F}(\mu(t))$ as $\epsilon\downarrow0$ for all $t\geq0$. Moreover, 
\begin{equation*}
L_{F_\epsilon}(s) = \begin{cases}
L_F(s)-L_F(\epsilon) &\text{ if } \epsilon < s, \\ 0 &\text{ if } 0\le s\le \epsilon.\end{cases}
\end{equation*}
An argumentation similar to that from the proof of Prop. 4.8 in \cite{ambrosio2006stability} yields $\mathsf{F}_\epsilon\circ\mu\in\mathrm{C}([0, +\infty))\cap \mathrm{W}_{\mathrm{loc}}^{1,1}((0, +\infty))$ with 
\begin{equation*}
(\mathsf{F}_\epsilon\circ\mu)'(r)  \ = \ \int_{\mathbb{R}^d}{\langle \nabla L_{F_\epsilon}(u(r))(x), v_r(x) \rangle\mathrm{d}x} \quad\quad \mathcal{L}^1\text{-a.e.,}
\end{equation*}
where \eqref{eq: subdifferential slope inequality}, \eqref{eq: nabla L_F(u(t))} and the assumption that  $(|\partial^-\phi|\circ\mu)\cdot|\mu'|$ belongs to $\mathrm{L}^1_{\mathrm{loc}}(0,+\infty)$ is used.   
Letting $\epsilon\downarrow0$ in
\begin{equation*}
\mathsf{F}_\epsilon(\mu(t)) - \mathsf{F}_\epsilon(\mu(0)) \ = \ \int_0^t{\int_{\mathbb{R}^d}{\langle \nabla L_{F_\epsilon}(u(r))(x), v_r(x) \rangle\mathrm{d}x}\mathrm{d}r}, 
\end{equation*}
we obtain $\mathsf{F}\circ\mu\in\mathrm{C}([0,+\infty))\cap\mathrm{W}^{1,1}_{\mathrm{loc}}((0,+\infty))$ and \eqref{eq: chain rule F}.

Assuming \ref{ass: B2}, the functional $\mathsf{F},$ 
\begin{equation*}
\mathsf{F}(\nu):=\int_{\mathbb{R}^d}{F(\rho(x))\mathrm{d}x} \quad\text{ for } \nu=\rho\mathcal{L}^d\in \mathcal{P}_p(\overline{\mathcal{B}(0; R)}),
\end{equation*}
is displacement convex in $(\mathcal{P}_p(\overline{\mathcal{B}(0; R)}), \mathcal{W}_p)$ according to Thm. 2.2 in \cite{mccann1997convexity}, Prop. 9.3.9 in \cite{AGS08} ($s\mapsto s^dF(s^{-d})$ is nonincreasing in $(0,+\infty)$ as $L_F\ge0$), and we infer $\mathsf{F}\circ\mu\in\mathrm{C}([0,+\infty))\cap\mathrm{W}^{1,1}_{\mathrm{loc}}((0,+\infty))$ and \eqref{eq: chain rule F} from Thms. 10.4.6, 10.3.11, Prop. 10.3.18  in \cite{AGS08}, \eqref{eq: subdifferential slope inequality}, \eqref{eq: nabla L_F(u(t))} and from $(|\partial^-\phi|\circ\mu)\cdot|\mu'|\in\mathrm{L}^1_{\mathrm{loc}}(0,+\infty)$. 

The proof of Proposition \ref{prop: chain rule} is complete. 
\end{proof}

So we have all the ingredients for an application of Proposition \ref{prop: 2.1} to our example; the corresponding differential equation \eqref{eq: differential inclusion} reads as 
\begin{equation}\label{eq: diffusion equation weak form}
\int_0^{+\infty}{\int_{\Omega}{u\Big(\partial_t\xi - \Big\langle j_q\Big(\frac{\nabla L_F(u)}{u} + \nabla V + (\nabla W) \ast u\Big), \nabla_x\xi\Big\rangle\Big) \mathrm{d}x}\mathrm{d}t} \ = \ 0
\end{equation}
for all $\xi\in\mathrm{C}^\infty_{\mathrm{c}}((0,+\infty)\times\mathbb{R}^d)$, which can be easily identified as a weak reformulation of \eqref{eq: diffusion equation intro}
\begin{equation*}
\partial_t u - \nabla\cdot\Big(uj_q\Big(\nabla F'(u) + \nabla V + (\nabla W)\ast u\Big) \Big) \ = \ 0\quad \text{ in } (0, +\infty)\times\Omega
\end{equation*}
for $u: [0, +\infty)\times\Omega\to[0, +\infty)$. Any solution $u$ thereto satisfies, in addition, a weak form of the no-flux boundary condition 
\begin{equation*}
uj_q\Big(\nabla F'(u)+ \nabla V + (\nabla W) \ast u\Big)\cdot {\sf n} \ = \ 0 \quad \text{ on } (0, +\infty)\times\partial\Omega,
\end{equation*}
because \eqref{eq: diffusion equation weak form} does not only allow test functions $\xi\in\mathrm{C}^\infty_{\mathrm{c}}((0,+\infty)\times\Omega)$ but is tested on  all $\xi\in\mathrm{C}^\infty_{\mathrm{c}}((0,+\infty)\times\mathbb{R}^d)$. We point to the well-known thermodynamic interpretation of the function $L_F(u(t, \cdot))$ as pressure associated with the density $u(t, \cdot)$, cf. Rem. 5.18 in \cite{Villani03}. 

Please note that even if $F$ satisfies \ref{ass: B2}, the energy functional $\phi$ from Example \ref{ex: chain rule} may \textit{not} be displacement convex / $(\lambda)$-convex along constant speed geodesics in $(\mathcal{P}_p(\mathbb{R}^d), \mathcal{W}_p)$ because we do not assume any convexity for $V$ nor for $\Omega$. 

Example \ref{ex: chain rule} is revisited in Sections \ref{sec: 3.2} - \ref{sec: 3.3} on the Kernel-Density-Estimator Minimizing Movement Scheme. 
\end{example}

We end this section with two remarks regarding relaxations of the assumptions on $\phi$. 

\begin{remark}[Relaxation of \ref{ass: phi1}, \ref{ass: phi4}]\label{rem: relaxation}
We assume \ref{ass: phi1} only for the sake of a clear presentation with a straightforward notation. The above theory of gradient flows in $(\mathcal{P}_p(\mathbb{R}^d), \mathcal{W}_p)$ can be generalized to the case that $\phi$ does not satisfy \ref{ass: phi1}. 
The subdifferential calculus from Definition \ref{def: subdifferentials} has to be adapted for this case, see Sect. 10.3 and Def. 10.3.1 in \cite{AGS08}, and the chain rule \eqref{eq: chain rule} has to be reformulated correspondingly; Prop. 10.4.2, Rem. 10.4.3 and Thm. 10.4.11 in \cite{AGS08} provide a characterization of subdifferentials for $\phi(\mu):=\int{V\mathrm{d}\mu}+\int{\int{W}\mathrm{d}(\mu\times\mu)}$. 

A natural setting for the Minimizing Movement scheme associated with \eqref{eq: MM p} includes coercivity and compactness properties like \ref{ass: phi2}, \ref{ass: phi3}, \ref{ass: phi4} so that the existence of both discrete solutions for $\tau < \big(\frac{1}{pB}\big)^{1/(p-1)}$ and converging subsequences thereof is guaranteed. The assumption \ref{ass: phi4} can be relaxed by using an auxiliary topology besides the metric topology, see Sects. 2 and 3 in \cite{AGS08}. A suitable auxiliary topology for $\mathcal{P}_p(\mathbb{R}^d)$ is the one induced by weak convergence $\nu_n\rightharpoonup\nu$  defined as
\begin{equation*}
\lim_{n\to+\infty}\int_{\mathbb{R}^d}{f(x)\mathrm{d}\nu_n(x)} \ = \ \int_{\mathbb{R}^d}{f(x)\mathrm{d}\nu(x)} \quad \quad \text{ for all } f\in\mathrm{C}_{b}(\mathbb{R}^d). 
\end{equation*}
Please note that 
\begin{equation*}
\lim_{n\to+\infty}\mathcal{W}_p(\nu_n, \nu) \ = \ 0 \quad \Leftrightarrow \quad \nu_n\rightharpoonup\nu, \ \  \lim_{R\to+\infty}\limsup_{n\to+\infty}\int_{|x|\ge R}{|x|^p\mathrm{d}\nu_n(x)} \ = \ 0 
\end{equation*}
(see e.g. Thm. 7.12 in \cite{Villani03}), and by Prokhorov's Theorem, $\mathcal{W}_p$-bounded sets are relatively compact w.r.t. weak convergence. Assuming \ref{ass: phi2} and 
\begin{equation}\label{eq: weak lsc}
\phi(\nu) \ \leq \ \liminf_{n\to+\infty}\phi(\nu_n) \quad\text{whenever}\quad v_n\rightharpoonup\nu, \ \sup_n\mathcal{W}_p(\nu_n, \nu) < +\infty
\end{equation}
instead of \ref{ass: phi3}, \ref{ass: phi4}, it is still possible to prove \eqref{eq: subdifferential slope inequality} and to establish statements very similar to those of Proposition \ref{prop: 2.1} following the same argumentation. The only differences in the statements are that, in this case, $\mathcal{W}_p(\mu_n, \mu) \to 0$ is replaced with $\sup_n\mathcal{W}_p(\mu_n, \mu) < +\infty, \ \mu_n\rightharpoonup\mu$ in the definitions of the relaxed slope $|\partial^-\phi|$, the limiting subdifferential $\partial_l\phi(\mu)$ (cf. Definitions \ref{def: slopes} and \ref{def: subdifferentials}) and the sequence of initial data, and we obtain convergence of discrete solutions \eqref{eq: piecewise constant interpolation} of the Minimizing Movement scheme \eqref{eq: general MM scheme}, \eqref{eq: MM p} w.r.t. the weak topology. 
\end{remark}

\begin{remark}[Absence of chain rule]\label{rem: absence of chain rule}
Even if the chain rule \eqref{eq: chain rule} cannot be validated, the Minimizing Movement scheme associated with \eqref{eq: MM p}, $p=2$, yields the existence of a curve $\mu$ satisfying both the energy inequality \eqref{eq: energy inequality} and the continuity equation  
\begin{equation*}
\partial\mu_t - \nabla\cdot (D_l\phi(\mu(t))\mu_t) \ = \ 0
\end{equation*}
(where $-D_l\phi(\mu(\cdot))$ is not necessarily tangent to $\mu$ in contrast to \eqref{eq: differential inclusion}), cf. Thm. 2.3.3 and Thm. 11.1.6 in \cite{AGS08}.
\end{remark}

\subsection{$\Gamma$-Convergence for Gradient Flows: A Minimizing Movement Approach}\label{sec: 2.2}
As explicated in Sect. 1 in \cite{fleissner2016gamma}, we cannot expect stability of continuous-time gradient flows under $\Gamma$-convergence of the energy functionals; in general, a corresponding sequence of gradient flow solutions is \textit{not} related to the gradient flow associated with the $\Gamma$-limit functional. Things are considerably better on the level of discrete-time steepest descents. 

Let $p\in(1, +\infty)$ and $\phi_n, \phi: \mathcal{P}_p(\mathbb{R}^d)\to(-\infty, +\infty] \ (n\in\mathbb{N})$ be energy functionals s.t. $\phi_n\stackrel{\Gamma}{\to}\phi$ in $(\mathcal{P}_p(\mathbb{R}^d), \mathcal{W}_p)$, i.e. 
\begin{equation}\label{eq: Gamma-liminf}
\phi(\mu) \ \leq \ \liminf_{n\to+\infty}{\phi_{n}(\mu_n)} \quad \text{ whenever }  \quad \lim_{n\to+\infty}\mathcal{W}_p(\mu_n, \mu)  =  0
\end{equation}
and for all  $\mu\in\mathcal{P}_p(\mathbb{R}^d)$ there exists a so-called \textit{recovery sequence} 
\begin{equation}\label{eq: recovery sequence}
\exists \ \text{ a sequence } \bar{\mu}_n, \lim_{n\to+\infty}\mathcal{W}_p(\bar{\mu}_n, \mu)  =  0:  \   \phi(\mu)  =  \lim_{n\to+\infty}{\phi_{n}(\bar{\mu}_n)}. 
\end{equation}

Every time step size $\tau>0$ is assigned a parameter $n=n(\tau)\in\mathbb{N}$ according to a rule that is specified below; we are concerned with the limiting behaviour of discrete solutions \eqref{eq: piecewise constant interpolation} to the \textit{relaxed Minimizing Movement scheme}
\begin{equation}\label{eq: relaxed MM 1}
\Phi(\tau, \mu_\tau^{m-1}, \mu_\tau^m) \ \leq \ \inf_{\nu\in\mathcal{P}_p(\mathbb{R}^d)}{\Phi(\tau, \mu_\tau^{m-1}, \nu)} \ + \ \gamma_\tau^{(m)} \quad \quad (m\in\mathbb{N})
\end{equation}
associated with 
\begin{equation}\label{eq: relaxed MM 2}
\Phi(\tau, \mu, \nu) \ := \ \phi_{n(\tau)}(\nu) + \frac{1}{p\tau^{p-1}}\mathcal{W}_p(\nu, \mu)^p, \quad\quad\quad \gamma_\tau^{(m)} > 0 
\end{equation}
as the time step sizes $\tau\downarrow0$ and parameters $n(\tau)\uparrow+\infty$ simultaneously. 

By comparison, the analysis of the limiting behaviour of the continuous-time gradient flows associated with $\phi_n, \ n\in\mathbb{N},$ amounts to first letting the time step sizes $\tau\downarrow0$ for fixed $n\in\mathbb{N}$ and only then letting the parameters $n\uparrow+\infty$. In view of the wide range of possible $\Gamma$-perturbations $\phi_n$ of $\phi$, such an approach to the study of stability of steepest descents is too restrictive, cf. Sect. 1 in \cite{fleissner2016gamma}. We need to take the steepest descent motion w.r.t. $\phi_n,\ n\in\mathbb{N},$ as a whole and compensate for the lack of control over the slopes of $\phi_n$ by bringing the interplay between parameters and time step sizes into focus. 

There are, however, particular cases $\phi_n\stackrel{\Gamma}{\to}\phi$ in which the continuous-time steepest descents w.r.t $\phi_n, \ n\in\mathbb{N},$ converge to the continuous-time gradient flow associated with $\phi$; typically related with a $\Gamma$-liminf inequality of the corresponding relaxed slopes $|\partial^-\phi|, \ |\partial^-\phi_n|, \ n\in\mathbb{N}$, these special cases are also covered by our general theory, see Remark \ref{rem: main ass}\ref{rem: main ass iii}. 

This section serves as a detailed introduction to our stability theory for steepest descents under the occurrence of $\Gamma$-perturbations of the energy functional; we establish a direct connection between the discrete-time steepest descents w.r.t. $\phi_n, \ n\in\mathbb{N},$ and the continuous-time gradient flow associated with the $\Gamma$-limit functional $\phi$. This link is formed through an appropriate correlation between time step sizes $\tau\downarrow0$ and parameters $n=n(\tau)\uparrow+\infty$ that solely depends on the velocity of $\Gamma$-convergence  $\phi_n\stackrel{\Gamma}{\to}\phi$. Further, we are able to relax the minimum problems in the corresponding scheme, see \eqref{eq: relaxed MM 1}, and prove the stability statements even for \textit{approximate} discrete-time steepest descents.

\begin{assumption}\label{ass: 1}
We assume 
\begin{enumerate}[label=($\phi_n$\arabic*)]
\item there exist $A, B > 0, \ \mu_\star\in\mathcal{P}_p(\mathbb{R}^d)$ s.t. 
$\phi_n(\cdot) \ \geq \ -A-B\mathcal{W}_p(\cdot, \mu_\star)^p$ for all $n\in\mathbb{N}$, \label{ass: phin1}
\item the combined compactness property 
\begin{equation*}
\sup_{n}\{\phi_n(\mu_n), \mathcal{W}_p(\mu_n, \mu_\star)\} < +\infty \quad \Rightarrow \quad \exists n_k\uparrow+\infty, \mu: \ \mathcal{W}_p(\mu_{n_k}, \mu) \to 0. 
\end{equation*}
\label{ass: phin2}
\end{enumerate} 
and that $\phi$ satisfies \ref{ass: phi1}. 
\end{assumption}

The natural conditions \ref{ass: phin1} and \ref{ass: phin2} guarantee the existence of discrete solutions to \eqref{eq: relaxed MM 1}, \eqref{eq: relaxed MM 2} for $\tau< \big(\frac{1}{pB}\big)^{1/(p-1)}$ and the existence of converging subsequences thereof (cf. Sects. 3.1 and 3.2 in \cite{fleissner2016gamma}).  Since we allow approximate minimizers in \eqref{eq: relaxed MM 1} (with small error term $\gamma_\tau^{(m)} > 0$), we do not need to impose any lower semicontinuity or compactness condition on the single functionals $\phi_n$. 

The study of a relaxed Minimizing Movement scheme along $(\phi_n)_{n}$ as approximation scheme for the gradient flow associated with $\phi$ entails a condition connecting the energy driven steepest descent motion \eqref{eq: relaxed MM 1}, \eqref{eq: relaxed MM 2} with the slope of $\phi$. 

Let $q\in(1,+\infty)$ be the conjugate exponent of $p$ and $\mathcal{Y}^{(p)}_\tau\phi_{n(\tau)}$ denote the $p$-Moreau-Yosida approximation of the functional $\phi_{n(\tau)}$, 
\begin{equation}\label{eq: Yosida}
\mathcal{Y}^{(p)}_\tau\phi_{n(\tau)}(\mu) := \inf_{\nu\in\mathcal{P}_p(\mathbb{R}^d)}\Big\{\phi_{n(\tau)}(\nu) + \frac{1}{p\tau^{p-1}}\mathcal{W}_p(\nu, \mu)^p\Big\}. 
\end{equation}
We assume the following: 
\begin{assumption}\label{ass: 2}
Whenever $\nu_\tau, \nu\in\mathcal{P}_p(\mathbb{R}^d)$ satisfy
\begin{equation*}
\lim_{\tau\to0}\mathcal{W}_p(\nu_\tau, \nu)  =  0, \quad \sup_\tau\phi_{n(\tau)}(\nu_\tau)  < +\infty, 
\end{equation*}
then
\begin{equation}\label{eq: main condition}
\liminf_{\tau\to0}\frac{\phi_{n(\tau)}(\nu_\tau) - \mathcal{Y}^{(p)}_\tau\phi_{n(\tau)}(\nu_\tau)}{\tau} \ \geq \ \frac{1}{q}|\partial^-\phi|^q(\nu). 
\end{equation}
\end{assumption}

Condition \eqref{eq: main condition} plays a central role in our Minimizing Movement approach to $\Gamma$-convergence for gradient flows; as stated in Theorem \ref{thm: MM approach I}, we manage to link the relaxed Minimizing Movement scheme \eqref{eq: relaxed MM 1}, \eqref{eq: relaxed MM 2} along $(\phi_n)_n$ to the gradient flow equation \eqref{eq: differential inclusion} driven by $\phi$ by means of Assumption \ref{ass: 2} and chain rule \eqref{eq: chain rule} from Proposition \ref{prop: 2.1}; please note that \eqref{eq: differential inclusion} represents a weak reformulation of the diffusion equation \eqref{eq: diffusion equation intro} according to Section \ref{sec: 2.1}.

It is noteworthy that our main Assumption \ref{ass: 2} arises quite naturally from the context of gradient flows, which will be clarified by Remark \ref{rem: main ass}. Moreover, there \textit{always} exist appropriate correlations $\tau\mapsto n(\tau)\in\mathbb{N}$ so that $n(\tau)\uparrow +\infty$ as $\tau\downarrow0$ and Assumption \ref{ass: 2} holds true, see Theorem \ref{thm: MM approach II}. 

The following convergence statement constitutes the first part of our theory. 

\begin{theorem}[$\Gamma$-convergence for discrete-time steepest descents I]\label{thm: MM approach I}

We suppose that $\phi, \phi_n: \mathcal{P}_p(\mathbb{R}^d)\to(-\infty, +\infty], \ n\in\mathbb{N},$ satisfy Assumption \ref{ass: 1} and $\phi_n\stackrel{\Gamma}{\to}\phi$ in $(\mathcal{P}_p(\mathbb{R}^d), \mathcal{W}_p)$. 

Let $(\tau_k)_{k\in\mathbb{N}}, \ \tau_k\downarrow0,$ be a sequence of time step sizes, $\bar{\mu}_{\tau_k}$ be discrete solutions \eqref{eq: piecewise constant interpolation} to the relaxed Minimizing Movement scheme \eqref{eq: relaxed MM 1}, \eqref{eq: relaxed MM 2}, associated with parameters $(n(\tau_k))_k, \ n(\tau_k)\in\mathbb{N}, \ n(\tau_k)\uparrow +\infty$, and error terms $\gamma_{\tau_k}^{(m)} = \gamma_{\tau_k}$, $\gamma_{\tau_k}>0$ such that 
\begin{equation}\label{eq: error term}
\lim_{k\to+\infty}\frac{{\gamma_{\tau_k}}}{\tau_k} \ = \ 0. 
\end{equation}
Let $\bar{\mu}_{\tau_k}(0)$ be a recovery sequence \eqref{eq: recovery sequence} for $(\phi_{n(\tau_k)})_k, \ \phi$ and some initial datum $\mu^0\in\{\phi<+\infty\}$. 

\begin{enumerate}[label=(\roman*)]
\item 
There exist a subsequence of time step sizes $(\tau_{k_l})_{l\in\mathbb{N}}, \ \tau_{k_l}\downarrow0,$ and a limit curve $\mu: [0, +\infty)\to\mathcal{P}_p(\mathbb{R}^d), \ \mu(t)=u(t, \cdot)\mathcal{L}^d$, such that
\begin{equation*}
\lim_{l\to+\infty}\mathcal{W}_p(\bar{\mu}_{\tau_{k_l}}(t), \mu(t)) \ = \ 0 \quad \quad \text{ for all } t\ge0. 
\end{equation*}
The curve $\mu$ is locally absolutely continuous in $(\mathcal{P}_p(\mathbb{R}^d), \mathcal{W}_p)$. \label{itm: 1}
\item \label{itm: 2} 
If Assumption \ref{ass: 2} holds true for the choice $(n(\tau_k))_k$, then $\mu$ 
satisfies the energy inequality \eqref{eq: energy inequality} for all $t\ge0$. 

Assuming, in addition, that $\phi\circ\mu$ belongs to $\mathrm{C}([0, +\infty))\cap\mathrm{W}_{1,1}^{\mathrm{loc}}((0, +\infty))$ satisfying $\partial_l\phi(\mu(t))=\{D_l\phi(\mu(t))\}$ and \eqref{eq: chain rule}
for $\mathcal{L}^1$-a.e. $t>0$, the curve $\mu$ and its tangent vector field solve the differential equation \eqref{eq: differential inclusion} and \eqref{eq: subdifferential slope equality} for $\mathcal{L}^1$-a.e. $t>0$ and the energy dissipation equality \eqref{eq: ede}
for all $0\leq s\leq t < +\infty$. In this case, 
\begin{equation}\label{eq: recovery sequence for all t}
\lim_{l\to+\infty}\phi_{n(\tau_{k_l})}(\bar{\mu}_{\tau_{k_l}}(t)) \ = \ \phi(\mu(t)) \quad \text{ for all } t\ge 0. 
\end{equation}   
\end{enumerate}
\end{theorem}

\begin{proof}
It is not difficult to extend Thm. 3.4 in \cite{fleissner2016gamma} to the case $p\in(1,+\infty)$. Hence, we refer thereto for statement \ref{itm: 1} and the proof of the energy inequality \eqref{eq: energy inequality} under Assumption \ref{ass: 2}, remarking that 
\begin{eqnarray*}
\phi_{n(\tau_k)}(\bar{\mu}_{\tau_k}(0)) - \phi_{n(\tau_k)}(\bar{\mu}_{\tau_k}(t)) \ = \ \sum_{m=1}^{N_{\tau_k}}{\Big[\phi_{n(\tau_k)}(\bar{\mu}_{\tau_k}^{(m-1)})-\phi_{n(\tau_k)}(\bar{\mu}_{\tau_k}^m)\Big]} \\
\ge \ \sum_{m=1}^{N_{\tau_k}}{\Big[\phi_{n(\tau_k)}(\bar{\mu}_{\tau_k}^{(m-1)})-\mathcal{Y}^{(p)}_{\tau_k}\phi_{n(\tau_k)}(\bar{\mu}_{\tau_k}^{(m-1)}) - \gamma_{\tau_k} + \frac{1}{p\tau_k^{p-1}}\mathcal{W}_p(\bar{\mu}^m_{\tau_k}, \bar{\mu}^{(m-1)}_{\tau_k})^p\Big]} \\
\ge \ \int_0^{t-\tau_k}{\frac{\phi_{n(\tau_k)}(\bar{\mu}_{\tau_k}(r))-\mathcal{Y}^{(p)}_{\tau_k}\phi_{n(\tau_k)}(\bar{\mu}_{\tau_k}(r))}{\tau_k}\mathrm{d}r}  +  \frac{1}{p}\int_0^t{|\bar{\mu}_{\tau_k}'|^p(r)\mathrm{d}r}  -  (N_{\tau_k}\tau_k)\cdot\frac{\gamma_{\tau_k}}{\tau_k}
\end{eqnarray*}
for $t>0$, $N_{\tau_k}\in\mathbb{N}$ s.t. $t\in((N_{\tau_k}-1)\tau_k, N_{\tau_k}\tau_k]$, and $\bar{\mu}_{\tau_k}^m := \bar{\mu}_{\tau_k}(m\tau_k)$, 
\begin{equation*}
|\bar{\mu}_{\tau_k}'|^p(r) \ := \ \frac{\mathcal{W}_p(\bar{\mu}_{\tau_k}^m, \bar{\mu}_{\tau_k}^{m-1})^p}{\tau_k^p} \quad \text{if } r\in((m-1)\tau_k, m\tau_k]; 
\end{equation*}
the significance of Assumption \ref{ass: 2} and condition \eqref{eq: error term} on the error term becomes apparent as we let the time step sizes $\tau_k\downarrow0$ (up to a subsequence) in these energy inequalities for the discrete solutions $\bar{\mu}_{\tau_k}$, see the proof of Thm. 3.4 in \cite{fleissner2016gamma} for more details. 

It follows from \ref{ass: phin1}, \ref{ass: phin2} and the $\Gamma$-convergence of $\phi_n$ to $\phi$ w.r.t. $ \mathcal{W}_p$ that $\phi$ satisfies \ref{ass: phi2}, \ref{ass: phi3}, \ref{ass: phi4}. We infer from Proposition \ref{prop: 2.1} that under the additional assumptions on $\phi\circ\mu$ made in the second part of statement \ref{itm: 2}, the limit curve $\mu$ is a solution to \eqref{eq: differential inclusion}, \eqref{eq: subdifferential slope equality} and \eqref{eq: ede}. Note that for our purposes, it is sufficient to consider $\partial_l\phi$ at points $\nu\in\mu([0, +\infty))\subset\{\phi<+\infty\}$. Finally, \eqref{eq: recovery sequence for all t} follows from the facts that $\lim_{k\to+\infty}\phi_{n(\tau_k)}(\bar{\mu}_{\tau_{k}}(0)) = \phi(\mu(0))$ and 
\begin{eqnarray*}
\phi(\mu(0)) - \phi(\mu(t)) &\geq& \limsup_{l\to+\infty}\Big[\phi_{n(\tau_{k_l})}(\bar{\mu}_{\tau_{k_l}}(0)) - \phi_{n(\tau_{k_l})}(\bar{\mu}_{\tau_{k_l}}(t))\Big] \\
&\ge& \liminf_{l\to+\infty}\Big[\phi_{n(\tau_{k_l})}(\bar{\mu}_{\tau_{k_l}}(0)) - \phi_{n(\tau_{k_l})}(\bar{\mu}_{\tau_{k_l}}(t))\Big] \\
&\ge&  \frac{1}{q} \int^{t}_{0}{|\partial^-\phi|^q(\mu(r))\mathrm{d}r} \ + \ \frac{1}{p} \int^{t}_{0}{|\mu'|^p(r) \mathrm{d}r} \\
&=& \phi(\mu(0)) - \phi(\mu(t))
\end{eqnarray*}
by \eqref{eq: ede} and the proof of \eqref{eq: energy inequality} in \cite{fleissner2016gamma}.
\end{proof}

The additional assumption from the second part of \ref{itm: 2} that $\phi\circ\mu$ belongs to $\mathrm{C}([0, +\infty))\cap\mathrm{W}_{1,1}^{\mathrm{loc}}((0, +\infty))$ satisfying $\partial_l\phi(\mu(t))=\{D_l\phi(\mu(t))\}$ and the chain rule \eqref{eq: chain rule} $\mathcal{L}^1$-a.e. typically arises out of the theory of gradient flows in $(\mathcal{P}_p(\mathbb{R}^d), \mathcal{W}_p)$, cf. Section \ref{sec: 2.1}, Proposition \ref{prop: 2.1}, Examples \ref{ex: convex case} and \ref{ex: chain rule}, Proposition \ref{prop: chain rule}. 

Before giving more information on our main Assumption \ref{ass: 2}, we consider Assumption \ref{ass: 1} and the error terms $\gamma_\tau^{(m)}$.
\begin{remark}[Relaxation of \eqref{eq: recovery sequence}, \ref{ass: phin2}, \ref{ass: phi1}]\label{rem: relaxation phin}
We present variations on Theorem \ref{thm: MM approach I} assuming \ref{ass: phi3} and that for small $\tau>0$ and all $\mu\in\mathcal{P}_p(\mathbb{R}^d)$ there exists a solution to the minimum problem
\begin{equation*}
\min_{\bar{\mu}\in\mathcal{P}_p(\mathbb{R}^d)}\Big\{\phi(\bar{\mu}) + \frac{1}{p\tau^{p-1}}\mathcal{W}_p(\bar{\mu}, \mu)^p\Big\}, 
\end{equation*}
so that \eqref{eq: subdifferential slope inequality} and 
the second part of Proposition \ref{prop: 2.1} are still applicable. 
\begin{enumerate}[label=(\roman*)]
\item A ``partial $\Gamma$-convergence'' of $\phi_n$ to $\phi$ is sufficient to obtain statements \ref{itm: 1} and \ref{itm: 2} from Theorem \ref{thm: MM approach I}: we need the existence of a recovery sequence \eqref{eq: recovery sequence} solely for the initial datum $\mu^0$, cf. Sect. 3.1, Thm. 3.4 in \cite{fleissner2016gamma}.
\item Theorem \ref{thm: MM approach I} still holds true (with weak instead of $\mathcal{W}_p$-convergence of the discrete solutions) if we use the topology induced by weak convergence as an auxiliary topology omitting \ref{ass: phin2} and correspondingly adapting \eqref{eq: Gamma-liminf}, \eqref{eq: recovery sequence}, Definitions \ref{def: slopes} and \ref{def: subdifferentials} of slopes and subdifferentials, the definition of the recovery sequence of initial data and Assumption \ref{ass: 2}, cf. Sects. 2, 3.1 and 3.2 in \cite{fleissner2016gamma} and Remark \ref{rem: relaxation}. 
\end{enumerate}
Condition \ref{ass: phi1} can be omitted, see Remark \ref{rem: relaxation}. 
\end{remark}

\begin{remark}[The error terms]\label{rem: error term}
The error order $o(\tau)$ in \eqref{eq: error term} is optimal, see Ex. 4.5 in \cite{fleissner2016gamma}. Further, the proof of Theorem \ref{thm: MM approach I} shows that we can extend our theory to a non-uniform distribution of the error terms $(\gamma_{\tau_k}^{(m)})_{m\in\mathbb{N}}$ replacing condition \eqref{eq: error term}, $\gamma_{\tau_k}^{(m)}=\gamma_{\tau_k}>0$, by the general condition $\gamma_{\tau_k}^{(m)} > 0$, 
\begin{equation}\label{eq: error term non-uniform}
\lim_{k\to+\infty} \sum_{m=1}^{N_{\tau_k}}{\gamma_{\tau_k}^{(m)}} \ = \ 0 \quad \text{for every } N_{\tau_k}\in\mathbb{N} \text{ s.t. } (N_{\tau_k}\tau_k)_{k\in\mathbb{N}} \text{ is bounded,}
\end{equation}
cf. Sect. 3.3 in \cite{fleissner2016gamma}. 
\end{remark}

The role of Assumption \ref{ass: 2} manifests itself in the proof of Theorem \ref{thm: MM approach I}. The next remark deals with special cases leading to a better understanding of Assumption \ref{ass: 2} as a natural condition. 

\begin{remark}[Comments on Assumption \ref{ass: 2}]\label{rem: main ass}

\begin{enumerate}[label=(\roman*)]
\item Assumption \ref{ass: 2} holds true for the case $\phi_n\equiv\phi$:  if $\phi: \mathcal{P}_p(\mathbb{R}^d)\to(-\infty, +\infty]$ satisfies \ref{ass: phi2}, \ref{ass: phi3}, \ref{ass: phi4}, 
then we have 
\begin{equation*}
\liminf_{\tau\to0}\frac{\phi(\nu_\tau) - \mathcal{Y}^{(p)}_\tau\phi(\nu_\tau)}{\tau} \ \geq \ \frac{1}{q}|\partial^-\phi|^q(\nu)
\end{equation*}
whenever
$\lim_{\tau\to0}\mathcal{W}_p(\nu_\tau, \nu)  =  0, \ \sup_\tau\phi(\nu_\tau)  < +\infty$, cf. Prop. 4.1 in \cite{fleissner2016gamma}; in consequence of Theorem \ref{thm: MM approach I}, we may allow a \textit{relaxed} form of minimization in the classical Minimizing Movement scheme \eqref{eq: general MM scheme}, \eqref{eq: MM p} associated with a single functional still obtaining gradient flow solutions w.r.t. $\phi$.  \label{rem: main ass i}

\item If $\phi_n\stackrel{\Gamma}{\to}\phi$ in $(\mathcal{P}_p(\mathbb{R}^d), \mathcal{W}_p)$, $\phi_n, \phi$ are lower semicontinuous and displacement convex, satisfy \ref{ass: phin1} and for small $\tau>0$ and all $\mu\in\mathcal{P}_p(\mathbb{R}^d)$ and $n\in\mathbb{N},$ there exists a solution to the minimum problem in the definition \eqref{eq: Yosida} of $\mathcal{Y}_\tau^{(p)}\phi_{n}(\mu)$, then Assumption \ref{ass: 2} is satisfied for \textit{every} choice $n=n(\tau)\uparrow+\infty$ (as $\tau\downarrow0$), 
cf. second part of Sect. 5, (5.4) and Prop. 5.2 in \cite{fleissner2016gamma}. \label{rem: main ass ii}

\item The well-known Serfaty-Sandier approach \cite{sandier2004gamma, serfaty2011gamma} offers the underlying structure of special cases $\phi_n\stackrel{\Gamma}{\to}\phi$ in which continuous-time gradient flow solutions w.r.t. $\phi_n$ converge to the continuous-time gradient flow associated with $\phi$; the Serfaty-Sandier theory relies on the assumption that the corresponding slopes satisfy a $\Gamma$-liminf inequality (often referred to as `Serfaty-Sandier condition'). In these cases, our Assumption \ref{ass: 2} is satisfied for \textit{every} choice $n=n(\tau)\uparrow+\infty$ (as $\tau\downarrow0$), cf. first part of Sect. 5 and Prop. 5.1 in \cite{fleissner2016gamma}. \label{rem: main ass iii}
\end{enumerate}

All the statements from \ref{rem: main ass i}, \ref{rem: main ass ii}, \ref{rem: main ass iii} can be extended to the case that the topology induced by weak convergence is included as auxiliary topology, cf. Remarks \ref{rem: relaxation phin} and \ref{rem: relaxation}, Sects. 4 and 5 in \cite{fleissner2016gamma}. 
\end{remark}

The next part of our theory treats the validation of Assumption \ref{ass: 2}. 
\begin{theorem}[$\Gamma$-convergence for discrete-time steepest descents II]\label{thm: MM approach II}
Let the sequence of functionals $\phi_n: \mathcal{P}_p(\mathbb{R}^d)\to(-\infty, +\infty], \ n\in\mathbb{N},$ satisfy \ref{ass: phin1}, \ref{ass: phin2}, and $\Gamma$-converge to $\phi: \mathcal{P}_p(\mathbb{R}^d)\to(-\infty, +\infty]$ in $(\mathcal{P}_p(\mathbb{R}^d), \mathcal{W}_p)$. 

Then there exists a sequence $(n_\tau)_{\tau >0}$ in $\mathbb{N}, \ n_\tau \uparrow +\infty$ as $\tau\downarrow0$, such that Assumption \ref{ass: 2} holds good for all choices $(n(\tau))_{\tau>0}, \ n(\tau)\in\mathbb{N},$ that satisfy $n(\tau)\geq n_\tau$.
\end{theorem}

\begin{proof}
We may apply Thm. 6.1 from \cite{fleissner2016gamma} since $(\mathcal{P}_p(\mathbb{R}^d), \mathcal{W}_p)$ is separable and complete (see e.g. [\cite{AGS08}, Prop. 7.1.5]). 
\end{proof}

The sequence $(n_\tau)_{\tau>0}$ from Theorem \ref{thm: MM approach II} solely depends on the velocity of $\Gamma$-convergence $\phi_n\stackrel{\Gamma}{\to} \phi$ (cf. Sect. 6 in \cite{fleissner2016gamma}). Taken together, Theorem \ref{thm: MM approach II} and Theorem \ref{thm: MM approach I} establish a robust stability theory for steepest descents under $\Gamma$-convergence of the energy functionals: under quite natural and mild conditions on $(\phi_n)_n$ and $\phi$, there exists a correlation $\tau\mapsto n(\tau)\in\mathbb{N}$ between time step sizes and parameters such that the corresponding joint discrete-time steepest descent motion along $(\phi_n)_n$ converges to the gradient flow of $\phi$. In actual fact, by Theorem \ref{thm: MM approach II}, there exist infinitely many such appropriate correlations precisely quantified by Assumption \ref{ass: 2} and completely independent of initial data $\mu_\tau^0, \mu^0 \in\mathcal{P}_p(\mathbb{R}^d)$ and of (approximate) minimizers $\mu_\tau^m$ in the (relaxed) Minimizing Movement scheme.

In \cite{fleissner2016gamma}, general methods for determining the right choices $n=n(\tau)$ with regard to Assumption \ref{ass: 2} are illustrated, see Sects. 7 and 8 therein. Also in this paper, we demonstrate the exact computation of appropriate correlations $\tau\mapsto n(\tau)$ as it is a crucial ingredient in the proofs of convergence statements for our Kernel-Density-Estimator Minimizing Movement Scheme (KDE-MM-Scheme), see Sect. \ref{sec: 3.3}. 
 
\subsection{KDE Convergence Rates}\label{sec: 2.3}

This section provides useful information on the asymptotic behaviour of Kernel Density Estimators (see Definition \ref{def: kernel density estimator}) as the sample size $n\uparrow+\infty$ and the bandwidth $h\downarrow0$ simultaneously, including uniform convergence rates. 
 
Please note that the expected value of a KDE $\hat{\rho}_{n,h}$ at some $x\in\mathbb{R}^d$ is equal to the value of the convolution between $\K_h$ and $\rho$ at $x$: 
\begin{equation}\label{eq: mean of kde}
\mathbb{E}[\hat{\rho}_{n,h}(x)] \ = \ \int_{\mathbb{R}^d}{\K_h(x-y)\rho(y)\mathrm{d}y}. 
\end{equation}

\begin{proposition}[cf. Thm. 12, Cor. 15, Lem. 14 and Thm. 27 in \cite{kim2019uniform}]\label{prop: convergence rates}
Let $\K$ be a Lipschitz continuous kernel function according to Definition \ref{def: kernel function} with compact support $\mathrm{spt}(\K)\subset\{|x|\leq1\}$ and $(\hat{\rho}_{n,h})_{n\in\mathbb{N}, h >0}$ an associated family of Kernel Density Estimators \eqref{eq: kernel density estimator} for a probability density $\rho$ on $\mathbb{R}^d$. Suppose that $\rho$ has compact support, i.e. $\mathrm{spt}(\rho)\subset\{|x|\leq R\}$ for some $R>0$. Every sample size $n\in\mathbb{N}$ is assigned a bandwidth $h(n)\in(0, \frac{9}{10}), \ \lim_{n\to+\infty}h(n) = 0$.

There exists a constant $C_{\K, \rho}$ (depending on $\K, R, d$ and $\rho$) so that for every $\alpha\in(0,1)$ and $n\in\mathbb{N}$ the following holds good with probability at least $1-\alpha$: 
\begin{equation}\label{eq: convergence rates 1} 
\sup_{x\in\mathbb{R}^d}\Big|\hat{\rho}_{n, h(n)}(x) - \mathbb{E}[\hat{\rho}_{n,h(n)}(x)]\Big| \ \leq \ C_{\K,\rho}\cdot \mathfrak{R}(n, \alpha)
\end{equation}
where
\begin{equation}\label{eq: convergence rates 2}
\mathfrak{R}(n, \alpha):= \sqrt{\frac{\log(1/h(n))}{nh(n)^{2d}}}+\sqrt{\frac{\log(1/\alpha)}{nh(n)^{2d}}} + \frac{\log(1/h(n))}{nh(n)^d} + \frac{\log(1/\alpha)}{nh(n)^d}. 
\end{equation}
\end{proposition}

\begin{remark}[Uniform convergence rates for general probability distributions and for KDE derivatives]\label{rem: fourth order}
It is noteworthy that in \cite{kim2019uniform} (cf. Thm. 12 and Cor. 15 therein) such uniform convergence rates \eqref{eq: convergence rates 1} are also established for densities with unbounded support, general probability distributions  and a wider class of kernel functions (involving more parameters than in Proposition \ref{prop: convergence rates}).

Moreover, if $\K$ is Lipschitz continuous and everywhere differentiable, then so is $x\mapsto\mathbb{E}[\hat{\rho}_{n,h}(x)]$ with $\nabla\mathbb{E}[\hat{\rho}_{n,h}(x)] = (\nabla\K_h\ast\rho)(x)$;  in Sect. 6  in \cite{kim2019uniform}, convergence rates for $\sup_{x\in\mathbb{R}^d}\Big|\nabla\hat{\rho}_{n, h(n)}(x) - \nabla\mathbb{E}[\hat{\rho}_{n,h(n)}(x)]\Big|$ similar to Proposition \ref{prop: convergence rates} are proved. They play a crucial role in the rigorous mathematical study of the KDE-MM-Scheme for fourth order diffusion equations. 
\end{remark}

Under the assumptions on $\K$ from Proposition \ref{prop: convergence rates}, the KDE expected value function $x\mapsto\mathbb{E}[\hat{\rho}_{n,h}(x)]$ defined in \eqref{eq: mean of kde} itself is a Lipschitz continuous probability density and the KDEs $\hat{\rho}_{n,h}$ are almost everywhere asymptotically unbiased, i.e.
\begin{equation}\label{eq: asymptotically unbiased}
\lim_{n\to+\infty} \mathbb{E}[\hat{\rho}_{n,h(n)}(x)] \ = \ \rho(x) \quad \text{ for } \mathcal{L}^d\text{-a.e. } x\in\mathbb{R}^d
\end{equation}
whenever $h(n)\downarrow 0$ as $n\uparrow+\infty$ (see e.g. Thm. 7 in Appx. C in \cite{evans2010partial}). Moreover, Lem. 7.1.10 in \cite{AGS08} and Prop. 7.10 in \cite{Villani03} provide us with the estimates
\begin{equation}\label{eq: Wasserstein convolution estimate}
\mathcal{W}_p(\mu, \mathbb{E}[\hat{\rho}_{n,h}(x)]\mathcal{L}^d)^p \ \leq \ h^p \cdot \int_{\mathbb{R}^d}{|x|^p\K(x)\mathrm{d}x}
\end{equation}
and 
\begin{equation}\label{eq: Wasserstein BV estimate}
\mathcal{W}_p(\hat{\mu}_{n, h}, \mathbb{E}[\hat{\rho}_{n,h}(x)]\mathcal{L}^d)^p \ \leq \ 2^{p-1}\int_{\mathbb{R}^d}{|x|^p|\hat{\rho}_{n,h}(x) - \mathbb{E}[\hat{\rho}_{n,h}(x)]|\mathrm{d}x}
\end{equation}
respectively, for $\mu:=\rho\mathcal{L}^d\in\mathcal{P}_p(\mathbb{R}^d)$ and associated KDE probability measure $\hat{\mu}_{n,h}:=\hat{\rho}_{n,h}\mathcal{L}^d$. 
Suitable choices for $h(n)$ and $\alpha=\alpha(n)$ in \eqref{eq: convergence rates 1} and \eqref{eq: convergence rates 2} s.t. 
\begin{equation*}
\sum_{n\in\mathbb{N}}{\alpha(n)} \ < \ +\infty, \quad \lim_{n\to+\infty}\mathfrak{R}(n, \alpha(n)) \ = \ 0,
\end{equation*}
an application of Borel-Cantelli lemma and \eqref{eq: asymptotically unbiased}, \eqref{eq: Wasserstein convolution estimate}, \eqref{eq: Wasserstein BV estimate} finally show both the almost everywhere strong consistency of the corresponding KDEs, i.e.
\begin{equation*}
\lim_{n\to+\infty} \hat{\rho}_{n,h(n)}(x) \ = \ \rho(x) \quad \text{for }\mathcal{L}^d\text{-a.e. } x\in\mathbb{R}^d \quad \text{a.s. ,}
\end{equation*}
and their strong consistency in $(\mathcal{P}_p(\mathbb{R}^d), \mathcal{W}_p)$, i.e. 
\begin{equation*}
\lim_{n\to+\infty}\mathcal{W}_p(\hat{\mu}_{n, h(n)}, \mu) \ = \ 0 \quad\quad\text{ a.s..}
\end{equation*}

This precise study of the asymptotic behaviour of Kernel Density Estimators is taken up again as ingredient in the procedure for selecting appropriate parameters for the KDE-MM-Scheme in Sections \ref{sec: 3.2}-\ref{sec: 3.3}.   

\section{The KDE-MM-Scheme}\label{sec: 3}
This section refocuses on the Kernel-Density-Estimator Minimizing Movement Scheme (KDE-MM-Scheme) outlined in Section \ref{sec: KDE-MM-Scheme intro} as approximation scheme for \eqref{eq: diffusion equation intro}, $q\in(1, +\infty)$. Let $p\in(1,+\infty)$ be the conjugate exponent of $q$. We fix a kernel function $\K$ according to Definition \ref{def: kernel function} assuming that $\K$ has finite moment of order $p$, i.e.  
\begin{equation}\label{eq: kernel finite moment}
\mathcal{M}_{\K, p} := \int_{\mathbb{R}^d}{|x|^p\K(x)\mathrm{d}x} < +\infty; 
\end{equation} 
$\K_h, \ h>0,$ denotes the associated family of functions \eqref{eq: K_h}.

\subsection{Definition and Consistency of KDE-MM-Scheme}\label{sec: 3.1}

The uniform convergence rates for Kernel Density Estimators and their derivatives according to Proposition \ref{prop: convergence rates} and Remark \ref{rem: fourth order} enable us to construct ``$\Gamma$-KDE-Approximations'' of energy functionals associated with second and fourth order diffusion equations \eqref{eq: diffusion equation intro}. We suppose \ref{ass: phi1} for the sake of a clear presentation with a straightforward notation but the condition can be omitted, see Remark \ref{rem: further relaxations}.

\begin{definition}[$\Gamma$-KDE-Approximation]\label{def: Gamma KDE Approximation}
We say that a sequence of energy functionals $\phi_n: \mathcal{P}_p(\mathbb{R}^d)\to(-\infty, +\infty], \ n\in\mathbb{N}$, is a $\Gamma$-KDE-Approximation of $\phi: \mathcal{P}_p(\mathbb{R}^d)\to(-\infty,+\infty]$ associated with the kernel function $\K$ and the correlation $n \mapsto h(n)$ between sample sizes and bandwidths if $h(n)>0,  \ h(n)\downarrow0$ as $n\uparrow+\infty$, 
\begin{enumerate}[label=($\mathcal{C}$\arabic*)]
\item \label{itm: C1} the effective domains $\{\phi_n<+\infty\}$ are concentrated in the KDE ranges corresponding to $\K, \ n, \ h(n)$, i.e.
\begin{equation*}
\phi_n(\mu) < +\infty \quad\Rightarrow\quad \exists y_1, .., y_n\in\mathbb{R}^d: \ \mu =  \Big(\frac{1}{n}\sum_{i=1}^{n}{\K_{h(n)}(\cdot-y_i)}\Big)\mathcal{L}^d, 
\end{equation*}
\item \label{itm: C2} $\phi_n\stackrel{\Gamma}{\to}\phi$ in $(\mathcal{P}_p(\mathbb{R}^d), \mathcal{W}_p)$, and 
\item \label{itm: C3} whenever $\mu=\rho\mathcal{L}^d\in\{\phi<+\infty\}$ and $\hat{\rho}_{n,h(n)}, \ n\in\mathbb{N},$ is a sequence of Kernel Density Estimators \eqref{eq: kernel density estimator} for $\rho$, the corresponding sequence of measure-valued random variables $\hat{\mu}_{n, h(n)}:=\hat{\rho}_{n,h(n)}\mathcal{L}^d, \ n\in\mathbb{N},$ almost surely constitutes a recovery sequence \eqref{eq: recovery sequence} for $\mu$, i.e. 
\begin{equation}\label{eq: KDE recovery sequence}
\lim_{n\to+\infty}\mathcal{W}_p(\hat{\mu}_{n, h(n)}, \mu) \ = \ 0 \quad \text{and} \quad \lim_{n\to+\infty}\phi_n(\hat{\mu}_{n, h(n)}) \ = \ \phi(\mu)
\end{equation}
with probability $1$. 
\end{enumerate}
\end{definition}

Variations (partial and weak $\Gamma$-KDE-Approximation) are introduced in Remark \ref{rem: Gamma KDE Approximation variation}. A $\Gamma$-KDE-Approximation of the energy functional from our illustrative Example \ref{ex: chain rule} is presented in Section \ref{sec: 3.2}. 

Our approach to second and fourth order diffusion equations \eqref{eq: diffusion equation intro} governed by an energy functional $\phi$ is to carry out the relaxed Minimizing Movement scheme \eqref{eq: relaxed MM 1}, \eqref{eq: relaxed MM 2} along a $\Gamma$-KDE-Approximation $\phi_n, \ n\in\mathbb{N},$ of $\phi$; $(\phi_n)_n$ may also be a partial, a weak or a partial weak $\Gamma$-KDE-Approximation of $\phi$ in accordance with Remark \ref{rem: Gamma KDE Approximation variation}. 

The simple structure of $\{\phi_n < +\infty\}$ as per \ref{itm: C1} brings an advantage for the implementation of the scheme in itself and moreover, it allows a significant simplification of the distance term in \eqref{eq: relaxed MM 2}. 

\begin{definition}[KDE-MM-Scheme]\label{def: KDE-MM-Scheme}
Let $\phi_n, \ n\in\mathbb{N},$ be a $\Gamma$-KDE-Approximation of $\phi: \mathcal{P}_p(\mathbb{R}^d)\to(-\infty, +\infty]$ associated with $\K$ and $n\mapsto h(n)$ according to Definition \ref{def: Gamma KDE Approximation}. We assume \ref{ass: phin1}. Every time step size $\tau >0$ is assigned a sample size $n(\tau)\in\mathbb{N}$ and the corresponding bandwidth $h(\tau):=h(n(\tau))$. 
 
For $\tau \in \big(0, \big(\frac{1}{pB}\big)^{1/(p-1)}\big)$ and a given initial datum $Y_\tau^0:=\big(y_{1,\tau}^0, ..., y_{n(\tau), \tau}^0\big)$, $y_{i,\tau}^0\in\mathbb{R}^d$, we find a sequence 
\begin{equation}\label{eq: MM sequence Y}
Y_\tau^m:=\big(y_{1,\tau}^m, ..., y_{n(\tau), \tau}^m\big),  \quad y_{i,\tau}^m\in\mathbb{R}^d, \quad \quad m\in\mathbb{N}_0
\end{equation}
by the scheme
\begin{equation}\label{eq: KDE-MM-Scheme}
\Psi(\tau, Y_\tau^{m-1}, Y_\tau^m) \ \leq \ \inf_{Z=(z_1, ..., z_{n(\tau)})}{\Psi(\tau, Y_\tau^{m-1}, Z)} \ + \ \gamma_\tau^{(m)} \quad \quad (m\geq1)
\end{equation}
associated with 
\begin{equation}\label{eq: KDE-MM-Scheme Psi}
\Psi(\tau, Y , Z) \ := \ \phi_{n(\tau)}\Big(\frac{1}{n(\tau)}\sum_{i=1}^{n(\tau)}{\K_{h(\tau)}(\cdot-z_i)\mathcal{L}^d}\Big) \ + \ \frac{1}{p\tau^{p-1}}\sum_{i=1}^{n(\tau)}{\frac{|z_i-y_i|^p}{n(\tau)}}
\end{equation}
for $Y:=\big(y_1,...,y_{n(\tau)}\big), \ Z:=\big(z_1, ..., z_{n(\tau)}\big),  \ y_i, \ z_i\in\mathbb{R}^d,$ and error terms $\gamma_\tau^{(m)} > 0, \ m\in\mathbb{N}$. 

We assign probability measures $\mu_\tau^m := u_\tau^m \mathcal{L}^d \in \mathcal{P}_p(\mathbb{R}^d), \ m\in\mathbb{N}_0$, 
\begin{equation}\label{eq: MM sequence}
u_{\tau}^m \ := \ \frac{1}{n(\tau)}\sum_{i=1}^{n(\tau)}{\K_{h(\tau)}(\cdot-y_{i,\tau}^m)}, 
\end{equation}
to the sequence $(Y_\tau^m)_{m\in\mathbb{N}_0}$; 
the corresponding piecewise constant interpolations $\mu_\tau: [0, +\infty)\to\mathcal{P}_p(\mathbb{R}^d),$
\begin{equation}\label{eq: discrete solution}
\mu_\tau(0)=\mu_\tau^0, \quad\quad \mu_\tau(t) \equiv \mu_\tau^m \quad \text{ if } t\in((m-1)\tau, m\tau], \ m\in\mathbb{N},
\end{equation}
are called discrete solutions. 
\end{definition}

The existence of solutions to the relaxed minimum problems \eqref{eq: KDE-MM-Scheme} for $0 < \tau < \big(\frac{1}{pB}\big)^{1/(p-1)}$ follows from Young's inequality, condition \ref{ass: phin1} and the fact that 
\begin{equation}\label{eq: 0}
\mathcal{W}_p\Big(\frac{1}{n}\sum_{i=1}^{n}{\K_{h(n)}(\cdot-y_i)\mathcal{L}^d}, \frac{1}{n}\sum_{i=1}^{n}{\K_{h(n)}(\cdot-z_i)\mathcal{L}^d}\Big)^p \ \leq \ \frac{1}{n}\sum_{i=1}^{n}{|y_i-z_i|^p} 
\end{equation}
for every $y_i, z_i \in\mathbb{R}^d$, which can be easily seen by testing the minimum problem in the definition of $\mathcal{W}_p$ on the measure $\frac{1}{n}\sum_{i=1}^{n}{(\mathrm{id}\times T_i)_{\#}(\K_{h(n)}(\cdot-y_i)\mathcal{L}^d)}$, $T_i(x):=z_i+x-y_i$. 
We are interested in the limiting behaviour of discrete solutions \eqref{eq: discrete solution} as the time step sizes $\tau\downarrow0$; we tacitly suppose in all our considerations that $\tau < \big(\frac{1}{pB}\big)^{1/(p-1)}$ so that $\inf_{Z=(z_1, ..., z_{n(\tau)})}{\Psi(\tau, Y, Z)} > -\infty$ for every $Y=(y_1, ..., y_{n(\tau)})$. 

The probabilistic aspect of the KDE-MM-Scheme is only implicitly visible in Definition \ref{def: KDE-MM-Scheme} through condition \ref{itm: C3} and the need for a concrete and feasible recovery sequence corresponding to some initial probability distribution (whose exact form is unknown in most cases).  The KDE-MM-Scheme yields (weak) solutions to \eqref{eq: diffusion equation intro}, see Theorem \ref{thm: KDE-MM-Scheme consistency} below; the discrete solutions  \eqref{eq: MM sequence}, \eqref{eq: discrete solution} of the KDE-MM-Scheme are discrete-time steepest descents in $(\mathcal{P}_p(\mathbb{R}^d), \mathcal{W}_p)$ driven by the motion of a finite number of particles / data points $y_{i, \tau}^m$.

\begin{theorem}[Strong consistency of KDE-MM-Scheme]\label{thm: KDE-MM-Scheme consistency}

Let $\phi_n, \ n\in\mathbb{N},$ be a $\Gamma$-KDE-Approximation of an energy functional $\phi: \mathcal{P}_p(\mathbb{R}^d)\to(-\infty, +\infty]$ according to Definition \ref{def: Gamma KDE Approximation}. We suppose that Assumption \ref{ass: 1} is satisfied.

Let $(\tau_k)_{k\in\mathbb{N}}$ be a sequence of time step sizes $\tau_k\downarrow0$. Every time step size $\tau_k$ is assigned a parameter (sample size) $n(\tau_k)$, the corresponding bandwidth $h(\tau_k):=h(n(\tau_k))$ and error terms $\gamma_{\tau_k}^{(m)}>0, \ m\in\mathbb{N},$ such that $n(\tau_k)\uparrow+\infty$, 
\begin{equation}\label{eq: KDE-MM-Scheme parameters}
\text{Assumption \ref{ass: 2} holds true,} \quad \quad \quad \lim_{k\to+\infty}\frac{h(\tau_k)}{\tau_k^p} \ = \ 0
\end{equation}
and the error terms satisfy \eqref{eq: error term non-uniform}. Assuming the associated KDE-MM-Scheme is performed with initial data $Y_{\tau_k}^0:=\big(X_1, ..., X_{n(\tau_k)}\big)$ with $X_1, ..., X_{n(\tau_k)}$ being an i.i.d. sample from some initial measure $\mu^0\in\{\phi<+\infty\}$, the following holds good for the corresponding discrete solutions $(\mu_{\tau_{k}})_{k\in\mathbb{N}}$ \eqref{eq: discrete solution} with probability $1$: 
\begin{enumerate}[label=(\roman*)]
\item There exist a subsequence of time step sizes $(\tau_{k_l})_{l\in\mathbb{N}}, \ \tau_{k_l}\downarrow0$ and a curve $\mu: [0, +\infty)\to\mathcal{P}_p(\mathbb{R}^d), \ \mu(t)=u(t, \cdot)\mathcal{L}^d,$ such that $\mu(0)=\mu^0,$
\begin{equation*}
\lim_{l\to+\infty}\mathcal{W}_p(\mu_{\tau_{k_l}}(t), \mu(t)) \ = \ 0 \quad \text{ for all }t\geq 0. 
\end{equation*}
The limit curve $\mu$ is locally absolutely continuous in $(\mathcal{P}_p(\mathbb{R}^d), \mathcal{W}_p)$ and satisfies the energy inequality \eqref{eq: energy inequality}. \label{itm: limit curve KDE-MM-Scheme}
\item If, in addition, $\phi\circ\mu$ belongs to $\mathrm{C}([0, +\infty))\cap\mathrm{W}_{1,1}^{\mathrm{loc}}((0, +\infty))$ satisfying $\partial_l\phi(\mu(t))=\{D_l\phi(\mu(t))\}$ and the chain rule \eqref{eq: chain rule}
for $\mathcal{L}^1$-a.e. $t>0$, then the limit curve $\mu$ from \ref{itm: limit curve KDE-MM-Scheme} and its tangent vector field solve the differential equation \eqref{eq: differential inclusion} (which is a weak reformulation of \eqref{eq: diffusion equation intro}) and \eqref{eq: subdifferential slope equality} for $\mathcal{L}^1$-a.e. $t>0$, the energy dissipation equality \eqref{eq: ede}
for all $0\leq s\leq t < +\infty$ and 
\begin{equation*}
\lim_{l\to+\infty}\phi_{n(\tau_{k_l})}(\mu_{\tau_{k_l}}(t)) \ = \ \phi(\mu(t)) \quad\quad \text{ for all } t\geq0. \label{itm: differential equation KDE-MM-Scheme}
\end{equation*}
\end{enumerate}

There exists a sequence $(n_\tau)_{\tau >0}$ with $n_\tau\in\mathbb{N}, \ n_\tau \uparrow +\infty$ as $\tau\downarrow0$, such that \eqref{eq: KDE-MM-Scheme parameters} holds good for all choices $(n(\tau))_{\tau>0}$ with $n(\tau)\in\mathbb{N}, \ n(\tau)\geq n_\tau$ and all sequences $(\tau_k)_{k\in\mathbb{N}}$ of time step sizes $\tau_k\downarrow0$. 

\end{theorem}

\begin{proof}
According to Definition \ref{def: Gamma KDE Approximation}, the initial data $\mu_{\tau_k}(0)$ associated with $Y_{\tau_k}^0$ almost surely form a recovery sequence, i.e. 
\begin{equation*}
\lim_{k\to+\infty}\mathcal{W}_p(\mu_{\tau_k}(0), \mu^0) = 0, \quad \lim_{k\to+\infty}\phi_{n(\tau_k)}(\mu_{\tau_k}(0)) = \phi(\mu^0)
\end{equation*}
with probability $1$. We apply Theorem \ref{thm: MM approach I} taking account of the differences in the distance terms of the relaxed Minimizing Movement scheme \eqref{eq: relaxed MM 1}, \eqref{eq: relaxed MM 2} and the KDE-MM-Scheme. They can be estimated in the following way. 

Let $Y_{\tau_k}^m:=\big(y_{1,\tau_k}^m, ..., y_{n(\tau_k), \tau_k}^m\big), \  y_{i, \tau_k}^m\in\mathbb{R}^d,$ be a solution \eqref{eq: MM sequence Y} of one step \eqref{eq: KDE-MM-Scheme} of the KDE-MM-Scheme and $\mu_{\tau_k}^m=u_{\tau_k}^m\mathcal{L}^d$ be the associated measure \eqref{eq: MM sequence}; $u_{\tau_k}^m$  equals the convolution between the scaled kernel function $\K_{h(\tau_k)}$ and the measure $\frac{1}{n(\tau_k)}\sum_{i=1}^{n(\tau_k)}{\delta_{y_{i, \tau_k}^m}}$ so that
\begin{equation}\label{eq: 1}
\mathcal{W}_p(\mu_{\tau_k}^m, \mu_{\tau_k}^{m-1})^p \ \leq \ \frac{1}{n(\tau_k)}\sum_{i=1}^{n(\tau_k)}{|y_{i, \tau_k}^m - y_{i, \tau_k}^{m-1}|^p} 
\end{equation}
according to \eqref{eq: 0}.
Estimate \eqref{eq: 1} is the first step towards the inequality
\begin{equation}\label{eq: final inequality}
\Phi(\tau_k, \mu_{\tau_k}^{m-1}, \mu_{\tau_k}^m) \ \leq \ \inf_{\nu\in\mathcal{P}_p(\mathbb{R}^d)}{\Phi(\tau, \mu_{\tau_k}^{m-1}, \nu)} \ + \ \bar{\gamma}_{\tau_k}^{(m)}
\end{equation}
for $\Phi$ defined as in \eqref{eq: relaxed MM 2} and a suitable error term $\bar{\gamma}_{\tau_k}^{(m)}>0$. 

Moreover, we apply the change of variables formula and Jensen's inequality to $\mathcal{W}_p(\mu_{\tau_k}(0), \delta_0)=\Big(\int_{\mathbb{R}^d}{|x|^p\mathrm{d}\mu_{\tau_k}(0)}\Big)^{1/p}$ so that we have
\begin{equation*}
\mathcal{W}_p(\mu_{\tau_k}(0), \delta_0) \ \geq \ \int_{\mathbb{R}^d}{\Big(\frac{1}{n(\tau_k)}\sum_{i=1}^{n(\tau_k)}{|h(\tau_k)z+X_i|^p}\Big)^{1/p}\K(z)\mathrm{d}z}
\end{equation*}
and using Minkowski's inequality for sequences, we obtain
\begin{equation}\label{eq: estimate sample}
\mathcal{W}_p(\mu_{\tau_k}(0), \delta_0) \ \geq \ \Big(\frac{1}{n(\tau_k)}\sum_{i=1}^{n(\tau_k)}{|X_i|^p}\Big)^{1/p} \ - \ h(\tau_k)\cdot C, 
\end{equation}
where $C:=\int_{\mathbb{R}^d}{|z|\K(z)\mathrm{d}z} \leq \mathcal{M}_{\K, p}^{1/p} < +\infty$ (cf. \eqref{eq: kernel finite moment}). 
Hence, there almost surely exist constants $R_{k,m}>0$ for all $k, m \in\mathbb{N}$ such that 
\begin{equation}\label{eq: constant R 1}
\Big(\frac{1}{n(\tau_k)}\sum_{i=1}^{n(\tau_k)}{ |y_{i, \tau_k}^m|^p}\Big)^{1/p} \ \leq \ R_{k,m} 
\end{equation}
and 
\begin{equation}\label{eq: constant R 2}
\sup \{R_{k,m}: \ k,m \text{ s.t. } m\tau_k\leq T\} \ < \ +\infty \quad\quad \text{for every} \quad\quad T>0; 
\end{equation}
in actual fact, it is not difficult to deduce \eqref{eq: constant R 1}, \eqref{eq: constant R 2} from the first part of the proof of Thm. 3.4 in \cite{fleissner2016gamma}, \eqref{eq: error term non-uniform}, \eqref{eq: estimate sample} and the facts that  
\begin{equation}\label{eq: initial estimate a.s.}
\sup_k\{\phi_{n(\tau_k)}(\mu_{\tau_k}(0)), \mathcal{W}_p(\mu_{\tau_k}(0), \delta_0)\} < +\infty \quad\quad \text{a.s.}
\end{equation}
and 
\begin{equation*}
\phi_{n(\tau_k)}(\mu_{\tau_k}^m) \geq  -A-B\cdot 2^{p-1}[h(\tau_k)\mathcal{M}_{\K, p}^{1/p}+\mathcal{W}_p(\delta_0, \mu_\star)]^p - B\cdot 2^{p-1}\cdot\Big(\frac{1}{n(\tau_k)}\sum_{i=1}^{n(\tau_k)}{|y_{i, \tau_k}^m|^p}\Big)
\end{equation*}
by \ref{ass: phin1}, \eqref{eq: 0} and the inequalities $(a+b)^p\leq 2^{p-1}(a^p+b^p)$ for $a,b \geq0$ and $\mathcal{W}_p(\K_{h(\tau_k)}(\cdot)\mathcal{L}^d, \mu_\star) \leq h(\tau_k)\mathcal{M}_{\K,p}^{1/p}+\mathcal{W}_p(\delta_0, \mu_\star)$. 

Further, we have 
\begin{equation*}
\mathcal{W}_p(\mu_{\tau_k}^{m-1}, \delta_0) \ \leq \ \Big(\frac{1}{n(\tau_k)}\sum_{i=1}^{n(\tau_k)}{|y_{i, \tau_k}^{m-1}|^p}\Big)^{1/p} + \underbrace{\mathcal{W}_p(\K_{h(\tau_k)}(\cdot)\mathcal{L}^d, \delta_0)}_{\leq h(\tau_k)\cdot \mathcal{M}_{\K,p}^{1/p}}
\end{equation*}
by \eqref{eq: 0}, \eqref{eq: kernel finite moment} and 
\begin{equation*}
\phi_{n(\tau_k)}(\mu_{\tau_k}^{m-1}) \ \leq \ \phi_{n(\tau_k)}(\mu_{\tau_k}(0)) + \sum_{j=1}^{m-1}{\gamma_{\tau_k}^{(j)}}
\end{equation*}
by \eqref{eq: KDE-MM-Scheme}. Consequently, the first part of the proof of Thm. 3.4 in \cite{fleissner2016gamma}, \eqref{eq: constant R 1}, \eqref{eq: constant R 2}, \eqref{eq: initial estimate a.s.}, \eqref{eq: error term non-uniform}, \ref{ass: phin1} and an estimate analogous to \eqref{eq: estimate sample} show that there almost surely exist constants $\bar{R}_{k,m}>0$ for all $k, m\in\mathbb{N}$ with $\tau_k$ small enough such that \eqref{eq: constant R 2} holds true for $\bar{R}_{k,m}$ and 
\begin{equation}\label{eq: constant bar R}
\Big(\frac{1}{n(\tau_k)}\sum_{i=1}^{n(\tau_k)}{ |y_{i}|^p}\Big)^{1/p} \ \leq \ \bar{R}_{k,m}
\end{equation}
whenever  $\nu := \Big( \frac{1}{n(\tau_k)}\sum_{i=1}^{n(\tau_k)}{\K_{h(\tau_k)}(\cdot - y_i)}\Big) \mathcal{L}^d$ satisfies
\begin{equation*}
\phi_{n(\tau_k)}(\nu) + \frac{1}{p\tau_k^{p-1}}\mathcal{W}_p(\nu, \mu_{\tau_k}^{m-1})^p \ \leq \ \inf_{\bar{\nu}\in\mathcal{P}_p(\mathbb{R}^d)}\Phi(\tau_k, \mu_{\tau_k}^{m-1}, \bar{\nu}) \ + \ 1, 
\end{equation*}
for $\Phi$ defined as in \eqref{eq: relaxed MM 2}. By Lem. 7.1.10 in \cite{AGS08}, 
\begin{equation}\label{eq: 2}
\mathcal{W}_p\Big( \nu, \ \frac{1}{n(\tau_k)}\sum_{i=1}^{n(\tau_k)}{\delta_{y_{i}}}\Big)^p \ \leq \ h(\tau_k)^p \cdot \mathcal{M}_{\K, p}. 
\end{equation}
As the measure $\nu$ is independent of the numbering order of $y_1, ..., y_{n(\tau_k)}$, we may assume w.l.o.g. that 
\begin{equation*}
\mathcal{W}_p\Big(\frac{1}{n(\tau_k)}\sum_{i=1}^{n(\tau_k)}{\delta_{y_{i}}} \ , \ \frac{1}{n(\tau_k)}\sum_{i=1}^{n(\tau_k)}{\delta_{y_{i, \tau_k}^{m-1}}}\Big)^p \ = \ \frac{1}{n(\tau_k)}\sum_{i=1}^{n(\tau_k)}{|y_{i}-y_{i,\tau_k}^{m-1}|^p}
\end{equation*}
(cf. pp. 5-6 in \cite{Villani03}). We obtain
\begin{equation}\label{eq: 5}
 \frac{1}{n(\tau_k)}\sum_{i=1}^{n(\tau_k)}{|y_{i}-y_{i,\tau_k}^{m-1}|^p} \ \leq \ \mathcal{W}_p(\nu, \mu_{\tau_k}^{m-1})^p \ + 2p\cdot h(\tau_k)\cdot\mathcal{M}_{\K,p}^{1/p}\cdot \tilde{R}_{k, m}^{p-1}, 
\end{equation}
$\tilde{R}_{k,m}:=\bar{R}_{k,m}+R_{k, m-1}$, 
by applying the estimate \eqref{eq: 2} to both $\nu$ and $\mu_{\tau_k}^{m-1}$, \eqref{eq: constant R 1}, \eqref{eq: constant bar R}, the triangle inequality and Minkowski's inequality for sequences. 

All in all, we infer from \eqref{eq: 1}, \eqref{eq: KDE-MM-Scheme}, \eqref{eq: 5}, the fact that the constants $\tilde{R}_{k,m}$ satisfy \eqref{eq: constant R 2} and from \eqref{eq: KDE-MM-Scheme parameters} that with probability $1$ the measures $\mu_{\tau_k}^m=u_{\tau_k}^m\mathcal{L}^d, \ m\in\mathbb{N},$ defined in \eqref{eq: MM sequence} solve the successive relaxed minimum problems \eqref{eq: final inequality}
for all small time step sizes $\tau_k$ and the corresponding error terms 
\begin{equation*}
\bar{\gamma}_{\tau_k}^{(m)} \ := \ \gamma_ {\tau_k}^{(m)}  + 2\mathcal{M}_{\K, p}^{1/p}\cdot\tilde{R}_{k, m}^{p-1}\cdot \frac{h(\tau_k)}{\tau_k^{p-1}}
\end{equation*}
satisfy condition \eqref{eq: error term non-uniform}. An application of Theorem \ref{thm: MM approach I} and Remark \ref{rem: error term} completes the proof of \ref{itm: limit curve KDE-MM-Scheme} and \ref{itm: differential equation KDE-MM-Scheme}. 

Since $h(n) \downarrow0$ as $n\uparrow+\infty$ by Definition \ref{def: Gamma KDE Approximation}, it is possible to select $n_\tau\in\mathbb{N} \ (\tau >0)$ according to Theorem \ref{thm: MM approach II} so that both Assumption \ref{ass: 2} and condition
\begin{equation*}
\lim_{\tau\downarrow0} \frac{h(n(\tau))}{\tau^p} \ = \ 0
\end{equation*}
hold true whenever $n(\tau)\in\mathbb{N}, \ n(\tau)\geq n_\tau$. The proof of Theorem \ref{thm: KDE-MM-Scheme consistency} is complete. 
\end{proof}

Section \ref{sec: 3.3} deals with the validation of Assumption \ref{ass: 2} and the precise selection of appropriate parameters $n=n(\tau), \ h= h(\tau)$. 

The reader is reminded that in the case of a uniform error distribution (i.e. $\gamma_\tau^{(m)} = \gamma_\tau$ for all $m\in\mathbb{N}$), condition \eqref{eq: error term non-uniform} means that the order of the error term $\gamma_\tau$ is $o(\tau)$, cf. \eqref{eq: error term} and Remark \ref{rem: error term}. 

We may relax conditions \ref{itm: C2} and \ref{itm: C3} in Definition \ref{def: Gamma KDE Approximation} and Assumption \ref{ass: 1}, perform the KDE-MM-Scheme along a partial or weak $\Gamma$-KDE-Approximation of the energy functional $\phi$ and still obtain the convergence / strong consistency statements from Theorem \ref{thm: KDE-MM-Scheme consistency}: 

\begin{remark}[Partial and weak $\Gamma$-KDE-Approximation]\label{rem: Gamma KDE Approximation variation}

We suppose that \ref{ass: phi3} is satisfied and that for small $\tau>0$ and all $\mu\in\mathcal{P}_p(\mathbb{R}^d)$ there exists a solution to the minimum problem
\begin{equation*}
\min_{\bar{\mu}\in\mathcal{P}_p(\mathbb{R}^d)}\Big\{\phi(\bar{\mu}) + \frac{1}{p\tau^{p-1}}\mathcal{W}_p(\bar{\mu}, \mu)^p\Big\}.
\end{equation*} 

We say that $(\phi_n)_{n\in\mathbb{N}}$ is a \textit{partial $\Gamma$-KDE-Approximation} of $\phi$ associated with $\K$ and $n\mapsto h(n)$ if $h(n)\downarrow0$ as $n\uparrow+\infty$, \ref{itm: C1} and the $\Gamma$-liminf-inequality \eqref{eq: Gamma-liminf} are satisfied and for a subset $\mathcal{I}\subset\{\phi < +\infty\}$ and all $\mu\in\mathcal{I}$, Kernel Density Estimation almost surely yields a recovery sequence \eqref{eq: KDE recovery sequence}. Statements \ref{itm: limit curve KDE-MM-Scheme} and \ref{itm: differential equation KDE-MM-Scheme} from Theorem \ref{thm: KDE-MM-Scheme consistency} still hold true if the KDE-MM-Scheme is performed along a partial $\Gamma$-KDE-Approximation for some initial measure $\mu^0\in\mathcal{I}$, cf. Remark \ref{rem: relaxation phin}.  Assuming \ref{itm: C2}, also Theorem \ref{thm: MM approach II} can be applied. 

We say that $(\phi_n)_{n\in\mathbb{N}}$ is a \textit{weak $\Gamma$-KDE-Approximation} of $\phi$ associated with $\K$ and $n\mapsto h(n)$ if $h(n)\downarrow0$ as $n\uparrow+\infty$, \ref{itm: C1} is satisfied, 
\begin{equation}\label{eq: weak Gamma-liminf}
\phi(\nu) \ \leq \ \liminf_{n\to+\infty}\phi_n(\nu_n) \quad\text{whenever}\quad v_n\rightharpoonup\nu, \ \sup_n\mathcal{W}_p(\nu_n, \nu) < +\infty
\end{equation}
and in \eqref{eq: KDE recovery sequence}, $\mathcal{W}_p(\hat{\mu}_{n,h(n)}, \mu) \to 0$ is replaced with 
\begin{equation*}
\sup_n\mathcal{W}_p(\hat{\mu}_{n,h(n)}, \mu) < +\infty, \ \hat{\mu}_{n,h(n)}\rightharpoonup\mu.
\end{equation*}
Theorem \ref{thm: KDE-MM-Scheme consistency} can be generalized to a weak $\Gamma$-KDE-Approximation according to Remark \ref{rem: relaxation phin}; condition \ref{ass: phin2} from Assumption \ref{ass: 1} is omitted in this case. The same is true of a \textit{partial weak $\Gamma$-KDE-Approximation}, whose definition is obvious.  
\end{remark}

A second order example of a partial weak $\Gamma$-KDE-Approximation is given in Remark \ref{rem: example weak Gamma KDE}. 

\begin{remark}[Further relaxation of Assumption \ref{ass: 1}]\label{rem: further relaxations}
The KDE approximation of general probability distributions is well examined and condition \ref{ass: phi1} from Assumption \ref{ass: 1} can be omitted, cf. Remarks \ref{rem: fourth order}, \ref{rem: relaxation} and \ref{rem: relaxation phin}. 
\end{remark}

Finally, 
we note that in view of Theorem \ref{thm: KDE-MM-Scheme consistency} \ref{itm: limit curve KDE-MM-Scheme} and Remark \ref{rem: absence of chain rule}, the KDE-MM-Scheme may be a suitable approximation scheme for \eqref{eq: diffusion equation intro} even if the chain rule \eqref{eq: chain rule} cannot be validated as there exist (at least for $q=p=2$) curves that satisfy both the energy inequality \eqref{eq: energy inequality} and a weak reformulation of \eqref{eq: diffusion equation intro}.

\subsection{$\Gamma$-KDE-Approximation: Second Order Example}\label{sec: 3.2}
We consider the class of energy functionals introduced in Example \ref{ex: chain rule} and prove that the obvious guess regarding a $\Gamma$-KDE-Approximation according to Definition \ref{def: Gamma KDE Approximation} succeeds.  

\begin{proposition}[$\Gamma$-KDE-Approximation: second order example]\label{prop: Gamma KDE}
Let $\K$ be a Lipschitz continuous kernel function according to Definition \ref{def: kernel function} with compact support $\mathrm{spt}(\K)\subset\{|x|\leq1\}$. Every $n\in\mathbb{N}$ is associated with parameters $h(n)\in(0,\frac{9}{10})$ and $\alpha(n)\in(0,1)$ such that $h(n)\downarrow 0$ and $\alpha(n)\to0$ as $n\uparrow+\infty$ and 
\begin{equation}\label{eq: h(n)}
\lim_{n\to+\infty}\frac{\log(h(n))+\log(\alpha(n))}{nh(n)^{2d}} \ = \ 0, \quad \quad \sum_{n\in\mathbb{N}}{\alpha(n)} \ < \ +\infty. 
\end{equation}
Under Assumptions \ref{ass: A0}, \ref{ass: A1}, \ref{ass: A2}, \ref{ass: A3}, the sequence of energy functionals $\phi_n: \mathcal{P}_p(\mathbb{R}^d)\to(-\infty, +\infty], \ n\in\mathbb{N}$, defined as
\begin{equation*}
\phi_n(\mu):= \int_{\mathbb{R}^d\times\mathbb{R}^d}{[F(u(x)) + V(x)u(x)+\frac{1}{2}W(x-y)u(x)]u(y)\mathrm{d}x\mathrm{d}y} 
\end{equation*}
whenever $\exists  y_1, ...,y_n\in\Omega: \ \mu=u\mathcal{L}^d, \ u(\cdot) =  \Big(\frac{1}{n}\sum_{i=1}^{n}{\K_{h(n)}(\cdot-y_i)}\Big)$, and $\phi_n(\mu):=+\infty$ else, is a $\Gamma$-KDE-Approximation of the energy functional $\phi$ from Example \ref{ex: chain rule}. 

The $\Gamma$-KDE-Approximation $\phi_n, \ n\in\mathbb{N}$, satisfies Assumption \ref{ass: 1} and therefore, all statements from Theorem \ref{thm: KDE-MM-Scheme consistency} hold true for the KDE-MM-Scheme corresponding to $(\phi_n)_{n\in\mathbb{N}}$. Moreover, assuming $\mu_t=u_t\mathcal{L}^d, \ t\geq0,$ is a locally absolutely continuous limit curve of discrete solutions $\bar{\mu}_{\tau_{k_l}}$ to the KDE-MM-Scheme solving the energy inequality \eqref{eq: energy inequality} according to Theorem \ref{thm: KDE-MM-Scheme consistency}\ref{itm: limit curve KDE-MM-Scheme} and \textsf{one} of the two conditions \ref{ass: B1}, \ref{ass: B2} is satisfied, then 
\begin{equation*}
L_F(u)\in\mathrm{L}^1_{\mathrm{loc}}([0, +\infty); \mathrm{W}^{1,1}(\mathbb{R}^d)), 
\end{equation*}
$\mu$  and its tangent vector field $v$ solve the differential equation 
\begin{equation}\label{eq: differential inclusion example}
v_t \ = \ -j_q\Big(\frac{\nabla L_F(u_t)}{u_t} + \nabla V + (\nabla W\ast\mu_t)\Big) \quad\quad \mu_t\text{-a.e.}
\end{equation}
for $\mathcal{L}^1$-a.e. $t>0$ ($j_q$ and $L_F$ defined in \eqref{eq: j_q} and \eqref{eq: L_F} respectively) and the energy dissipation equality \eqref{eq: ede}
for all $0\leq s\leq t < +\infty$, in which
\begin{equation}\label{subdifferential slope equality example}
|\partial^-\phi|^q(\mu(t)) \ = \ |\mu'|^p(t) \ = \ \Big\|\frac{\nabla L_F(u_t)}{u_t} + \nabla V + (\nabla W\ast\mu_t)\Big\|_{\mathrm{L}^q(\mu_t; \mathbb{R}^d)}^q  \quad \mathcal{L}^1\text{-a.e.,}
\end{equation}
and
\begin{equation*}
\lim_{l\to+\infty}\phi_{n(\tau_{k_l})}(\bar{\mu}_{\tau_{k_l}}(t)) \ = \ \phi(\mu(t)) \quad\quad \text{ for all } t\geq0. 
\end{equation*}
\end{proposition}

\begin{proof}
For the sake of a clear presentation with little notation, we set $V\equiv W\equiv0$ and $p=q=2$; it is absolutely straightforward to include $V$ and $W$ satisfying \ref{ass: A2} and \ref{ass: A3} respectively and the case $p\neq2$. 

Let $\Omega_1:=\{x\in\mathbb{R}^d: \ \mathrm{dist}(x, \Omega) < 1\}, \ \mathrm{dist}(x, \Omega):=\inf_{y\in\Omega}|x-y|$, choose $R>0$ in such a way that $\Omega_1\subset\{|x|\leq R\}$ and define 
\begin{equation*}
\bar{\phi}_1(\mu):= \begin{cases} \int_{\mathbb{R}}{F(u(x)) \mathrm{d}x} &\text{ if }\mu=u\mathcal{L}^d\ll\mathcal{L}^d\llcorner\Omega_1, \\
+\infty, &\text{ else.}\end{cases}
\end{equation*}

The $\Gamma$-liminf inequality \eqref{eq: Gamma-liminf} for $\phi_n, \phi$ follows from the lower semicontinuity of $\bar{\phi}_1$ in $(\mathcal{P}_2(\mathbb{R}^d), \mathcal{W}_2)$ and the facts that $\bar{\phi}_1\leq\phi_n$ for all $n\in\mathbb{N}$ and $\bar{\phi}_1(\mu)=\phi(\mu)$ if $\mu\ll\mathcal{L}^d\llcorner\Omega$. 

Let $\rho$ be a probability density with $\mu:=\rho\mathcal{L}^d\in\{\phi < +\infty\}$ and corresponding Kernel Density Estimators $\hat{\rho}_{n,h(n)}, \ \hat{\mu}_{n,h(n)}:=\hat{\rho}_{n,h(n)}\mathcal{L}^d$. 
By \eqref{eq: Wasserstein convolution estimate}, \eqref{eq: Wasserstein BV estimate}, the triangle inequality and Proposition \ref{prop: convergence rates}, we have 
\begin{equation*}
\mathcal{W}_2(\hat{\mu}_{n,h(n)}, \mu) \ \leq \  h(n) \ + \ R\cdot\sqrt{2\mathcal{L}^d(\Omega_1)\cdot C_{\K,\rho}\cdot \mathfrak{R}(n, \alpha(n))}
\end{equation*}
with probability at least $1-\alpha(n)$, $n\in\mathbb{N}$, where $C_{\K,\rho}$ is the constant depending on $\K,R,d$ and $\rho$ from \eqref{eq: convergence rates 1} and $\mathfrak{R}(n, \alpha(n))$ is defined according to \eqref{eq: convergence rates 2}. An application of \eqref{eq: h(n)} and Borel-Cantelli lemma then shows that 
\begin{equation*}
\lim_{n\to+\infty} \mathcal{W}_2(\hat{\mu}_{n,h(n)}, \mu) \ = \ 0 \quad\text{ a.s..}
\end{equation*}
Moreover, it follows from \eqref{eq: asymptotically unbiased}, \eqref{eq: convergence rates 1}, \eqref{eq: h(n)} and Borel-Cantelli lemma that
\begin{equation*}
\lim_{n\to+\infty}\hat{\rho}_{n,h(n)}(x) \ = \ \rho(x) \quad \text{ for } \mathcal{L}^d\text{-a.e. } x\in\mathbb{R}^d \quad \quad \text{a.s..}
\end{equation*}
We infer from the a.s. almost-everywhere-convergence of the probability densities and a general version of the dominated convergence theorem (see e.g. Chap. 2/ Ex. 20 in \cite{folland1999real}) that 
\begin{equation*}
\lim_{n\to+\infty}\phi_{n}(\hat{\mu}_{n, h(n)}) \ = \ \phi(\mu) \quad \quad \text{a.s.,}
\end{equation*}
using the following facts: according to Assumption \ref{ass: A1}, $F$ is bounded from below with global minimum at some $s_{\mathrm{min}}\in[0, +\infty)$, 
$F(\hat{\rho}_{n, h(n)}(x))=0$ outside $\Omega_1$,
\begin{equation*}
F(\hat{\rho}_{n,h(n)}(x)) \ \leq \ \max_{s\in[0,s_{\mathrm{min}}]}{F(s)} \ + \ F(\mathbb{E}[\hat{\rho}_{n,h(n)}(x)])  -  F(s_{\mathrm{min}})
\end{equation*}
if $\hat{\rho}_{n, h(n)}(x) \leq \mathbb{E}[\hat{\rho}_{n,h(n)}(x)]$, 
\begin{equation*}
F(\hat{\rho}_{n, h(n)}(x)) \ \leq \ C_F\big(1 + F(\mathbb{E}[\hat{\rho}_{n,h(n)}(x)]) + F(\hat{\rho}_{n, h(n)}(x)-\mathbb{E}[\hat{\rho}_{n, h(n)}(x)])\big)
\end{equation*}
if $\hat{\rho}_{n, h(n)}(x) > \mathbb{E}[\hat{\rho}_{n,h(n)}(x)]$, 
\begin{equation*}
F(\mathbb{E}[\hat{\rho}_{n,h(n)}(x)]) \ \leq \ \int_{\mathbb{R}^d}{F(\rho(y))\K_h(x-y)\mathrm{d}y}
\end{equation*}
by Jensen's inequality, and 
\begin{equation*}
\lim_{n\to+\infty} \sup_{x\in\mathbb{R}^d} F(|\hat{\rho}_{n, h(n)}(x)-\mathbb{E}[\hat{\rho}_{n, h(n)}(x)]|) \ = \ 0 \quad\quad\text{a.s.}
\end{equation*}
by \eqref{eq: convergence rates 1}, \eqref{eq: h(n)} and Borel-Cantelli lemma. 

The proof of \ref{itm: C3}, \ref{itm: C2} is complete. Further, it is not difficult to see that Assumption \ref{ass: 1} is satisfied so that Theorem \ref{thm: KDE-MM-Scheme consistency} is applicable. The rest of Proposition \ref{prop: Gamma KDE} follows from  an application of Young's inequality to  \eqref{eq: energy inequality}, Proposition \ref{prop: chain rule}, Theorem \ref{thm: KDE-MM-Scheme consistency}\ref{itm: differential equation KDE-MM-Scheme} and the characterization \eqref{eq: nabla L_F} of the limiting subdifferential of $\phi$ according to Proposition \ref{prop: limiting subdifferential}. The proof of Proposition \ref{prop: Gamma KDE} is complete. 
\end{proof}

Example \ref{ex: chain rule} includes energy functionals $\phi$ that are neither displacement convex nor $\lambda$-convex along constant speed geodesics in $(\mathcal{P}_p(\mathbb{R}^d), \mathcal{W}_p)$ because Assumption \ref{ass: A0} allows a non-convex domain $\Omega$, $F$ may not satisfy \ref{ass: B2} and Assumption \ref{ass: A2} allows $V$ that are not ($\lambda$)-convex. 

The differential equation \eqref{eq: differential inclusion example} for $t\mapsto\mu_t=u_t\mathcal{L}^d\ll\mathcal{L}^d\llcorner\Omega$ is a weak reformulation of the second order diffusion equation
\begin{equation}\label{eq: diffusion equation example}
\partial_t u - \nabla\cdot\Big(u j_q\Big(\nabla F'(u) + \nabla V + (\nabla W)\ast u \Big)\Big) \ = \ 0 \quad\text{ in } (0, +\infty)\times\Omega
\end{equation}
with no-flux boundary condition
\begin{equation}\label{eq: no-flux bc example}
uj_q\Big(\nabla F'(u) + \nabla V + (\nabla W) \ast u\Big)\cdot {\sf n} \ = \ 0 \quad \text{ on } (0, +\infty)\times\partial\Omega,
\end{equation}
cf. \eqref{eq: diffusion equation weak form}. We know by Proposition \ref{prop: Gamma KDE} and Theorem \ref{thm: KDE-MM-Scheme consistency} that there exist infinitely many appropriate correlations $\tau\mapsto n=n(\tau)$ between time step sizes and parameters (associated with $(\phi_n)_{n}$) so that the corresponding KDE-MM-Scheme,  performed according to the instructions from Definition \ref{def: KDE-MM-Scheme} and Theorem \ref{thm: KDE-MM-Scheme consistency}, functions as a sound approximation scheme for \eqref{eq: diffusion equation example}, \eqref{eq: no-flux bc example}. 
Our detailed examinations in Section \ref{sec: 3.3} demonstrate the practical selection of such appropriate parameters $n=n(\tau)$. 

\begin{remark}[Theoretical simplifications of the KDE-MM-Scheme]\label{rem: example simplifications KDE-MM-Scheme}
Let $\phi_n, \ \phi: \mathcal{P}_p(\mathbb{R}^d)\to(-\infty, +\infty]$ be the energy functionals from Proposition \ref{prop: Gamma KDE}.
The first part of this remark shows that for every time step size, the relaxed minimum problems \eqref{eq: KDE-MM-Scheme} of the KDE-MM-Scheme can be restricted to the selection of $y_{1, \tau}^m, ..., y_{n(\tau), \tau}^m$ from a fixed \textit{finite} set $\mathcal{S}_{\omega(\tau)}$. 

Let $\mathcal{S}_\omega$, for $\omega > 0$, be defined as a finite set of points in $\Omega$
such that
\begin{equation}\label{eq: finite set}
\bar{\Omega} \ \subset \ \bigcup_{s\in\mathcal{S}_\omega}\{x: \ |x-s| < \omega\}
\end{equation}
(which exists by compactness of $\bar{\Omega}$).  Every $n\in\mathbb{N}$ is associated with parameters $h(n)\in(0,\frac{9}{10})$, $\alpha(n)\in(0,1)$ and $\omega(n)\in(0,1)$ such that $h(n)\downarrow 0$ and $\alpha(n)\to0$ as $n\uparrow+\infty$, \eqref{eq: h(n)} is satisfied (as per our assumptions on the parameters from Proposition \ref{prop: Gamma KDE}) and in addition, 
\begin{equation}\label{eq: w(n)}
 \lim_{n\to+\infty} \frac{\omega(n)}{h(n)^{d+1}} \  = \ 0. 
\end{equation}   
We define the energy functionals $\psi_n: \mathcal{P}_p(\mathbb{R}^d)\to(-\infty, +\infty], \ n\in\mathbb{N},$ as
\begin{equation*}
\psi_n(\mu):=\begin{cases}\phi_n(\mu) &\text{ if } \exists y_1, ..., y_n\in\mathcal{S}_{\omega(n)}: \ \mu=\frac{1}{n}\sum_{i=1}^{n}{\K_{h(n)}(\cdot-y_i)}\mathcal{L}^d, \\ +\infty &\text{ else,}\end{cases}
\end{equation*} 
and $\mathfrak{p}_{\omega(n)}: \mathbb{R}^d\to\mathcal{S}_{\omega(n)}$ as a mapping satisfying 
\begin{equation*}
\mathfrak{p}_{\omega(n)}(x)\in\mathcal{S}_{\omega(n)}, \quad |\mathfrak{p}_{\omega(n)}(x) - x| \ = \ \min_{z\in\mathcal{S}_{\omega(n)}}{|z-x|} \quad \text{ for all } x\in\mathbb{R}^d. 
\end{equation*}
Following the same argumentation as in the proof of Proposition \ref{prop: Gamma KDE} and using \eqref{eq: w(n)} and the estimate
\begin{equation}\label{eq: estimate finite subset}
\sup_{x\in\mathbb{R}^d}\Big|\frac{1}{n}\sum_{i=1}^{n}{\K_{h(n)}(x-y_i)}-\frac{1}{n}\sum_{i=1}^{n}{\K_{h(n)}(x-\mathfrak{p}_{\omega(n)}(y_i))}\Big| \ \leq \ L_\K \cdot \frac{\omega(n)}{h(n)^{d+1}}
\end{equation}
($L_\K$ Lipschitz constant of $\K$), we can prove that $\psi_n\stackrel{\Gamma}{\to}\phi$ in $(\mathcal{P}_p(\mathbb{R}^d), \mathcal{W}_p)$ and whenever $X_1, ..., X_n$ is an i.i.d. sample from $\mu\in\{\phi<+\infty\}$ and 
\begin{equation*}
\check{\mu}_{n, h(n), \omega(n)}:=\frac{1}{n}\sum_{i=1}^{n}{\K_{h(n)}(\cdot-\mathfrak{p}_{\omega(n)}(X_i))}\mathcal{L}^d, \quad n\in\mathbb{N},
\end{equation*}
then
\begin{equation*}
\lim_{n\to+\infty}\mathcal{W}_p(\check{\mu}_{n, h(n), \omega(n)}, \mu) \ = \ 0 \quad \text{and} \quad \lim_{n\to+\infty}\psi_n(\check{\mu}_{n, h(n), \omega(n)}) \ = \ \phi(\mu)
\end{equation*}
with probability $1$. The functionals $\psi_n, \ n\in\mathbb{N},$ satisfy Assumption \ref{ass: 1}. 

If performing the KDE-MM-Scheme \eqref{eq: KDE-MM-Scheme} associated with 
\begin{equation}\label{eq: KDE-MM-Scheme Psi modified}
\Psi(\tau, Y , Z) \ := \ \psi_{n(\tau)}\Big(\frac{1}{n(\tau)}\sum_{i=1}^{n(\tau)}{\K_{h(\tau)}(\cdot-z_i)\mathcal{L}^d}\Big) \ + \ \frac{1}{p\tau^{p-1}}\sum_{i=1}^{n(\tau)}{\frac{|z_i-y_i|^p}{n(\tau)}}
\end{equation} 
in place of \eqref{eq: KDE-MM-Scheme Psi} and with initial data $\check{Y}_{\tau}^0:=\big(\mathfrak{p}_{\omega(\tau)}(X_1), ..., \mathfrak{p}_{\omega(\tau)}(X_{n(\tau)})\big)$ where $X_1, ..., X_{n(\tau)}$ is an i.i.d. sample from some initial measure $\mu^0\in\{\phi<+\infty\}$ and $\omega(\tau):=\omega(n(\tau))$, we can prove all the statements from Theorem \ref{thm: KDE-MM-Scheme consistency} for this modified version of the KDE-MM-Scheme, too; of course, also the second part of Proposition \ref{prop: Gamma KDE} still holds true. There exist infinitely many correlations $\tau\mapsto n(\tau)$ corresponding to $(\psi_n)_{n\in\mathbb{N}}$ such that \eqref{eq: KDE-MM-Scheme parameters} is satisfied and the selection procedure demonstrated in Section \ref{sec: 3.3} can be easily adapted for \eqref{eq: KDE-MM-Scheme}, \eqref{eq: KDE-MM-Scheme Psi modified} using the estimate \eqref{eq: estimate finite subset}. A typical by-product of such procedure for selecting appropriate parameters $n=n(\tau)$ is the proof that the modified KDE-MM-Scheme \eqref{eq: KDE-MM-Scheme}, \eqref{eq: KDE-MM-Scheme Psi modified} can be performed directly with initial data $Y_\tau^0:=\big(X_1,...,X_{n(\tau)}\big)$ instead of $\check{Y}_\tau^0$. 

It should be emphasized that the set $\{\psi_{n(\tau)} < +\infty\}$ merely consists of a \textit{finite} number of probability measures and that each of the successive steps \eqref{eq: KDE-MM-Scheme} associated with \eqref{eq: KDE-MM-Scheme Psi modified} consists in selecting a \textit{finite} number of points from a fixed \textit{finite} subset of $\Omega$. Moreover, if 
\begin{equation*}
\Psi(\tau, Y_\tau^{m-1}, Y_\tau^{m}) \ = \ \min\big\{\Psi(\tau, Y_\tau^{m-1}, Z): \ Z=(z_1, ..., z_{n(\tau)}), \ z_i\in\mathcal{S}_{\omega(\tau)}\big\}
\end{equation*}
for $Y_\tau^m=\big(y_{1,\tau}^m, ..., y_{n(\tau),\tau}^m\big), \ y_{i,\tau}^m\in\mathcal{S}_{\omega(\tau)},$ and $Y_\tau^m \neq  Y_\tau^{m-1}$, then
\begin{equation*}
\psi_{n(\tau)}\Big(\frac{1}{n(\tau)}\sum_{i=1}^{n(\tau)}{\K_{h(\tau)}(\cdot- y_{i, \tau}^{m})}\mathcal{L}^d\Big) \ < \ \psi_{n(\tau)}\Big(\frac{1}{n(\tau)}\sum_{i=1}^{n(\tau)}{\K_{h(\tau)}(\cdot- y_{i, \tau}^{m-1})}\mathcal{L}^d\Big). 
\end{equation*}
Consequently, for every $\tau \in \big(0, \big(\frac{1}{pB}\big)^{1/(p-1)}\big)$, there exists a discrete solution $\mu_\tau$ \eqref{eq: discrete solution} to the scheme \eqref{eq: KDE-MM-Scheme}, \eqref{eq: KDE-MM-Scheme Psi modified} such that
\begin{equation}\label{eq: merely finite steps}
\exists \ \bar{m}\in\mathbb{N}: \quad\quad \mu_\tau(t)\equiv\mu_\tau^{\bar{m}} \quad \text{ if } t\in((\bar{m}-1)\tau, +\infty), 
\end{equation}
meaning that for every time step size only a \textit{finite} number of minimization steps is necessary. 

A second simplification of the KDE-MM-Scheme concerns the potential energy 
\begin{equation*}
\mathsf{V}(\nu):=\int_{\mathbb{R}^d}{V(x)\mathrm{d}\nu}
\end{equation*}
and the interaction energy
\begin{equation*}
\mathsf{W}(\nu):=\frac{1}{2}\int_{\mathbb{R}^d}{\int_{\mathbb{R}^d}{W(x-y)\mathrm{d}\nu(x)}\mathrm{d}\nu(y)}. 
\end{equation*}
Let $\Omega_1, R$ be as in the proof of Proposition \ref{prop: Gamma KDE} and let $L_V$ and $L_W$ denote Lipschitz constants of $V\big|_{\bar{\Omega}_1}$ and $W\big|_{\{|x|\leq 2R\}}$ respectively. It is not difficult to see that 
\begin{equation*}
\Big|\mathsf{V}(\nu) - \frac{1}{n(\tau)}\sum_{i=1}^{n(\tau)}{V(y_i)}\Big| \ \leq L_V\ \cdot h(\tau)\cdot \mathcal{M}_{\K,p}^{1/p}
\end{equation*}
and 
\begin{equation*}
\Big|\mathsf{W}(\nu) - \frac{1}{2n(\tau)^2}\sum_{i=1}^{n(\tau)}{\sum_{j=1}^{n(\tau)}{W(y_i-y_j)}}\Big| \ \leq \ L_W\cdot h(\tau)\cdot \mathcal{M}_{\K,p}^{1/p}
\end{equation*}
whenever $\nu = \Big( \frac{1}{n(\tau)}\sum_{i=1}^{n(\tau)}{\K_{h(\tau)}(\cdot - y_i)}\Big) \mathcal{L}^d$ for some $y_i\in\Omega$, using the estimate \eqref{eq: 2}. Therefore, if condition \eqref{eq: KDE-MM-Scheme parameters} is satisfied (and hence $h(\tau)=o(\tau)$), we may replace $\phi_{n(\tau)}\Big(\frac{1}{n(\tau)}\sum_{i=1}^{n(\tau)}{\K_{h(\tau)}(\cdot-z_i)\mathcal{L}^d}\Big)$ in the definition \eqref{eq: KDE-MM-Scheme Psi} of $\Psi$ with
\begin{equation*}
\int_{\mathbb{R}^d}{F\Big(\frac{1}{n(\tau)}\sum_{i=1}^{n(\tau)}{\K_{h(\tau)}(x-z_{i})}\Big)\mathrm{d}x} \ + \ \frac{1}{n(\tau)} \sum_{i=1}^{n(\tau)}{\Big(V(z_i) + \frac{1}{n(\tau)} \sum_{i<j}{W(z_i-z_j)}\Big)}
\end{equation*}
and still, all statements from Theorem \ref{thm: KDE-MM-Scheme consistency} hold true. 
\end{remark}

The concept of $\Gamma$-KDE-Approximations is perfectly suited for fourth order diffusion equations, too: 

\begin{remark}[$\Gamma$-KDE-Approximation: fourth order case]\label{rem: Gamma KDE fourth order}
We can construct $\Gamma$-KDE-Approximations corresponding to fourth order examples of \eqref{eq: diffusion equation intro} in a similar way using the uniform convergence rates for KDE derivatives from \cite{kim2019uniform}, cf. Remark \ref{rem: fourth order}. 
\end{remark}

Considering the applications of \eqref{eq: diffusion equation intro} to concrete physical, biological, chemical, etc. processes (see Section \ref{sec: intro}), it is reasonable to assume that the domain $\Omega$ is bounded imposing a no-flux boundary condition as in Example \ref{ex: chain rule}; for the sake of completeness however, we give an example of a partial weak $\Gamma$-KDE-Approximation corresponding to \eqref{eq: diffusion equation intro} on $\Omega=\mathbb{R}^d$. 
\begin{remark}[$\Gamma$-KDE-Approximation: unbounded domain]\label{rem: example weak Gamma KDE}
We deal with the porous medium equation and the heat equation on $\mathbb{R}^d$ and construct a partial weak $\Gamma$-KDE-Approximation according to Remark \ref{rem: Gamma KDE Approximation variation} for the energy functional
\begin{equation*}
\mathsf{F}: \mathcal{P}_2(\mathbb{R}^d) \to (-\infty, +\infty], \quad \mathsf{F}(\mu):=\begin{cases}\int_{\mathbb{R}^d}{F(u(x))\mathrm{d}x} &\text{ if } \mu=u\mathcal{L}^d, \\ +\infty &\text {else, }\end{cases} 
\end{equation*}
 $F:[0,+\infty) \to \mathbb{R}^d$ defined as $F(s):= \frac{1}{m-1}s^m, \ m>1$ (for the porous medium equation) or $F(s):=s\log s$ (for the heat equation). Let $\K$ be a Lipschitz continuous kernel function with compact support  $\mathrm{spt}(\K)\subset\{|x|\leq1\}$ and let $R(n)\uparrow+\infty$ and $h(n)\downarrow0$ as $n\uparrow+\infty$ and \eqref{eq: h(n)} be satisfied for some $\alpha(n)\to0$. The sequence of energy functionals $\mathsf{F}_n: \mathcal{P}_2(\mathbb{R}^d) \to (-\infty, +\infty]$, 
\begin{equation*}
\mathsf{F}_n(\mu):=\begin{cases}\mathsf{F}(\mu) &\text{ if } \exists y_1, ..., y_n\in\{|x|\leq R(n)\}: \ \mu=\frac{1}{n}\sum_{i=1}^{n}{\K_{h(n)}(\cdot-y_i)}\mathcal{L}^d, \\ +\infty &\text{ else,}\end{cases}
\end{equation*}
forms a partial weak $\Gamma$-KDE-Approximation of $\mathsf{F}$. The weak $\Gamma$-liminf inequality \eqref{eq: weak Gamma-liminf} follows from the facts that $\mathsf{F}$ satisfies \eqref{eq: weak lsc} (cf. Prop. 4.1 in \cite{jordan1998variational}) and $\mathsf{F}\leq\mathsf{F}_n$. 
Moreover, whenever $\mu=\rho\mathcal{L}^d\in\{\phi<+\infty\}$ has compact support, $\hat{\rho}_{n,h(n)}, \ n\in\mathbb{N},$ is a sequence of Kernel Density Estimators \eqref{eq: kernel density estimator} for $\rho$ and $\hat{\mu}_{n, h(n)}:=\hat{\rho}_{n,h(n)}\mathcal{L}^d$, then \eqref{eq: KDE recovery sequence} holds true with probability $1$, cf. the proof of Proposition \ref{prop: Gamma KDE}. 

The topology induced by weak convergence is included as auxiliary topology so that we can apply a modified version of Theorem \ref{thm: KDE-MM-Scheme consistency} omitting condition \ref{ass: phin2} from Assumption \ref{ass: 1}, see Remarks \ref{rem: Gamma KDE Approximation variation} and \ref{rem: relaxation phin}; we note that condition \ref{ass: phin1} is satisfied and for small $\tau>0$ and all $\mu\in\mathcal{P}_p(\mathbb{R}^d)$ there exists a solution to the minimum problem
\begin{equation*}
\min_{\bar{\mu}\in\mathcal{P}_p(\mathbb{R}^d)}\Big\{\mathsf{F}(\bar{\mu}) + \frac{1}{p\tau^{p-1}}\mathcal{W}_p(\bar{\mu}, \mu)^p\Big\}
\end{equation*}
(cf.  Prop. 4.1 in \cite{jordan1998variational}). We refer the reader to Remark \ref{rem: selection unbounded domain} for the selection of appropriate parameters $n=n(\tau)$ according to \eqref{eq: KDE-MM-Scheme parameters}. 
\end{remark}

\subsection{Selection of Parameters for KDE-MM-Scheme}\label{sec: 3.3}
The KDE-MM-Scheme corresponding to \eqref{eq: diffusion equation intro} is performed along a $\Gamma$-KDE-Approximation $\phi_n, n\in\mathbb{N},$ of the energy functional $\phi$ as per the instructions from Definition \ref{def: KDE-MM-Scheme} and Theorem \ref{thm: KDE-MM-Scheme consistency}. 
Whilst Theorem \ref{thm: KDE-MM-Scheme consistency} shows the existence of appropriate parameters $n(\tau)\in\mathbb{N}, \ h(\tau):=h(n(\tau))$ assigned to every time step size $\tau>0$ according to \eqref{eq: KDE-MM-Scheme parameters}, the purpose of this section is to precisely quantify them. 

We demonstrate a possible strategy for proving Assumption \ref{ass: 2} with Example \ref{ex: chain rule} and the $\Gamma$-KDE-Approximation from Proposition \ref{prop: Gamma KDE}. As in the proof of Proposition \ref{prop: Gamma KDE}, we set $V\equiv W\equiv 0$ and $p=q=2$. For $M>0$, we define 
\begin{equation*}
\Lambda_M:=\{(a,b): a,b \geq0, \ |F(s_1)-F(s_2)| \ \leq \ a + b|s_1-s_2| \quad \forall s_1,s_2\in[0,M]\} 
\end{equation*}
 and the concave modulus of continuity 
\begin{equation}\label{eq: concave modulus of continuity}
f_M(r):= \inf\{a+br: \ (a,b) \in \Lambda_M\}, \quad r>0, 
\end{equation}
which is concave, increasing and bounded by $2\cdot\max_{s\in[0,M]} F(s)$ and satisfies 
\begin{equation*}
\lim_{r\downarrow0}f_M(r) \  = \ 0, 
\end{equation*}
see e.g. Sect. 4.1 in \cite{fleissner2020reverse}. 

\begin{theorem}[Selection of parameters: second order example]\label{thm: selection of parameters}
Let $\Omega$ be an open and bounded subset of $\mathbb{R}^d$ with $\mathrm{C}^2$-boundary $\partial\Omega$ and $F: [0, +\infty)\to\mathbb{R}$ satisfy \ref{ass: A1}. We define the $\Gamma$-KDE-Approximation $\phi_n: \mathcal{P}_2(\mathbb{R}^d)\to(-\infty, +\infty],$
\begin{equation*}
\phi_n(\mu):= \begin{cases} \int_{\mathbb{R}^d}{F(u(x))\mathrm{d}x} &\text{ if } \exists  y_i\in\Omega: \ \mu=u\mathcal{L}^d, \ u(\cdot) = \frac{1}{n}\sum_{i=1}^{n}{\K_{h(n)}(\cdot-y_i)}, \\ +\infty &\text{ else,}
\end{cases}
\end{equation*}
of the energy functional $\phi: \mathcal{P}_2(\mathbb{R}^d)\to(-\infty, +\infty]$, 
\begin{equation*}
\phi(\mu):= \begin{cases} \int_{\mathbb{R}^d}{F(u(x))\mathrm{d}x} &\text{ if } \mu=u\mathcal{L}^d\ll\mathcal{L}^d\llcorner\Omega, \\ +\infty &\text{ else,}\end{cases}
\end{equation*}
as per Proposition \ref{prop: Gamma KDE}. In addition to \eqref{eq: h(n)} and $h(n)\downarrow0$ as $n\uparrow+\infty$, we suppose 
\begin{equation*}
\liminf_{n\to+\infty} f_{\frac{||\K||_\infty}{h(n)^d}}\Bigg(\sqrt{\frac{\log(1/h(n))}{nh(n)^{2d}}}\Bigg) \ = \ 0. 
\end{equation*}
Every time step size $\tau>0$ is assigned a sample size $n(\tau)\in\mathbb{N}$ and bandwidth $h(\tau):=h(n(\tau))$ in such a way that $n(\tau)\uparrow+\infty$ as $\tau\downarrow0$, 
\begin{equation}\label{eq: s h(tau)}
\lim_{\tau\to0} \ \frac{h(\tau)}{\tau^2} \ = \ 0, 
\end{equation}
\begin{equation}\label{eq: s par 1}
\lim_{\tau\to0} \ \frac{\log(1/h(\tau))}{\tau^6\cdot n(\tau)h(\tau)^{2d}} \ = \ 0,
\end{equation}
and 
\begin{equation}\label{eq: s par 2}
\lim_{\tau\to0} \ \Bigg[\frac{1}{\tau}\cdot f_{\frac{||\K||_\infty}{h(\tau)^d}}\Bigg(\sqrt{\frac{\log(1/h(\tau))}{n(\tau)h(\tau)^{2d}}}\Bigg)\Bigg] \ = \ 0. 
\end{equation}
Then the parameters satisfy condition \eqref{eq: KDE-MM-Scheme parameters} from Theorem \ref{thm: KDE-MM-Scheme consistency} and all statements from Theorem \ref{thm: KDE-MM-Scheme consistency} and Proposition \ref{prop: Gamma KDE} apply. 
\end{theorem}

\begin{proof}
Let $\nu_\tau, \nu\in\mathcal{P}_2(\mathbb{R}^d)$ satisfy
\begin{equation*}
\lim_{\tau\to0}\mathcal{W}_2(\nu_\tau, \nu)  =  0, \quad \sup_\tau\phi_{n(\tau)}(\nu_\tau)  < +\infty. 
\end{equation*}
The following is our strategy for proving 
\begin{equation}\label{eq: nu_tau main condition}
\liminf_{\tau\to0}\frac{\phi_{n(\tau)}(\nu_\tau) - \mathcal{Y}_\tau\phi_{n(\tau)}(\nu_\tau)}{\tau} \ \geq \ \frac{1}{2}|\partial^-\phi|^2(\nu)
\end{equation}
($=$ condition \eqref{eq: main condition} from Assumption \ref{ass: 2}; $\mathcal{Y}_\tau\cdot$ denotes the $2$-Moreau-Yosida approximation \eqref{eq: Yosida}). First, we find probability measures $\eta_\tau\in\{\phi<+\infty\}$ satisfying
\begin{equation}\label{eq: eta_tau}
\lim_{\tau\to0}\frac{\phi_{n(\tau)}(\nu_\tau) - \phi(\eta_\tau)}{\tau} \ = \ 0, \quad \quad \lim_{\tau\to0}\frac{\mathcal{W}_2(\eta_\tau, \nu_\tau)}{\tau^2} \ = \ 0. 
\end{equation}
By Prop. 4.1 in \cite{fleissner2016gamma}, we have 
\begin{equation}\label{eq: eta_tau main condition}
\liminf_{\tau\to0}\frac{\phi(\eta_\tau) - \mathcal{Y}_\tau\phi(\eta_\tau)}{\tau} \ \geq \ \frac{1}{2}|\partial^-\phi|^2(\nu)
\end{equation}
since $\phi$ satisfies \ref{ass: phi2}, \ref{ass: phi3}, \ref{ass: phi4}, see Remark \ref{rem: main ass}\ref{rem: main ass i} and Example \ref{ex: chain rule}. The next step is to break the proof of \eqref{eq: nu_tau main condition} down into \eqref{eq: eta_tau main condition} and the proof of 
\begin{equation}\label{eq: main inequality}
\liminf_{\tau\to0}\frac{\mathcal{Y}_\tau\phi(\eta_\tau) - \mathcal{Y}_\tau\phi_{n(\tau)}(\nu_\tau)}{\tau} \ \geq \ 0: 
\end{equation}
indeed, the term $\phi_{n(\tau)}(\nu_\tau) - \mathcal{Y}_\tau\phi_{n(\tau)}(\nu_\tau)$ equals 
\begin{equation*}
\phi_{n(\tau)}(\nu_\tau) - \phi(\eta_\tau) \ + \ \phi(\eta_\tau) - \mathcal{Y}_\tau\phi(\eta_\tau) \ + \ \mathcal{Y}_\tau\phi(\eta_\tau) - \mathcal{Y}_\tau\phi_{n(\tau)}(\nu_\tau) 
\end{equation*}
so that \eqref{eq: nu_tau main condition} will directly follow from \eqref{eq: eta_tau}, \eqref{eq: eta_tau main condition} and \eqref{eq: main inequality}. The estimate $\mathcal{W}_2(\eta_\tau, \nu_\tau)=o(\tau^2)$ will help to prove \eqref{eq: main inequality}. 

Let us construct suitable measures $\eta_\tau\in\{\phi<+\infty\}$. 
The support of $\nu_\tau$ lies in $\Omega_{h(\tau)}:=\{x\in\mathbb{R}^d: \ \mathrm{dist}(x, \Omega) < h(\tau)\}, \ \mathrm{dist}(x, \Omega):=\inf_{y\in\Omega}|x-y|,$ and we `compress' its density function $\upsilon_\tau$ into $\Omega$. 
We define $\bar{d}_{\Omega}: \mathbb{R}^d\to\mathbb{R}$ as the signed distance function 
\begin{equation*}
\bar{d}_{\Omega}(x) \ := \ \mathrm{dist}(x, \Omega) - \mathrm{dist}(x, \mathbb{R}^d\setminus\Omega). 
\end{equation*}
As $\partial\Omega$ is $\mathrm{C}^2$, $\bar{d}_\Omega$ is twice continuously differentiable in a tubular neighbourhood $\mathcal{U}:=\{-h_\mathcal{U} < \bar{d}_\Omega < h_\mathcal{U}\} = \{x+\sigma\nabla\bar{d}_\Omega(x): \ x\in\partial\Omega, \ \sigma\in(-h_\mathcal{U}, h_\mathcal{U})\}$ of $\partial\Omega$,  $\nabla\bar{d}_\Omega$ represents the outer normal to $\Omega$ and for all $y\in\mathcal{U}$ there exists a unique $x\in\partial\Omega$ s.t. $y=x+\bar{d}_\Omega(y)\nabla\bar{d}_\Omega(x)$ and $\nabla\bar{d}_\Omega(y)=\nabla\bar{d}_\Omega(x)$, 
see Thm. 4.3 (i) and Rem. 4.2 in \cite{ambrosio2000calculus}. 
Let $\chi: \mathbb{R}^d\to[0, 1]$ be a smooth function satisfying
\begin{equation*}
\chi(x) \begin{cases} = 0 &\text{ if } \big(\bar{d}_\Omega(x)\geq h_\mathcal{U}/2 \text{ or } \bar{d}_\Omega(x)\leq -h_\mathcal{U}/2\big), \\ =1 &\text{ if } \bar{d}_\Omega(x)\in[0,h_\mathcal{U}/4], \\ \in[0,1] &\text{ else}.    \end{cases}
\end{equation*}
For $h\in(0, h_\mathcal{U}/4)$ small enough, let's say for all $h\in(0,h^\star)$, the function $\mathcal{A}_{h}: \mathbb{R}^d\to\mathbb{R}^d$, 
\begin{equation*}
\mathcal{A}_{h}(x) := \begin{cases} x - h\cdot\chi(x)\cdot\nabla\bar{d}_\Omega(x) &\text{ if } x\in\mathcal{U}, \\ x &\text{ else},\end{cases}
\end{equation*}
is a diffeomorphism mapping $\Omega_h$ onto $\Omega$; since $h(\tau)\downarrow0$ as $\tau\downarrow0$, we may suppose w.l.o.g. that $h=h(\tau)\in(0,h^\star)$. We define the measure $\eta_\tau$ as the push-forward of $\nu_\tau$ through the mapping $\mathcal{A}_{h(\tau)}$: 
\begin{equation*}
\eta_\tau \ := \ \mathcal{A}_{h(\tau)_{\#}}\nu_\tau \ = \  \mathcal{A}_{h(\tau)_{\#}}(\upsilon_\tau\mathcal{L}^d). 
\end{equation*}
The support of $\eta_\tau$ is a subset of $\Omega$, $\eta_\tau\ll\mathcal{L}^d$ and its Lebesgue density $\varrho_\tau$ satisfies
\begin{equation}\label{eq: density of push-forward}
|\det(\mathrm{D}\mathcal{A}_{h(\tau)}(x))|\varrho_\tau(\mathcal{A}_{h(\tau)}(x)) \ = \ \upsilon_\tau(x), \quad \quad x\in\mathbb{R}^d, 
\end{equation}
by the change of variables formula, where $\mathrm{D}\mathcal{A}_{h(\tau)}$ denotes the differential of $\mathcal{A}_{h(\tau)}$. It is not difficult to see that the mapping 
\begin{equation*}
(0,h^\star)\ni h\mapsto a_x(h):=\det(\mathrm{D}\mathcal{A}_{h}(x))
\end{equation*}
is smooth for every $x\in\mathbb{R}^d$ and the derivative $a_x'(h)$ is uniformly bounded in both $h\in(0, h^\star)$ and $x\in\mathbb{R}^d$: 
\begin{equation*}
\sup_{h\in(0, h^\star), x\in\mathbb{R}^d} { |a_x'(h)| } \ =: \ C_a < +\infty. 
\end{equation*}
For the rest of the proof, let the time step size $\tau>0$ be so small that $h(\tau)\leq\min\big\{h^\star, \frac{1}{2C_a}\big\}$ and hence
\begin{equation}\label{eq: estimate a_x}
\frac{1}{2} \ \leq \ a_x(s) \ \leq \frac{3}{2} \quad\text{for all } s\in(0,h(\tau)), \ x\in\mathbb{R}^d
\end{equation}
by Taylor's theorem (as $a_x(0)=1$). The change of variables formula and \eqref{eq: density of push-forward} yield 
\begin{equation*}
\phi_{n(\tau)}(\nu_\tau) - \phi(\eta_\tau) \ = \ \int_{\mathbb{R}^d}{\Big[F(\upsilon_\tau(x)) - F\Big(\frac{\upsilon_\tau(x)}{\det(\mathrm{D}\mathcal{A}_{h(\tau)}(x))}\Big)\det(\mathrm{D}\mathcal{A}_{h(\tau)}(x))\Big]\mathrm{d}x};
\end{equation*}
we apply Taylor's theorem to the integrand and can estimate
\begin{equation}\label{eq: integrand push-forward}
\Big|F(\upsilon_\tau(x))- F\Big(\frac{\upsilon_\tau(x)}{a_x(h(\tau))}\Big)a_x(h(\tau))\Big| \ \leq \ h(\tau)\cdot C_a\cdot L_F(2\upsilon_\tau(x))
\end{equation}
because $L_F$ (see \eqref{eq: L_F}) is increasing and $\frac{\upsilon_\tau(x)}{a_x(s)}\leq2\upsilon_\tau(x)$ for all $s\in(0, h(\tau))$. We deduce 
\begin{equation*}
\lim_{\tau\to0}\frac{\phi_{n(\tau)}(\nu_\tau) - \phi(\eta_\tau)}{\tau} \ = \ 0
\end{equation*}
from \eqref{eq: integrand push-forward}, \eqref{eq: s h(tau)} and the facts that
\begin{equation*}
0 \ \leq \ L_F(2\upsilon_\tau) \ \leq \ C_F(1+2C_F(1+2F(\upsilon_\tau))) - 2\min_{r\ge0}{F(r)}
\end{equation*}
by \ref{ass: A1}, $L_F(0)=0$ and $\sup_\tau{\phi_{n(\tau)}(\nu_\tau)} < +\infty$. Moreover, the definition of $\eta_\tau$ and \eqref{eq: s h(tau)} allow the simple estimate
\begin{equation*}
\mathcal{W}_2(\eta_\tau, \nu_\tau) \ \leq \ \Big(\int_{\mathbb{R}^d}{|x-\mathcal{A}_{h(\tau)}(x)|^2\mathrm{d}\nu_\tau(x)}\Big)^{1/2} \ 
= \ o(\tau^2). 
\end{equation*}
So we have proved \eqref{eq: eta_tau}; all that remains to be proved is \eqref{eq: main inequality}. The density $\varrho_\tau$ of $\eta_\tau$ satisfies 
\begin{equation}\label{eq: estimate varrho_tau}
||\varrho_\tau||_\infty \ \leq \ 2||\upsilon_\tau||_\infty \ \leq \ \frac{2||\K||_\infty}{h(\tau)^d}
\end{equation}
by \eqref{eq: density of push-forward}, \eqref{eq: estimate a_x} and the fact that $\nu_\tau=\upsilon_\tau\mathcal{L}^d\in\{\phi_{n(\tau)}<+\infty\}$. By Prop. 2.3 in \cite{agueh2002existence}, there is a unique measure $\varpi_\tau=\rho_\tau\mathcal{L}^d$ satisfying
\begin{equation*}
\mathcal{Y}_\tau\phi(\eta_\tau) \ = \ \phi(\varpi_\tau) + \frac{1}{2\tau}\mathcal{W}_2(\varpi_\tau, \eta_\tau)^2
\end{equation*}
and 
\begin{equation}\label{eq: estimate rho_tau}
||\rho_\tau||_\infty \ \leq \ ||\varrho_\tau||_\infty \ \leq \ \frac{2||\K||_\infty}{h(\tau)^d}
\end{equation}
 (the author of \cite{agueh2002existence} assumes that the PDE domain is a convex subset of $\mathbb{R}^d$ but his proof of Prop. 2.3 shows that the corresponding statement holds true for non-convex $\Omega$, too). We note that the sequence 
\begin{equation}\label{eq: rho_tau Wasserstein estimate}
\Big(\frac{\mathcal{W}_2(\varpi_\tau, \eta_\tau)}{\sqrt{\tau}}\Big)_{\tau>0} \quad\quad\text{ is bounded }
\end{equation}
because $\phi$ is bounded from below and $\sup_\tau\mathcal{Y}_\tau\phi(\eta_\tau)\leq\sup_\tau\phi(\eta_\tau) < +\infty$. 
Using \eqref{eq: estimate rho_tau} and applying Thm. 27, Lem. 11, Lem. 14 and Cor. 15 in \cite{kim2019uniform} (cf. Proposition \ref{prop: convergence rates}) and the corresponding proofs therein, we construct a probability measure $\hat{\varpi}_\tau:=\hat{\rho}_\tau\mathcal{L}^d\in\{\phi_{n(\tau)}<+\infty\}$ with Lebesgue density
\begin{equation*}
\hat{\rho}_\tau(\cdot) \ = \ \frac{1}{n(\tau)}\sum_{i=1}^{n(\tau)}{\K_{h(\tau)}(\cdot-y_i)}, \quad y_1,...,y_{n(\tau)}\in\Omega
\end{equation*}
that satisfies
\begin{equation}\label{eq: hat rho_tau estimate}
||\hat{\rho}_\tau-\tilde{\rho}_\tau||_\infty \ \leq \ C_{\K, \Omega} \cdot  \Bigg(\sqrt{\frac{\log(1/h(\tau))}{n(\tau)h(\tau)^{2d}}} + \frac{\log(1/h(\tau))}{n(\tau)h(\tau)^{d}}\Bigg),
\end{equation}
in which $C_{\K, \Omega}$ is a constant depending on $\K$ and $\Omega$ only and 
\begin{equation*}
\tilde{\rho}_\tau(\cdot) \ := \ \int_{\mathbb{R}^d}{\K_{h(\tau)}(\cdot-y)\rho_{\tau}(y)\mathrm{d}y}. 
\end{equation*}
We obtain 
\begin{equation}\label{eq: hat rho_tau Wasserstein estimate}
\lim_{\tau\to0}\frac{\mathcal{W}_2(\hat{\varpi}_\tau, \varpi_\tau)}{\sqrt{\tau^3}} \ = \ 0
\end{equation}
in direct consequence of the triangle inequality, \eqref{eq: Wasserstein convolution estimate}, \eqref{eq: s h(tau)}, \eqref{eq: Wasserstein BV estimate}, \eqref{eq: hat rho_tau estimate} and \eqref{eq: s par 1}. Furthermore, Jensen's inequality provides us with the estimate
\begin{equation*}
\int_{\mathbb{R}^d}{F(\tilde{\rho}_\tau(x))\mathrm{d}x} \ \leq \ \int_{\mathbb{R}^d}{F(\rho_\tau(x))\mathrm{d}x}
\end{equation*}
and we have 
\begin{equation*}
\lim_{\tau\to0} \ \frac{1}{\tau}\cdot\Big(\int_{\mathbb{R}^d}{F(\tilde{\rho}_\tau(x))\mathrm{d}x} \ - \ \int_{\mathbb{R}^d}{F(\hat{\rho}_\tau(x))\mathrm{d}x}\Big) \ = \ 0
\end{equation*}
by \eqref{eq: hat rho_tau estimate}, \eqref{eq: h(n)}, \eqref{eq: s par 2} and the facts that $\tilde{\rho}_\tau, \hat{\rho}_\tau \equiv 0$ on $\mathbb{R}^d\setminus\Omega_{h(\tau)}$, both $||\hat{\rho}_\tau||_\infty$ and $||\tilde{\rho}_\tau||_\infty$ are bounded from above by $\frac{||\K||_\infty}{h(\tau)^d}$ and $f_M(Cr)\leq (C\vee 1)\cdot f_M(r)$ for every $C,M, r > 0$; it follows that 
\begin{equation}\label{eq: hat rho_tau energy estimate}
\liminf_{\tau\to0}\frac{\phi(\varpi_\tau)-\phi_{n(\tau)}(\hat{\varpi}_\tau)}{\tau} \ \geq \ 0. 
\end{equation}
We can conclude \eqref{eq: main inequality} from \eqref{eq: hat rho_tau energy estimate}, \eqref{eq: hat rho_tau Wasserstein estimate}, \eqref{eq: eta_tau} and \eqref{eq: rho_tau Wasserstein estimate} using the fact that
\begin{equation*}
\mathcal{Y}_\tau\phi(\eta_\tau) - \mathcal{Y}_\tau\phi_{n(\tau)}(\nu_\tau) \ \geq \ \phi(\varpi_\tau)-\phi_{n(\tau)}(\hat{\varpi}_{\tau}) + \frac{1}{2\tau}(\mathcal{W}_2(\varpi_\tau, \eta_\tau)^2 - \mathcal{W}_2(\hat{\varpi}_\tau, \nu_\tau)^2), 
\end{equation*}
the triangle inequality
\begin{equation*}
\mathcal{W}_2(\hat{\varpi}_\tau, \nu_\tau) \ \leq \ \mathcal{W}_2(\hat{\varpi}_\tau, \varpi_\tau) + \mathcal{W}_2(\varpi_\tau, \eta_\tau) + \mathcal{W}_2(\eta_\tau, \nu_\tau)
\end{equation*}
and the estimate 
\begin{equation*}
\mathcal{W}_2(\varpi_\tau, \eta_\tau)^2 - \mathcal{W}_2(\hat{\varpi}_\tau, \nu_\tau)^2 \ \geq \ - 2\cdot (\mathcal{W}_2(\eta_\tau, \nu_\tau) + \mathcal{W}_2(\varpi_\tau, \hat{\varpi}_\tau))\cdot\mathcal{W}_2(\hat{\varpi}_\tau, \nu_\tau). 
\end{equation*}
The proof of Theorem \ref{thm: selection of parameters} is complete. 
\end{proof}

Assuming the initial datum is uniformly bounded, i.e. there exists $M>0$ such that
\begin{equation}\label{eq: uniformly bounded initial datum}
||u^0||_\infty \ \leq \ M, \quad\quad \mu^0=u^0\mathcal{L}^d\in\{\phi<+\infty\}, 
\end{equation}
it may be reasonable to modify the KDE-MM-Scheme and the selection of the parameters correspondingly: 

\begin{remark}[Uniformly bounded solution]\label{rem: uniformly bounded solution}
Let $\phi$ be the energy functional from Theorem \ref{thm: selection of parameters}, in which $\Omega$ is an open and bounded subset of $\mathbb{R}^d$ with $\mathrm{C}^2$-boundary $\partial\Omega$ and $F: [0, +\infty)\to\mathbb{R}$ satisfies \ref{ass: A1}. 

We define $\phi_n: \mathcal{P}_2(\mathbb{R}^d)\to(-\infty, +\infty]$, 
\begin{equation*}
\phi_n(\mu) \ := \begin{cases} \int_{\mathbb{R}^d}{F(u(x))\mathrm{d}x} &\text{ if } \exists y_i\in\Omega: \ \mu=u\mathcal{L}^d, \ u(\cdot) =  \frac{1}{n}\sum_{i=1}^{n}{\K_{h(n)}(\cdot-y_i)}, \\ +\infty &\text{ else.}\end{cases}
\end{equation*}
Assuming $h(n)\downarrow0$ as $n\uparrow+\infty$ and 
\begin{equation*}
\lim_{n\to+\infty}\frac{\log(h(n))+\log(\alpha(n))}{nh(n)^d} \ = \ 0, \quad \quad \sum_{n\in\mathbb{N}}{\alpha(n)} < +\infty, 
\end{equation*}
the sequence $\phi_n, \ n\in\mathbb{N},$ forms a partial $\Gamma$-KDE-Approximation of $\phi$ according to Remark \ref{rem: Gamma KDE Approximation variation} in which Kernel Density Estimation almost surely yields a recovery sequence \eqref{eq: KDE recovery sequence} for all $\mu\in\{\phi<+\infty\}$ with uniformly bounded Lebesgue density, see Proposition \ref{prop: Gamma KDE} and Thm. 27, Lem. 11, Lem. 14 and Cor. 15 in \cite{kim2019uniform} and the corresponding proofs therein for a refinement of \eqref{eq: convergence rates 1}, \eqref{eq: convergence rates 2}. 
If $M>0$ is a given uniform upper bound \eqref{eq: uniformly bounded initial datum} for initial data, we fix a constant $\bar{M}>M$ and set a correlation $\tau\mapsto n(\tau)$ between time step sizes $\tau>0$ and parameters $n(\tau)\in\mathbb{N}, \ h(\tau):=h(n(\tau))>0$ in such a way that $n(\tau)\uparrow+\infty$ as $\tau\downarrow0$, 
\begin{equation}\label{eq: s h(tau) bounded case}
\lim_{\tau\to0} \ \frac{h(\tau)}{\tau^2} \ = \ 0, 
\end{equation}
\begin{equation}\label{eq: s par 1 bounded case}
\lim_{\tau\to0} \ \frac{\log(1/h(\tau))}{\tau^6\cdot n(\tau)h(\tau)^{d}} \ = \ 0,
\end{equation}
and 
\begin{equation}\label{eq: s par 2 bounded case}
\lim_{\tau\to0} \ \Bigg[\frac{1}{\tau}\cdot f_{\bar{M}}\Bigg(\sqrt{\frac{\log(1/h(\tau))}{n(\tau)h(\tau)^{d}}}\Bigg)\Bigg] \ = \ 0
\end{equation}
($f_{\bar{M}}$ being the concave modulus of continuity \eqref{eq: concave modulus of continuity}). It is possible to refine estimates \eqref{eq: estimate varrho_tau}, \eqref{eq: estimate rho_tau}, \eqref{eq: hat rho_tau estimate} from the proof of Theorem \ref{thm: selection of parameters} and thus determine $\epsilon_\tau:=C_1\cdot h(\tau) + C_2\cdot\sqrt{\frac{\log(1/h(\tau))}{n(\tau)h(\tau)^{d}}}$ ($C_1, C_2 > 0$ can be calculated) so that 
\begin{equation}\label{eq: nu_tau main condition bounded case}
\liminf_{\tau\to0}\frac{\phi_{n(\tau)}(\nu_\tau) - \tilde{\mathcal{Y}}_{\tau, \epsilon_\tau}\phi_{n(\tau)}(\nu_\tau)}{\tau} \ \geq \ \frac{1}{2}|\partial^-\phi|^2(\nu)
\end{equation}
whenever 
\begin{equation*}
\nu_\tau=\upsilon_\tau\mathcal{L}^d, \quad ||\upsilon_\tau||_\infty+\epsilon_\tau \leq \bar{M}, \quad \lim_{\tau\to0}\mathcal{W}_2(\nu_\tau, \nu)  =  0, \quad \sup_\tau\phi_{n(\tau)}(\nu_\tau)  < +\infty,  
\end{equation*}
and 
\begin{equation*}
\tilde{\mathcal{Y}}_{\tau, \epsilon_\tau}\phi_{n(\tau)}(\nu_\tau) := \inf\Big\{\phi_{n(\tau)}(\sigma) + \frac{1}{2\tau}\mathcal{W}_2(\sigma, \nu_\tau)^2: \ \sigma=w\mathcal{L}^d, \ ||w||_\infty\leq||\upsilon_\tau||_\infty + \epsilon_\tau \Big\}. 
\end{equation*}

We modify the KDE-MM-Scheme associated with $(\phi_n)_n$ and $\tau\mapsto n(\tau)$ by replacing $\Psi(\tau, Y_\tau^{m-1},\cdot)$ from \eqref{eq: KDE-MM-Scheme Psi} with
\begin{equation*}
\Psi_m(\tau, Y_\tau^{m-1}, Z) \ := \ \begin{cases} \Psi(\tau, Y_\tau^{m-1}, Z) &\text{ if } \big\|\frac{1}{n(\tau)}\sum_{i=1}^{n(\tau)}{\K_{h(\tau)}(\cdot-z_i)}\big\|_\infty \leq M_{\tau, m}, \\ +\infty &\text{ else,}\end{cases}
\end{equation*}
in each minimum problem \eqref{eq: KDE-MM-Scheme}, where $M_{\tau, m}:=\frac{M+\bar{M}}{2} + m\cdot\epsilon_\tau$. This modified KDE-MM-Scheme is performed with time step sizes $(\tau_k)_{k\in\mathbb{N}}, \ \tau_k\downarrow0,$ and error terms $\gamma_{\tau_k}^{(m)}$ satisfying \eqref{eq: error term non-uniform}; assuming the initial probability density $u^0$ is uniformly bounded from above by $M$ \eqref{eq: uniformly bounded initial datum},  $X_1, ..., X_{n(\tau_k)}$ is an i.i.d. sample from $\mu^0=u^0\mathcal{L}^d$  and the initial data $Y_{\tau_k}^0$ for the modified KDE-MM-Scheme are defined as $\big(X_1, ..., X_{n(\tau_k)}\big)$, then the following holds good for the corresponding discrete solutions $\bar{\mu}_{\tau_k}$ \eqref{eq: discrete solution}: 

There exists a locally absolutely continuous curve $\mu: [0, +\infty)\to\mathcal{P}_2(\mathbb{R}^d)$, $\mu(t)=u(t, \cdot)\mathcal{L}^d=u_t\mathcal{L}^d,$ with tangent vector field $v$ such that
\begin{equation*}
\lim_{k\to+\infty}\mathcal{W}_2(\bar{\mu}_{\tau_{k}}(t), \mu(t)) \ = \ 0 \quad \text{ and } \quad \lim_{k\to+\infty}\phi_{n(\tau_{k})}(\bar{\mu}_{\tau_{k}}(t)) \ = \ \phi(\mu(t)) 
\end{equation*}
for all $t\geq0$ with probability $1$, 
\begin{equation*}
\max_{t\geq0}||u_t||_\infty \ = \ ||u^0||_\infty, \quad\quad L_F(u)\in\mathrm{L}^2_{\mathrm{loc}}([0, +\infty); \mathrm{W}^{1,2}(\mathbb{R}^d))
\end{equation*}
($L_F$ defined in \eqref{eq: L_F}), and the triple $(\mu, u, v)$ of curve $\mu$ in $\mathcal{P}_2(\mathbb{R}^d)$, corresponding time-dependent density function $u$ and tangent vector field $v$ is the unique solution to the differential equation
\begin{equation}\label{eq: differential equation unique}
v_t \ = \ -\frac{\nabla L_F(u_t)}{u_t}  \quad\quad \mu(t)\text{-a.e.}\quad\text{ for } \mathcal{L}^1\text{-a.e. } t>0, \quad \mu(t)\ll\mathcal{L}^d\llcorner\Omega, 
\end{equation}
with $u\in\mathrm{L}^\infty([0, +\infty)\times\Omega)$ and initial datum $\mu(0)=\mu^0$; \eqref{eq: differential equation unique} is a weak reformulation of the second order diffusion equation
\begin{equation*}
\partial_t u - \nabla\cdot(u \nabla F'(u))  \ = \ 0 \quad\text{ in } (0, +\infty)\times\Omega
\end{equation*}
with no-flux boundary condition
\begin{equation*}
u\nabla F'(u) \cdot {\sf n} \ = \ 0 \quad \text{ on } (0, +\infty)\times\partial\Omega. 
\end{equation*}

Moreover, $\mu$ solves the energy dissipation equality \eqref{eq: ede} (with $q=p=2$) for all $0\leq s\leq t < +\infty$ and
\begin{equation*}
|\partial^-\phi|(\mu(t)) \ = \ |\mu'|(t) \ = \ \Big\|\frac{\nabla L_F(u_t)}{u_t}\Big\|_{\mathrm{L}^2(u_t\mathcal{L}^d; \mathbb{R}^d)} \quad\quad \mathcal{L}^1\text{-a.e..}
\end{equation*}

We can prove this statement by using \eqref{eq: nu_tau main condition bounded case}, following and adapting the proofs of Theorems \ref{thm: MM approach I}, \ref{thm: KDE-MM-Scheme consistency} and \ref{thm: selection of parameters}, applying Proposition \ref{prop: chain rule} and finally using the uniqueness of bounded weak solutions (see e.g. Prop. 4.8 in \cite{ambrosio2006stability} for the uniqueness statement). 
\end{remark}

Lastly, we sketch the selection of parameters in accordance with Assumption \ref{ass: 2} and condition \eqref{eq: KDE-MM-Scheme parameters} for the example from Remark \ref{rem: example weak Gamma KDE}. 

\begin{remark}[Selection of parameters: unbounded domain]\label{rem: selection unbounded domain}
Let $\mathsf{F}$ be the energy functional and $\mathsf{F}_n, \ n\in\mathbb{N},$ the associated partial weak $\Gamma$-KDE-Approximation from Remark \ref{rem: example weak Gamma KDE}. We define $\bar{\mathsf{F}}_n: \mathcal{P}_2(\mathbb{R}^d)\to(-\infty, +\infty]$ as 
\begin{equation*}
\bar{\mathsf{F}}_n(\mu) := \begin{cases} \mathsf{F}(\mu) &\text{ if } \mu=u\mathcal{L}^d\ll\mathcal{L}^d\llcorner\{|x|\leq R(n)+h(n)\}, \\ 
+\infty &\text{ else.}\end{cases}
\end{equation*}
The functional $\bar{\mathsf{F}}_n$ is displacement convex in $(\mathcal{P}_{2}(\mathbb{R}^d), \mathcal{W}_2)$ because both the set $\{\mu=u\mathcal{L}^d\in\mathcal{P}_2(\mathbb{R}^d): \ \mu\ll\mathcal{L}^d\llcorner\{|x|\leq R(n)+h(n)\}\}$ 
and the functional $\mathsf{F}$ are displacement convex in $(\mathcal{P}_2(\mathbb{R}^d), \mathcal{W}_2)$ (cf. 
Thm. 2.2 in \cite{mccann1997convexity}). Moreover, the functionals $\bar{\mathsf{F}}_n$ $\Gamma$-converge to $\mathsf{F}$ w.r.t. the topology induced by weak convergence as $n\uparrow+\infty$: the weak $\Gamma$-liminf inequality \eqref{eq: weak Gamma-liminf} is obviously satisfied and for $\nu\in\{\mathsf{F}<+\infty\}$, the sequence of probability measures 
\begin{equation*}
\bar{\nu}_n \ := \ \frac{\nu\llcorner\mathcal{B}(0; R(n))}{\nu(\mathcal{B}(0; R(n)))} 
\end{equation*}
($\mathcal{B}(0; R(n)):=\{|z| < R(n)\}$)
forms a recovery sequence \eqref{eq: recovery sequence}. The analysis of the strong subdifferential from Example \ref{ex: chain rule}, Lem. 10.1.5, Thm. 10.4.6 and Prop. 10.4.14 from \cite{AGS08} show that
\begin{equation*}
\nu_n\rightharpoonup\nu, \ \sup_n\{\mathcal{W}_2(\nu_n, \nu), \bar{\mathsf{F}}_n(\nu_n)\} < +\infty \ \Rightarrow \ \liminf_{n\to+\infty} |\partial\bar{\mathsf{F}}_n|(\nu_n) \ \geq \ |\partial\mathsf{F}|(\nu). 
\end{equation*}
It follows from this  weak $\Gamma$-liminf inequality for the local slopes, Sect. 5 and Prop. 5.2 in \cite{fleissner2016gamma} that the sequence $\bar{\mathsf{F}}_n, n\in\mathbb{N},$ satisfies our main condition \eqref{eq: main condition} for every choice $n=n(\tau)\uparrow+\infty \ (\tau\downarrow0)$, i.e.  whenever $\nu_\tau, \nu\in\mathcal{P}_2(\mathbb{R}^d)$, 
\begin{equation*}
\nu_\tau\rightharpoonup\nu, \quad \sup_\tau\big\{\mathcal{W}_2(\nu_\tau, \nu), \bar{\mathsf{F}}_{n(\tau)}(\nu_\tau)\big\}  < +\infty, 
\end{equation*}
then
\begin{equation}\label{eq: main condition displacement convex}
\liminf_{\tau\to0}\frac{\bar{\mathsf{F}}_{n(\tau)}(\nu_\tau) - \mathcal{Y}_\tau\bar{\mathsf{F}}_{n(\tau)}(\nu_\tau)}{\tau} \ \geq \ \frac{1}{2}|\partial^-\mathsf{F}|^2(\nu), 
\end{equation}
cf. Remark \ref{rem: main ass}\ref{rem: main ass ii}-\ref{rem: main ass iii}. A similar argumentation as in the proof of Theorem \ref{thm: selection of parameters} and an application of Thm. 27, Lem. 11, Lem. 14, Cor. 15 and of the corresponding proofs from \cite{kim2019uniform} (cf. Proposition \ref{prop: convergence rates}) enable us to determine suitable correlations $n\mapsto h(n)$, $n\mapsto R(n)$ and $\tau\mapsto n(\tau)$ so that $h(\tau)=o(\tau^2)$ and 
\begin{equation}\label{eq: main inequality displacement convex}
\liminf_{\tau\to0}\frac{\mathcal{Y}_\tau\bar{\mathsf{F}}_{n(\tau)}(\nu_\tau) - \mathcal{Y}_\tau\mathsf{F}_{n(\tau)}(\nu_\tau)}{\tau} \ \geq \ 0
\end{equation}
whenever
\begin{equation*}
\nu_\tau\rightharpoonup\nu, \quad \sup_\tau\big\{\mathsf{F}_{n(\tau)}(\nu_\tau), \mathcal{W}_2(\nu_\tau, \nu)\big\}  < +\infty. 
\end{equation*}
Assumption \ref{ass: 2} for $\mathsf{F}, \ \mathsf{F}_n$ and the selected parameters, then, is a direct consequence of \eqref{eq: main condition displacement convex} and \eqref{eq: main inequality displacement convex}, and the associated KDE-MM-Scheme performed as per the instructions from Definition \ref{def: KDE-MM-Scheme} and Theorem \ref{thm: KDE-MM-Scheme consistency} almost surely yields the unique gradient flow solution \eqref{eq: differential inclusion} corresponding to $\mathsf{F}$ and each initial datum $\mu^0\in\{\mathsf{F}<+\infty\}$ with compact support, cf. Remarks \ref{rem: example weak Gamma KDE}, \ref{rem: Gamma KDE Approximation variation}, \ref{rem: relaxation phin} and Example \ref{ex: convex case}. We refer the reader to Thm. 11.1.4 in \cite{AGS08} for the uniqueness statement and `contraction estimates', showing that \textit{every} gradient flow solution can be approximated by solutions corresponding to initial data with compact support (which, in turn, can be approximated by our KDE-MM-Scheme). Finally, we remark that the differential equation \eqref{eq: differential inclusion} corresponding to $\mathsf{F}$ represents a weak reformulation of the porous medium equation (for $F(s):= \frac{1}{m-1}s^m, \ m>1$) or the heat equation (for $F(s):=s\log s$). 
\end{remark}

\bibliographystyle{siam}
\bibliography{kdebib}
\end{document}